\documentclass[10pt,draft]{amsart}
\usepackage{amsthm,amsfonts,amsmath,amssymb,latexsym}
\newtheorem{PARA}{}
\newtheorem{theorem}[PARA]{Theorem}
\newtheorem{corollary}[PARA]{Corollary}
\newtheorem{lemma}[PARA]{Lemma}
\newtheorem{proposition}[PARA]{Proposition}
\newtheorem{definition}[PARA]{Definition}

\theoremstyle{definition}
\newtheorem{remark}[PARA]{Remark}
\newtheorem{example}[PARA]{Example}
\numberwithin{equation}{section}
\newcommand{\para}{\begin{PARA}\rm}
\newcommand{\arap}{\end{PARA}\rm}
\newcommand{\dfn}{\begin{definition}\rm}
\newcommand{\nfd}{\end{definition}\rm}
\newcommand{\rmk}{\begin{remark}\rm}
\newcommand{\kmr}{\end{remark}\rm}
\newcommand{\xmpl}{\begin{example}\rm}
\newcommand{\lpmx}{\end{example}\rm}
\newcommand{\cA}{\mathcal{A}}

\newcommand{\cC}{\mathcal{C}}

\newcommand{\cH}{\mathcal{H}}

\newcommand{\cJ}{\mathcal{J}}
\newcommand{\cK}{\mathcal{K}}
\newcommand{\cL}{\mathcal{L}}
\newcommand{\cM}{\mathcal{M}}

\newcommand{\cO}{\mathcal{O}}
\newcommand{\cP}{\mathcal{P}}

\newcommand{\cS}{\mathcal{S}}

\newcommand{\cV}{\mathcal{V}}
\newcommand{\cW}{\mathcal{W}}

\newcommand{\oalpha}{{\overline{\alpha}}}

\newcommand{\og}{{\overline{\gamma}}}
\newcommand{\ug}{{\underline{\gamma}}}
\newcommand{\oev}{\overline{\mathrm{ev}}}
\newcommand{\uev}{\underline{\mathrm{ev}}}

\newcommand{\olambda}{{\overline{\lambda}}}
\newcommand{\ulambda}{{\underline{\lambda}}}
\newcommand{\op}{{\overline{p}}}
\newcommand{\up}{{\underline{p}}}

\newcommand{\ox}{\overline{x}}

\newcommand{\one}
{{{\mathchoice \mathrm{ 1\mskip-4mu l} \mathrm{ 1\mskip-4mu l}
\mathrm{ 1\mskip-4.5mu l} \mathrm{ 1\mskip-5mu l}}}}

\newcommand{\C}{{\mathbb{C}}}

\newcommand{\N}{{\mathbb{N}}}
\newcommand{\Q}{{\mathbb{Q}}}
\newcommand{\R}{{\mathbb{R}}}

\renewcommand{\u}{{\mathbf{u}}}

\newcommand{\Z}{{\mathbb{Z}}}
\newcommand{\ind}{\mathrm{ind}}
\newcommand{\Jreg}{\cJ_{\mathrm{reg}}}   

\newcommand{\Sp}{{\mathrm{Sp}}}
\newcommand{\OO}{{\mathrm{O}}}
\newcommand{\Mat}{{\mathrm{Mat}}}
\newcommand{\reg}{{\mathrm{reg}}}
\newcommand{\eps}{{\varepsilon}}
\newcommand{\om}{{\omega}}

\newcommand{\tH}{{\widetilde{H}}}

\newcommand{\tgamma}{{\widetilde{\gamma}}}
\newcommand{\tmu}{{\widetilde{\mu}}}
\newcommand{\tM}{{\widetilde{M}}}

\def\NABLA#1{{\mathop{\nabla\kern-.5ex\lower1ex\hbox{$#1$}}}}
\def\Nabla#1{\nabla\kern-.5ex{}_{#1}}
\def\Tabla#1{\Tilde\nabla\kern-.5ex{}_{#1}}
\renewcommand{\Tilde}{\widetilde}

\newcommand{\p}{{\partial}}
\newcommand{\dbar}{{\bar\partial}}

\begin{document}

\title[The index of parametrized action functionals]{The index of Floer moduli problems for parametrized action functionals}
\author{Fr\'ed\'eric Bourgeois}
\address{Universit\'e Libre de Bruxelles, B-1050 Bruxelles, Belgium}
\author{Alexandru Oancea}
\address{Institut de Recherche Math\'ematique Avanc\'ee, UMR 7501, CNRS \& Universit\'e de
Strasbourg, France \\  and Institute for Advanced Study, Princeton, NJ 08540, USA}
\date{July 17, 2012}


\begin{abstract}
We define an index for the critical points of parametrized Hamiltonian action functionals. The expected dimension of moduli spaces of parame-\break trized Floer trajectories equals the difference of indices of the asymptotes.
\end{abstract}

\maketitle



\section{Main definition and main theorem}

\subsection{The parametrized action functional} Let $\Lambda$ be a manifold of dimension $m$, $(W,\omega)$ a symplectic manifold of dimension $2n$, and $H:S^1\times W\times \Lambda\to \R$, $H(\theta,x,\lambda)=H_\lambda(\theta,x)$ a smooth family of Hamiltonians defined on $W$. Let $\cL W$ denote the space of loops in $W$ and assume for symplicity that $\omega=d\alpha$ is exact. We are interested in the \emph{parametrized Hamiltonian action functional}
$$
A_H:\cL W\times \Lambda \to \R, \quad (\gamma,\lambda)\longmapsto -\int_\gamma \alpha - \int_{S^1}H_\lambda(\theta,\gamma(\theta))\, d\theta.
$$
Such functionals appear in a variety of settings (Appendix~\ref{app:examples}), and we analyzed in~\cite{BOtransv} their Fredholm theory and their transversality theory.

\subsection{Critical points}
A pair $(\gamma,\lambda)\in\cL W\times \Lambda$ is a critical point of $A_H$ iff it solves the system
\begin{equation} \label{eq:periodicpar}
\dot\gamma(\theta) = X_{H_\lambda}(\theta,\gamma(\theta)), \ \theta\in S^1 \ \mbox{and} \ \int_{S^1} \frac{\partial H}{\partial \lambda}(\theta,\gamma(\theta),\lambda)\, d\theta = 0.
\end{equation}
Our convention for the definition of $X_{H_\lambda}$  is $\omega(X_{H_\lambda},\cdot)=d H_\lambda$. 
We say that the critical point $(\gamma,\lambda)$ is \emph{nondegenerate} if the Hessian $d^2A_H(\gamma,\lambda)$ is injective. If the critical points of $A_H$ are all nondegenerate (which is a generic assumption), they can be used to define a Floer chain complex whose differential is expressed as a count of rigid $L^2$-gradient trajectories~\cite{BOGysin}. The purpose of the present paper is to associate an index to each critical point of $A_H$, in such a way that the dimension of the moduli space of connecting Floer trajectories is expressed as the difference of the indices at the endpoints.

Equation~\eqref{eq:periodicpar} can be interpreted as follows. Every
loop $\gamma:S^1\to W$ determines a function
\begin{equation} \label{eq:Fgamma}
F_\gamma:\Lambda \to \R, \qquad \lambda \mapsto \int_{S^1}
H(\theta,\gamma(\theta),\lambda) \, d\theta.
\end{equation}
A pair $(\gamma,\lambda)$ belongs to $\mathrm{Crit}(A_H)$ iff $\gamma$ is a $1$-periodic orbit of $X_{H_\lambda}$ and $\lambda$ is a critical point of $F_\gamma$.   However, the nondegeneracy of $(\gamma,\lambda)$
  does not imply that $\gamma$ is a nondegenerate orbit of
  $H_\lambda$, nor that $\lambda$ is a nondegenerate critical point of
  $F_\gamma$. This situation is already present in Morse theory, as
  the following example shows.

\begin{example} Consider the Morse function
  $f:\R\times \R\to\R$, $(x,\lambda)\mapsto x\lambda$. Then $(x_0,\lambda_0)=(0,0)$ is a
  nondegenerate critical point, but $f$ is constant along
  $\R\times \{0\}$ and $\{0\}\times \R$, hence $x_0=0$
  and $\lambda_0=0$ are degenerate critical points.
\end{example}

It is thus not \emph{a priori} clear how to define the index of a critical point $(\gamma,\lambda)$, unless the Hamiltonian $H$ is \emph{split}, i.e. of the form $H(\theta,x,\lambda)=K(\theta,x)+f(\lambda)$, in which case the system~\eqref{eq:periodicpar} is uncoupled. Our discovery is that one can define the index using a parametrized version of the Robbin-Salamon index which we now explain. 
Our method works in general and our approach is fundamentally different from other attempts dealing with particular cases~\cite{Viterbo99,Cieliebak-Frauenfelder}.

\subsection{The parametrized Robbin-Salamon index} \label{sec:paramRS}
Given a Hamiltonian $H:S^1\times W\times
\Lambda\to\R$, we extend it to $\tH:S^1\times  W\times
T^*\Lambda\to\R$ by the formula
$$
\tH(\theta,x,(\lambda,p)):=H(\theta,x,\lambda)=H_\lambda(\theta,x),
$$
so that
$$
X_\tH=X_{H_\lambda}-\frac {\p H} {\p \lambda}\frac \p
{\p p}.
$$
(We use the symplectic form $d\lambda \wedge dp$ on $T^*\Lambda$.) 
A $1$-periodic orbit $\tgamma$ of $X_\tH$ is of the form
$\tgamma=(\gamma(\cdot),\lambda,p(\cdot))$, with $\gamma$ a $1$-periodic orbit of
$X_{H_\lambda}$ and 
$p(\theta)=p(0)-\int_0^\theta\frac {\p H} {\p \lambda}
(\tau,\gamma(\tau),\lambda)\, d\tau$. The closing condition
$p(1)=p(0)$ is equivalent to
$\int_0^1\frac {\p H} {\p \lambda}
(\tau,\gamma(\tau),\lambda)\, d\tau=0$, while $p(0)\in T^*_\lambda
\Lambda$ can be chosen arbitrarily. Thus critical
points of $A_H$ are in one-to-one
bijective correspondence with families of $1$-periodic
orbits of $X_{\tH}$, of dimension
$\dim\,T^*_\lambda\Lambda=\dim\,\Lambda$.

We assume in this paper that
$$
\langle c_1(W),\pi_2(W)\rangle=0
$$
and we consider only critical points $(\gamma,\lambda)$ such that $\gamma$ is contractible in $W$. These restrictions are only meant to focus the discussion and are by no means essential. The associated periodic orbits $\tgamma$ are then contractible in $W\times T^*\Lambda$, and we have $\langle c_1(W\times T^*\Lambda),\pi_2(W\times T^*\Lambda)\rangle =0$. In this situation we can associate without ambiguity to the periodic orbit $\tgamma$ a half-integer called the \emph{Robbin-Salamon index}. This index is defined as the Maslov index~\cite{RS} of the path of symplectic matrices obtained by linearizing the Hamiltonian flow of $\widetilde H$ along $\tgamma$ and by trivializing $T(W\times T^*\Lambda)$ over a disc bounded by $\tgamma$.

\medskip

\noindent {\bf Main Definition.} \emph{The \emph{parametrized Robbin-Salamon
  index} $\mu(\gamma,\lambda)$ of a critical point of $A_H$ is
the Robbin-Salamon index of one of the corresponding $1$-periodic orbits
$(\gamma(\cdot),\lambda,p(\cdot))$ of $\tH$.}

\subsection{The parametrized Floer equation}

Let $J=(J_\lambda^\theta)$, $\lambda\in\Lambda$, $\theta\in S^1$ be a
family of compatible
almost complex structures on $W$. This induces a
  $\Lambda$-family of $L^2$-metrics on $C^\infty(S^1,W)$,
  defined by
$$
\langle \zeta,\eta\rangle_\lambda := \int_{S^1}
\om(\zeta(\theta),J_\lambda^\theta\eta(\theta)) d\theta, \quad \zeta,\eta\in
T_\gamma C^\infty(S^1,W)=\Gamma(\gamma^*TW).
$$
Such a metric can be coupled with any metric $g$ on $\Lambda$ and
gives rise to a metric on
$C^\infty(S^1,W)\times \Lambda$ acting
at a point $(\gamma,\lambda)$ by
$$
\langle(\zeta,\ell), (\eta,k)\rangle_{J,g}:= \langle \zeta,\eta\rangle_\lambda +
g(\ell,k), \qquad (\zeta,\ell),(\eta,k)\in \Gamma(\gamma^*TW)\oplus T_\lambda\Lambda.
$$

The \emph{parametrized Floer equation} is the negative gradient equation for
$A_H$ with respect to such a metric
$\langle\cdot,\cdot\rangle_{J,g}$. More precisely, given
$(\og,\olambda),(\ug,\ulambda)\in \mathrm{Crit}(A_H)$
we denote by
$$
\cM((\og,\olambda),(\ug,\ulambda);H,J,g)
$$
the \emph{space of parametrized Floer trajectories}, consisting of
pairs $(u,\lambda)$ with
$$
u:\R\times S^1 \to W, \qquad \lambda:\R\to \Lambda,
$$
satisfying
\begin{eqnarray}
\label{eq:Floer1par}
 \p_s u + J_{\lambda(s)}^\theta (\p_\theta u -
X_{H_{\lambda(s)}}^\theta (u)) & = & 0, \\
\label{eq:Floer2par}
 \dot \lambda (s) - \int_{S^1} \vec \nabla_\lambda
H(\theta,u(s,\theta),\lambda(s)) d\theta & = & 0,
\end{eqnarray}
and
\begin{equation} \label{eq:asymptoticpar}
 \lim_{s\to -\infty} (u(s,\cdot),\lambda(s)) = (\og,\olambda), \quad
 \lim_{s\to +\infty} (u(s,\cdot),\lambda(s)) = (\ug,\ulambda).
\end{equation}
Here and in the sequel we use the notation $\vec \nabla$ for a
gradient vector field, whereas $\nabla$ will denote a
covariant derivative.

\subsection{The index theorem for the linearized operator}

Let us fix $p> 1$. By linearizing
equations~(\ref{eq:Floer1par}-\ref{eq:Floer2par}) we obtain the
operator
$$
D_{(u,\lambda)} : W^{1,p}(u^*TW) \oplus
W^{1,p}(\lambda^*T\Lambda) \to
L^p(u^*TW) \oplus L^p(\lambda^*T\Lambda),
$$
\begin{equation}\label{eq:Dulambda}
D_{(u,\lambda)} (\zeta,\ell) :=
\left(\begin{array}{c}
D_u\zeta + (D_\lambda J\cdot \ell)(\p_\theta u - X_{H_\lambda}(u)) -
J_\lambda (D_\lambda X_{H_\lambda}\cdot \ell) \\
\nabla_s \ell - \nabla_\ell \int_{S^1} \vec \nabla_\lambda H
(\theta,u,\lambda)
- \int_{S^1} \nabla_\zeta \vec \nabla_\lambda H(\theta,u,\lambda)
\end{array}\right),
\end{equation}
where
$$
D_u : W^{1,p}(u^*TW) \to L^p(u^*TW)
$$
is the usual Floer-Gromov operator
$$
D_u\zeta := \nabla_s \zeta + J_\lambda \nabla_\theta \zeta -
J_\lambda \nabla_\zeta X_{H_\lambda} + \nabla_\zeta J_\lambda (\p_\theta u -
X_{H_\lambda}).
$$

Let us denote
\begin{eqnarray*}
\cW^{1,p} & := & W^{1,p}(\R\times S^1,u^*T W) \oplus
W^{1,p}(\R,\lambda^*T\Lambda), \\
\cL^p & := & L^p(\R\times S^1,u^*TW) \oplus
L^p(\R,\lambda^*T\Lambda).
\end{eqnarray*}
We proved in~\cite[Theorem~2.6]{BOtransv} that, given $(\og,\olambda),(\ug,\ulambda)\in\mathrm{Crit}(A_H)$ which are nondegenerate, and given $(u,\lambda)\in
\cM((\og,\olambda),(\ug,\ulambda); H,J,g)$, the operator
$$
D_{(u,\lambda)} : \cW^{1,p}\to \cL^p
$$
is Fredholm for $1<p<\infty$.
Moreover, for a generic choice of the triple $(H,J,g)$, the space of Floer trajectories
$\cM((\og,\olambda),(\ug,\ulambda); H,J,g)$ is a smooth manifold whose local dimension at $(u,\lambda)$ is equal to $\ind \, D_{(u,\lambda)}$~\cite[Theorem~4.1]{BOtransv}.

\medskip

\noindent {\bf Main Theorem.} \emph{Assume $(\og,\olambda),(\ug,\ulambda)\in\mathrm{Crit}(A_H)$ are
  nondegenerate and fix $1<p<\infty$. For any $(u,\lambda)\in
\cM((\og,\olambda),(\ug,\ulambda); H,J,g)$ the index of the
Fredholm operator $D_{(u,\lambda)}:\cW^{1,p}\to\cL^p$ is}
$$
\ind \, D_{(u,\lambda)} = -\mu(\og, \olambda) + \mu(\ug, \ulambda) .
$$


\section{Proof of the main theorem}

\subsection{A subgroup of $\Sp(2n+2m)$}

Let $n,m\ge 1$ be integers and define the subgroup $\cS_{n,m}\subset \Sp(2n+2m)$ to consist of
matrices of the form
\begin{equation} \label{eq:M}
M=M(\Psi,X,E)=\left(\begin{array}{ccc}
\Psi & \Psi X & 0 \\
0 & \one & 0 \\
X^TJ_0 & E+\frac 1 2 X^T J_0 X & \one
\end{array}\right),
\end{equation}
with $\Psi\in\Sp(2n)$,
$X\in\mathrm{Mat}_{2n,m}(\R)$, and $E\in \mathrm{Mat}_m(\R)$
symmetric. Here we have denoted $J_0:=\left(\begin{array}{cc} 0 &
    -\one \\ \one & 0 \end{array}\right)$ the standard complex
structure on $\R^{2n}$, and the elements $\Psi\in\Sp(2n)$ are
characterized by the condition $\Psi^T J_0\Psi=J_0$.
Similarly, we denote the standard complex structure on $\R^{2n}\times \R^{2m}$ by
$$
\widetilde J_0:=\left(
  \begin{array}{ccc}
J_0 & 0 & 0 \\
0 & 0 & -\one \\
0 & \one & 0
  \end{array}\right),
$$
and the elements $\widetilde \Psi\in\Sp(2n+2m)$ are characterized by
the condition $\widetilde \Psi^T \widetilde J_0\widetilde
\Psi=\widetilde J_0$.
We have that $\cS_{n,m}$
is a subgroup (but we shall not use this fact). The subgroup property follows from the relations 
\begin{eqnarray*}
\lefteqn{M(\Psi_1,X_1,E_1)\cdot M(\Psi_2,X_2,E_2)}\\
&=& M(\Psi_1\Psi_2,X_2+\Psi_2^{-1}X_1, E_1+E_2+\mathrm{Sym}(X_1^TJ_0\Psi_2X_2))
\end{eqnarray*}
and
\begin{equation} \label{eq:inverse}
M(\Psi,X,E)^{-1}=M(\Psi^{-1},-\Psi X,-E).
\end{equation}
Here we have used the notation
$$
\mathrm{Sym}(P):=(P+P^T)/2
$$
for the symmetric part of a square matrix $P$.

The form of the elements of $\cS_{n,m}$ may seem less artificial in view of the following Lemma. Elements of the form~\eqref{eq:MABC} arise naturally in the next section.

\begin{lemma} \label{lem:MABC}
Let
\begin{equation}\label{eq:MABC}
M=\left(\begin{array}{ccc}\Psi & A & 0 \\ 0 & \one & 0 \\ B & C & \one \end{array}\right)
\end{equation}
be a square $(2n+2m)$-matrix, such that $\Psi$ is a square $2n$-matrix and $\one$ is the identity $m$-matrix. Then $M$ is symplectic if and only if $\Psi$ is symplectic and there exists a matrix $X$ and a symmetric matrix $E$ such that $M=M(\Psi,X,E)$.
\end{lemma}

\begin{proof}
The proof is a straightforward computation using block matrices and the condition $M^T\widetilde J_0 M=\widetilde J_0$.
\end{proof}

We refer to Appendix~\ref{app:RS} for a summary of the properties of
the Robbin-Salamon index of paths with values in $\cS_{n,m}$.

\subsection{The linearized flow of $\tH$}

Recall from~\S\ref{sec:paramRS} the Hamiltonian
$$
\widetilde H:S^1\times W\times T^*\Lambda\to \R, \qquad \widetilde H(\theta,x,(\lambda,p)):=H(\theta,x,\lambda),
$$
whose flow is given by
\begin{equation} \label{eq:flowtH}
\varphi^\theta_\tH(x,\lambda,p)=\left(\varphi^\theta_{H_\lambda}(x),\lambda,p-\int_0^\theta\frac {\partial H} {\partial \lambda}(\tau,\varphi^\tau_{H_\lambda}(x),\lambda)\, d\tau\right).
\end{equation}
Let $(\gamma,\lambda)\in\mathrm{Crit}(A_H)$ be a critical point and 
$\tgamma=(\gamma(\cdot),\lambda,p(\cdot))$ be an associated $1$-periodic orbit of $X_{\tH}$. We fix a unitary trivialization of $\gamma^*TW$ coming from a spanning disc and we fix an isometry $T_\lambda\Lambda\equiv\R^m$, and these together determine a unitary trivialization of $\tgamma^*T(W\times T^*\Lambda)$. The linearized flow $d\varphi^\theta_\tH$ read in such a trivialization determines a path $M(\theta)$, $\theta\in[0,1]$ of symplectic matrices of the form~\eqref{eq:MABC}, and this path takes values in $\cS_{n,m}$ by Lemma~\ref{lem:MABC}. By definition, the index $\mu(\gamma,\lambda)$ is equal to the Robbin-Salamon index of the path $M$.

The matrices $\Psi(\theta)$, $X(\theta)$, and $E(\theta)$ that determine $M(\theta)=M(\Psi(\theta),X(\theta),E(\theta))$ are expressed as follows. We denote $\Psi$ and $A$ the
components of the linearization of the flow $\varphi^\theta_{H_\lambda}$ in the given trivializations of $\gamma^*TW$ and of $T_\lambda\Lambda$, and set $X:=\Psi^{-1}A$. Thus the linearized flow acts as
\begin{eqnarray}
T_{(\gamma(0),\lambda)} (W \times \Lambda) &\to&
T_{\gamma(\theta)} W, \nonumber \\
(\zeta_0,\ell) &\mapsto& \Psi(\theta) \zeta_0 + \Psi(\theta) X(\theta)
\ell . \label{eq:linflow}
\end{eqnarray}
The matrix $E(\theta)$ is the symmetric part of the
endomorphism 
\begin{eqnarray}
T_\lambda \Lambda &\to& T_\lambda \Lambda, \nonumber \\
\ell &\mapsto& -\frac{d}{d\lambda} \int_0^\theta \vec \nabla_\lambda H
(\tau, \Phi^\tau(\gamma(0),\lambda), \lambda) \, d\tau \cdot \ell . \label{eq:Easy}
\end{eqnarray}

\subsection{The spectral flow of the linearized operator $D_{(u,\lambda)}$}

Let us fix a connecting trajectory $(u,\lambda)\in\cM((\og,\olambda),(\ug,\ulambda);H,J,g)$  between two nondegenerate critical points of $A_H$. We recall here from~\cite[Lemma~2.3]{BOtransv} that the nondegeneracy of a critical point $(\gamma,\lambda)$ is equivalent to the bijectivity of the \emph{asymptotic operator}
$$
D_{(\gamma,\lambda)} : H^1(S^1,\gamma^*TW) \times T_\lambda
\Lambda \to L^2(S^1,\gamma^*T W) \times T_\lambda
\Lambda,
$$
\begin{equation}  \label{eq:Dasy}
D_{(\gamma,\lambda)}(\zeta,\ell) = \left(\begin{array}{c}
J_\lambda(\nabla_\theta \zeta - \nabla_\zeta X_{H_\lambda} -
(D_\lambda X_{H_\lambda})\cdot \ell) \\
-\int_{S^1} \nabla_\zeta \frac {\partial H} {\partial \lambda} d\theta
- \int_{S^1} \nabla_\ell \frac {\partial H} {\partial \lambda} d\theta
\end{array}\right).
\end{equation}
The  operator $D_{(\gamma,\lambda)}$ is formally obtained from the linearized operator $D_{(u,\lambda)}$ in equation~\eqref{eq:Dulambda}
by setting $(u(s,\theta),\lambda(s))\equiv (\gamma(\theta),\lambda)$ and
$(\zeta(s,\theta),\ell(s)) \equiv (\zeta(\theta),\ell)$.

Given a unitary trivialization of $u^*TW$ and an orthogonal trivialization of $\lambda^*T\Lambda$, the operator $D_{(u,\lambda)}$ defined by~\eqref{eq:Dulambda} can be written for $p=2$ as
$$
D_{(u,\lambda)}:H^1(\R\times S^1,\R^{2n}) \times H^1(\R,\R^m) \to
L^2(\R\times S^1,\R^{2n})\times L^2(\R,\R^m),
$$
$$
D_{(u,\lambda)}(\zeta,\ell) = \left(\begin{array}{c} \p_s\zeta \\
    \p_s\ell \end{array}\right) +
A(s)\left(\begin{array}{c}\zeta\\\ell\end{array}\right).
$$
Here $A(s):H^1(S^1,\R^{2n}) \times \R^m \to
L^2(S^1,\R^{2n})\times \R^m$ has the property that $A(s)\to A^\pm$,
$s\to\pm\infty$ and $A^\pm$ coincide through the given trivializations
with the asymptotic operators
$D_{(\og,\olambda)}$ and $D_{(\ug,\ulambda)}$, which are bijective in view of our nondegeneracy assumption.
The operators $A(s)$ are of order one and their principal part is self-adjoint. Thus, up to an order zero (and hence compact) perturbation, we can assume for the purpose of computing the index that $A(s)$ is self-adjoint for all $s\in \R$. In this situation, the Fredholm index of the operator $D_{(u,\lambda)}$ is equal to
the spectral flow of the family of self-adjoint operators $A(s)$, $s
\in \R$~\cite[Theorem~A]{RSspec}.

The spectral flow is described as follows. Let us call $s\in\R$ a \emph{crossing} if $\ker\, A(s)\neq 0$, and define the \emph{crossing form} $\Gamma(A,s):\ker\, A(s) \to\R$ by $\Gamma(A,s)\xi=\langle \xi,\frac d {ds} A(s)\xi\rangle$. A crossing $s\in\R$ is called \emph{regular} if the crossing form $\Gamma(A,s)$ is nondegenerate. Such crossings are isolated. If all crossings are nondegenerate, the \emph{spectral flow} is given by the sum over all crossings of the signature of the crossing form $\Gamma(A,s)$, which is the number of positive minus the number of negative eigenvalues. Heuristically, the spectral flow measures the net difference between the number of eigenvalues of $A(s)$ which cross from $-$ to $+$ and those  which cross from $+$ to $-$. Up to a compact perturbation we can always assume that all the crossings of $A(s)$ are regular.

In view of \eqref{eq:Dasy}, the operators can be written in the
given trivializations of $TW$ and $T\Lambda$ along $u$ and
$\lambda$ as
\begin{equation} \label{eq:A}
A(s)(\zeta,\ell) = \left( \begin{array}{c}
J_0 \partial_\theta \zeta(\theta) + S(s,\theta) \zeta(\theta)
+ C(s,\theta)^T \ell \\
\int_{S^1} C(s,\theta) \zeta(\theta) \, d\theta + \int_{S^1} D(s,\theta) \, d\theta \ \ell
\end{array} \right) ,
\end{equation}
where $S(s,\theta) = S(s,\theta)^T$ and $D(s,\theta) = D(s,\theta)^T$
are symmetric matrices.

\subsubsection{Computation of $\ker\,A(s)$, $s\in\R$}

We define
$$
\Psi:\R\times [0,1]\to\Sp(2n)
$$
by $\dot \Psi(s,\theta) = J_0 S(s,\theta) \Psi(s,\theta)$ and
$\Psi(s,0)=\one$, so that
$$
\lim_{s\to-\infty} \Psi(s,\cdot)=\overline \Psi(\cdot), \qquad
\lim_{s\to\infty} \Psi(s,\cdot)=\underline \Psi(\cdot).
$$

For $(\zeta, \ell) \in \ker A(s)$, we write $\zeta(\theta) =
\Psi(s,\theta) \eta(\theta)$ for some smooth
function $\eta : [0,1] \to \R^{2n}$. Substituting this in the first
component of $A(s)(\zeta, \ell)$, we obtain
\begin{equation}  \label{eq:diffeta}
\dot \eta(\theta) = \Psi(s,\theta)^{-1} J_0 C(s,\theta)^T \ell .
\end{equation}

We define $X : \R\times [0,1] \to \Mat_{2n,m}(\R)$ by
\begin{equation}\label{eq:X}
\dot X(s,\theta) = \Psi(s,\theta)^{-1} J_0 C(s,\theta)^T
\end{equation}
and $X(s,0) = 0$. The solution
of \eqref{eq:diffeta} is then $\eta(\theta) = X(s,\theta) \ell + \eta(0)$, so that
\begin{equation} \label{eq:zeta}
\zeta(\theta) = \Psi(s,\theta) \zeta_0 + \Psi(s,\theta) X(s,\theta) \ell ,
\end{equation}
with $\zeta_0 = \zeta(0) = \eta(0)$. Comparing~\eqref{eq:zeta}
with~\eqref{eq:linflow} we see that
$$
\lim_{s\to-\infty} X(s,\cdot)=\overline X(\cdot), \qquad
\lim_{s\to\infty} X(s,\cdot)=\underline X(\cdot).
$$
The solution $\zeta(\theta)$ given by~\eqref{eq:zeta} descends to $S^1 = \R/\Z$
if and only if
\begin{equation}  \label{eq:degen1}
\zeta_0 = \Psi(s,1) \zeta_0 + \Psi(s,1) X(s,1) \ell .
\end{equation}
Substituting the expression~\eqref{eq:zeta} for $\zeta(\theta)$ in the
second component of\break $A(s)(\zeta, \ell)$, we obtain
\begin{equation}  \label{eq:degen2}
\int_0^1 C(s,\theta) \Psi(s,\theta)  d\theta \, \zeta_0
+ \int_0^1 \big(C(s,\theta) \Psi(s,\theta) X(s,\theta) + D(s,\theta)\big) d\theta \, \ell = 0 .
\end{equation}
We now notice that we have
\begin{equation} \label{eq:CPsi}
C(s,\theta)\Psi(s,\theta) = \dot X(s,\theta)^T J_0,
\end{equation}
which implies in particular
\begin{equation}  \label{eq:CT=B}
\int_0^\theta C(s,\tau) \Psi(s,\tau) \, d\tau = \int_0^\theta \dot
X(s,\tau)^T J_0\, d\tau = X(\theta)^T J_0.
\end{equation}

We define
$$
E : \R\times [0,1] \to \Mat_m(\R)
$$
by
\begin{equation}  \label{eq:E}
E(s,\theta) = \int_0^\theta \big( C(s,\tau) \Psi(s,\tau) X(s,\tau) +
D(s,\tau) \big)d\tau  - \frac 12 X(s,\theta)^T J_0 X(s,\theta).
\end{equation}
We claim that the matrix $\frac 1 2 X(s,\theta)^T J_0 X(s,\theta)$ is the
anti-symmetric part of the matrix $\int_0^\theta
C(s,\tau)\Psi(s,\tau)X(s,\tau)d\tau$, so that $E(s,\theta)$ is
symmetric. Omitting the $s$-variable for clarity and using that
$C(\tau)\Psi(\tau)=\dot X(\tau)^TJ_0$ we obtain
\begin{eqnarray*}
\lefteqn{\hspace{-1.5cm}\int_0^\theta C(\tau)\Psi(\tau)X(\tau)d\tau - \int_0^\theta X(\tau)^T
\Psi(\tau)^TC(\tau)^Td\tau}\\
& = & \int_0^\theta \dot
X(\tau)^TJ_0X(\tau)d\tau+\int_0^\theta X(\tau)^TJ_0\dot X(\tau)d\tau
\\
& = & X(\theta)^TJ_0X(\theta).
\end{eqnarray*}
It follows that $E(s,\theta)$ is the symmetric part of
$\int_0^\theta
(C\Psi X+D)(s,\tau)d\tau$. Comparing
this with~\eqref{eq:Easy}, it follows that
$$
\lim_{s\to-\infty} E(s,\cdot)=\overline E(\cdot), \qquad
\lim_{s\to\infty} E(s,\cdot)=\underline E(\cdot).
$$

With these notations in place, we see that~\eqref{eq:degen1}
and~\eqref{eq:degen2} are equivalent to the $(2n+m) \times (2n+m)$
system of linear equations
\begin{equation}  \label{eq:degen}
\left( \begin{array}{cc}
\Psi(s,1) - \one & \Psi(s,1) X(s,1) \\
X(s,1)^TJ_0 & E(s,1) + \frac 1 2 X(s,1)^TJ_0X(s,1)
\end{array} \right)
\left( \begin{array}{c}
\zeta_0 \\ \ell
\end{array} \right) =
\left( \begin{array}{c}
0 \\ 0
\end{array} \right) .
\end{equation}
The solutions of the system~\eqref{eq:degen} are in bijective correspondence
with the elements $(\zeta, \ell) \in \ker A(s)$ through equation~\eqref{eq:zeta}.
On the other hand, it follows from the definition of $\cS_{n,m}$ that
solutions of~\eqref{eq:degen} are in bijective correspondence with elements
$$
(\zeta_0,\ell,0)\in \ker\, \big(M(\Psi(s,1), X(s,1)
, E(s,1)) - \one\big).
$$
Since $(0,0,v)\in \ker\, \big(M(\Psi(s,1), X(s,1)
, E(s,1)) - \one\big)$ for all $v\in \R^m$, we infer that $\ker A(s)
\neq 0$ if and only if
\begin{equation}
\dim\ker \, \big( M(\Psi(s,1), X(s,1) , E(s,1)) - \one \big) >m.
\end{equation}

\begin{remark} {\rm
We associated to each operator $A(s)$ of the form \eqref{eq:A}
a path of matrices $M : [0,1] \to \cS_{n,m}$,
$M(\theta) = M(\Psi(\theta), X(\theta), E(\theta))$  such that $M(0)
=\one$. Conversely, any such path $M$ determines a unique
operator $A(s)$ of the form \eqref{eq:A} by the formulas
\begin{eqnarray*}
S(\theta) &=& -J_0 \dot \Psi(\theta) \Psi(\theta)^{-1} \\
C(\theta) &=& \dot X(\theta)^T \Psi(\theta)^T J_0  \\
D(\theta) &=& \dot E(\theta) + \mathrm{Sym} \left( X(\theta)^T J_0 \dot X(\theta) \right).
\end{eqnarray*}}
\end{remark}

\subsubsection{Computation of the crossing form $\Gamma(A,s)$ on $\ker\, A(s)$} \label{sec:Gamma}
We have
$$
\frac d {ds}  A(s) (\zeta, \ell) = \left( \begin{array}{c}
\partial_s S(s,\theta) \zeta(\theta) + \partial_s C(s,\theta)^T \ell \\
\int_{S^1} \partial_s C(s,\theta) \zeta(\theta)  d\theta
+ \int_{S^1} \partial_s D(s,\theta) d\theta \ \ell
\end{array} \right) .
$$
Since $(\zeta, \ell) \in \ker A(s)$, we have
$\zeta(\theta)=\Psi(s,\theta)\zeta_0+\Psi(s,\theta)X(s,\theta)\ell$. We obtain
\begin{eqnarray}
\lefteqn{\Gamma(A,s)(\zeta,\ell)} \nonumber \\
&=&\langle(\zeta,\ell),\frac d {ds} A(s) (\zeta,\ell)\rangle\nonumber \\
&=&\int_{S^1} \left\langle \zeta(\theta) , \partial_s S(s,\theta) \zeta(\theta)
+ \partial_s C(s,\theta)^T \ell \right\rangle \, d\theta \nonumber \\
&&+ \left\langle \ell, \int_{S^1} \partial_s C(s,\theta) \zeta(\theta) \, d\theta
+ \int_{S^1} \partial_s D(s,\theta) \, d\theta \ \ell \right\rangle \nonumber \\
&=& \int_0^1  (\zeta_0 + X (s,\theta) \ell)^T \Psi(s,\theta)^T
\partial_s S(s,\theta) \Psi(s,\theta) (\zeta_0 +
X(s,\theta) \ell) \, d\theta \label{eq:first-term} \\
&& + \int_0^1  (\zeta_0 + X (s,\theta) \ell)^T \Psi(s,\theta)^T
\partial_s C(s,\theta)^T \ell  \, d\theta \nonumber \\
&& + \ \ell^T \int_0^1 \partial_s C(s,\theta)
\Psi(s,\theta) (\zeta_0 + X(s,\theta) \ell) \, d\theta  \nonumber \\
&& + \ \ell^T \int_{S^1} \partial_s D(s,\theta) d\theta \, \ell.    \label{eq:cross}
\end{eqnarray}
Let us define symmetric matrices $\widehat S(s,\theta)$ by
$\partial_s \Psi(s,\theta) = J_0 \widehat S(s,\theta)
\Psi(s,\theta)$. The condition $\Psi(s,0)=\one$ implies $\widehat
S(s,0)=0$. We claim that (see
also~\cite[proof of Lemma~2.6]{Salamon-lectures})
\begin{equation} \label{eq:stheta}
\Psi(s,\theta)^T \partial_s S(s,\theta) \Psi(s,\theta)
= \partial_\theta\left( \Psi(s,\theta)^T \widehat S(s,\theta) \Psi(s,\theta) \right).
\end{equation}
Dropping the $(s,\theta)$ variables for clarity, we have~\cite{Salamon-lectures}
\begin{eqnarray*}
\partial_\theta  \left( \Psi^T \widehat S \Psi \right)
 &=&  \Psi^T S^T (- J_0)  \widehat S \Psi
+ \Psi^T \partial_\theta ( \widehat S  \Psi ) \\
 &=& - \Psi^T S \partial_s \Psi
+ \Psi^T \partial_\theta \left( - J_0 \partial_s \Psi \right)  \\
 &=& - \Psi^T S \partial_s \Psi  - \Psi^T J_0 \partial_s
 \partial_\theta \Psi  \\
 &=& - \Psi^T S \partial_s \Psi  - \Psi^T J_0 \partial_s
 (J_0 S \Psi)  \\
 &=& \Psi^T \partial_s S \Psi .
\end{eqnarray*}
Using~\eqref{eq:stheta}, the term~\eqref{eq:first-term} becomes
\begin{eqnarray*}
\lefteqn{ \int_0^1  (\zeta_0 + X (s,\theta) \ell)^T
\partial_\theta  \left( \Psi(s,\theta)^T \widehat S(s,\theta) \Psi(s,\theta) \right)
(\zeta_0 + X(s,\theta) \ell) \, d\theta } \\
&=& (\zeta_0 + X(s,1) \ell)^T  \Psi(s,1)^T \widehat S(s,1) \Psi(s,1)  (\zeta_0 + X (s,1) \ell) \\
&& - \ell^T  \int_0^1 \partial_\theta X(s,\theta)^T
\Psi(s,\theta)^T \widehat S(s,\theta) \Psi(s,\theta) (\zeta_0 + X(s,\theta) \ell) \, d\theta  \\
&& -  \int_0^1 (\zeta_0 + X (s,\theta) \ell)^T  \Psi(s,\theta)^T \widehat S(s,\theta) \Psi(s,\theta)
\partial_\theta X(s,\theta)  \, d\theta \ \ell \\
&=& \zeta_0^T    \widehat S(s,1)   \zeta_0
 + \ell^T  \int_0^1 C(s,\theta) J_0
\widehat S(s,\theta) \Psi(s,\theta) (\zeta_0 + X(s,\theta) \ell) \, d\theta  \\
&& -  \int_0^1 (\zeta_0 + X (s,\theta) \ell)^T  \Psi(s,\theta)^T \widehat S(s,\theta)
J_0 C(s,\theta)^T  \, d\theta \ \ell \\
&=& \zeta_0^T  \widehat S(s,1) \zeta_0
 + \ell^T  \int_0^1 C(s,\theta) \partial_s \Psi(s,\theta) (\zeta_0 + X(s,\theta) \ell) \, d\theta  \\
&& +  \int_0^1 (\zeta_0 + X (s,\theta) \ell)^T  \partial_s \Psi(s,\theta)^T
C(s,\theta)^T  \, d\theta \ \ell .
\end{eqnarray*}
The second equality uses~\eqref{eq:degen1} and~\eqref{eq:X}. Thus, equation \eqref{eq:cross} becomes
\begin{eqnarray}
\Gamma(A,s)(\zeta,\ell) &=&
 \zeta_0^T \widehat S(s,1) \zeta_0
  + \ell^T  \int_0^1  \partial_s \big(C(s,\theta) \Psi(s,\theta) \big)
 (\zeta_0 + X(s,\theta) \ell) d\theta  \nonumber \\
&& +  \int_0^1 (\zeta_0 + X (s,\theta) \ell)^T  \partial_s \big(\Psi(s,\theta)^T
C(s,\theta)^T \big) d\theta \, \ell \nonumber \\
&& + \ \ell^T \int_{S^1} \partial_s D(s,\theta) d\theta \, \ell \nonumber \\
&=& \zeta_0^T \widehat S(s,1) \zeta_0 + \ell^T \big(\partial_s X(s,1)^T J_0\big) \zeta_0 +
\zeta_0^T\big(-J_0\partial_s X(s,1)\big) \ell \nonumber \\
&& + \ \ell^T
\int_0^1\Big(\partial_s (C\Psi)X + X^T\partial_s (\Psi^T C^T) + \partial_s D\Big)
\ell . \label{eq:lastterm}
\end{eqnarray}
We used~\eqref{eq:CPsi} in the second equality. We claim that the matrix of the quadratic form $\Gamma(A,s)$ acting on
the space of elements $(\zeta_0,\ell)$
satisfying~\eqref{eq:degen1} is given by
\begin{equation} \label{eq:crossA2}
\left( \begin{array}{cc}
\widehat S(s,1) & - J_0\partial_s X(s,1) \\
\partial_s X(s,1)^T J_0 &
\partial_s E(s,1) -\mathrm{Sym}\big(X^T(s,1) J_0 \partial_s X(s,1)\big)
\end{array} \right) .
\end{equation}
This amounts to proving the identity
\begin{equation} \label{eq:longidentity}
\partial_sE(s,1) -\mathrm{Sym}\big(X^T(s,1) J_0 \partial_s X(s,1)\big)=\int_0^1\partial_s(C\Psi)X + X^T\partial_s (\Psi^T C^T) + \partial_s D
\end{equation}
for the term in the lower right corner. This is seen by a direct computation:
\begin{eqnarray*}
\lefteqn{\int_0^1\partial_s(C\Psi)X + X^T\partial_s (\Psi^T C^T) + \partial_s D} \\
&=& \partial_s \mathrm{Sym}\int_0^1(C\Psi X +D) +\mathrm{Sym}\int_0^1\partial_s(C\Psi)X
- \mathrm{Sym}\int_0^1C\Psi\partial_s X \\
&=& \partial_s E(s,1) +\mathrm{Sym}\int_0^1\partial_s(C\Psi)X \\
&& - \ \mathrm{Sym}\big(X(s,1)^TJ_0\partial_s X(s,1)\big) + \mathrm{Sym}\int_0^1X^TJ_0\partial_s\dot X \\
&=& \partial_s E(s,1) - \mathrm{Sym}\big(X(s,1)^TJ_0\partial_s X(s,1)\big).
\end{eqnarray*}
The second equality uses the definition of $E$, the identity $C\Psi=\dot X^T J_0$ from~\eqref{eq:CPsi}, and integration by parts. The third equality uses that $X^TJ_0\partial_s \dot X=-\big(\partial_s(C\Psi)X\big)^T$, which is a consequence of $C\Psi=\dot X^T J_0$.

\subsection{Proof of the Main Theorem}

Let us compute the crossing form $\Gamma(M,s)$ for the
Robbin-Salamon index of the path
$$
s\mapsto M(s,1)=M(\Psi(s,1),X(s,1),E(s,1)).
$$
By definition, the crossing form is $\Gamma(M,s)(\zeta_0,\ell,v)=\langle
(\zeta_0,\ell,v), Q(s)(\zeta_0,\ell,v)\rangle$, with $Q(s):=-\widetilde J_0
\partial_s M(s,1)M(s,1)^{-1}$.
Using~\eqref{eq:inverse} and the definition of $\widehat
S(s,1)=-J_0\partial_s \Psi(s,1) \Psi(s,1)^{-1}$ from~\S\ref{sec:Gamma}, a
straightforward computation shows that $Q(s)$ is given by
$$
\left(
  \begin{array}{ccc}
\widehat S(s,1) & -J_0\Psi(s,1)\partial_s X(s,1) & 0
\\
\partial_s X(s,1)^T\Psi(s,1)^TJ_0 & \partial_s E(s,1) + \mathrm{Sym}\big(X(s,1)^TJ_0\partial_s X(s,1)\big) & 0 \\
0 & 0 & 0
  \end{array}\right).
$$

The key observation now is that, for any $(\zeta_0,\ell,0)\in \ker \,
(M(s,1)-\one)$, we have
$$
\Gamma(M,s)(\zeta_0,\ell,0) =\Gamma(A,s)(\zeta,\ell),
$$
with $\zeta(\theta)=\Psi(s,\theta)\zeta_0 +
\Psi(s,\theta)X(s,\theta)\ell$. This is seen by a direct
computation, substituting
$\zeta_0=\Psi(s,1)\zeta_0+\Psi(s,1)X(s,1)\ell$ in the non-diagonal terms
of $\Gamma(M,s)(\zeta_0,\ell,0)$:
\begin{eqnarray*}
\lefteqn{\Gamma(M,s)(\zeta_0,\ell,0)} \\
&=& \zeta_0^T\widehat S \zeta_0 + \ell^T\big(\partial_s E+\mathrm{Sym}(X^TJ_0\partial_s X)\big)\ell
+ \ell^T \partial_s X^T\Psi^TJ_0\zeta_0 + \zeta_0(-J_0\Psi\partial_s X)\ell \\
&=& \zeta_0^T\widehat S \zeta_0 + \ell^T\big(\partial_s E+\mathrm{Sym}(X^TJ_0\partial_s X)\big)\ell \\
&& + \ \ell^T\partial_sX^TJ_0\zeta_0 + \ell^T\partial_s X^T J_0 X\ell + \zeta_0^T(-J_0\partial_s X)\ell + \ell^T X^T (-J_0) \partial_s X \ell \\
&=& \zeta_0^T\widehat S \zeta_0 + \ell^T\big(\partial_s E+\mathrm{Sym}(X^TJ_0\partial_s X)\big)\ell \\
&& + \ \ell^T\partial_sX^TJ_0\zeta_0 + \zeta_0^T(-J_0\partial_s X)\ell - 2\ell^T \mathrm{Sym}(X^TJ_0\partial_s X)\ell.
\end{eqnarray*}
This last expression is equal to $\Gamma(A,s)(\zeta,\ell)$ in view of~\eqref{eq:crossA2}.

By Proposition~\ref{prop:RSstratum} in Appendix~\ref{app:RS} (applied with
$E(s)\equiv \{0\}\oplus\{0\}\oplus\R^m$), it follows that the spectral
flow of $A(s)$ coincides with the Robbin-Salamon index of the
degenerate path $s\mapsto M(s,1)$. Thus
 $$
 \ind \, D_{(u,\lambda)} = \mu_{RS}\left( M(\Psi(s,1), X(s,1), E(s,1)), s \in \R \right) .
 $$
 By the \emph{(Homotopy)} and \emph{(Catenation)} axioms for the
 Robbin-Salamon index~\cite{RS}, and using that
 $\lim_{s\to-\infty}M(s,\theta)=\overline M(\theta)$ and
 $\lim_{s\to\infty}M(s,\theta)=\underline M(\theta)$, we obtain
\begin{eqnarray*}
 \ind \, D_{(u,\lambda)}
 &=& \mu_{RS}\left( M(\underline \Psi(\theta), \underline X(\theta),
 \underline E(\theta)), \theta \in [0,1] \right) \\
&& -  \mu_{RS}\left( M(\overline \Psi(\theta), \overline X(\theta),
 \overline E(\theta)), \theta \in [0,1] \right) \\
&=& \mu(\ug, \ulambda) - \mu(\og, \olambda).
\end{eqnarray*}
\hfill{$\square$}

\appendix

\section{Examples} \label{app:examples}

We explain in this appendix several examples in which parametrized Hamiltonian action functionals appear naturally.

\subsection{$S^1$-equivariant Floer homology~\cite{Viterbo99, BOGysin}} One takes $\Lambda=ES^1$ (or rather a finite-dimensional approximation of it), where $ES^1$ is up to equivariant homotopy the unique contractible $S^1$-space carrying a free action. The $S^1$-equivariant Floer homology groups are defined using Hamiltonians which are invariant
$$
H(\theta+\tau,x,\tau\lambda)=H(\theta,x,\lambda), \qquad \tau\in S^1.
$$
Compared to the classical, non-equivariant Floer homology groups, these carry refined information coming from the $S^1$-action on $\cL W$ given by reparametrization at the source.

\subsection{Rabinowitz-Floer homology~\cite{Cieliebak-Frauenfelder}} 
\label{sec:RFH}
One takes $\Lambda=\R$ and
$$
H(\theta,x,\lambda)=\lambda K(x),
$$
with $K:W\to \R$ an autonomous Hamiltonian. The critical points of $A_H$ solve the equations $\dot\gamma = \lambda X_K$ and $\int_{S^1} K(\gamma(\theta))\,d\theta = 0$, which are equivalent to $\dot\gamma=\lambda X_K$ and $\mathrm{im}\,\gamma\subset K^{-1}(0)$. Thus critical points of $A_H$ correspond to \emph{closed characteristics on the fixed energy level} $K^{-1}(0)$.

\subsection{Rabinowitz-Floer homology for leafwise intersections of hypersurfaces~\cite{AF10}} One takes again
$\Lambda=\R$ but
$$
H(\theta,x,\lambda)=\lambda \rho(\theta)K(x)+F(\theta,x).
$$
Here $\rho:S^1\to \R$ is such that $\mathrm{supp}(\rho)\subset[0,\frac 1 2]$ and $\int_{S^1}\rho(\theta)\, d\theta =1$, while $F(\theta,\cdot)=0$ for $\theta\in[0,\frac 1 2]$. The equations for a critical point $(\gamma,\lambda)$ are equivalent to $x:=\gamma(\frac 1 2)\in K^{-1}(0)$ and $\gamma(0)=\gamma(1)=\phi^1_F(x)\in L_x$, with $L_x$ the orbit of the characteristic flow passing through $x$. One calls $x$ a \emph{leafwise intersection of the flow $\phi^1_F$}.

\subsection{Rabinowitz-Floer homology for leafwise coisotropic intersections~\cite{Kang10}} Let $\cK=(K_1,\dots,K_k):W\to\R^k$ be a system of autonomous Poisson-commuting Hamiltonians. The preimage $\cK^{-1}(a)$ of a regular value $a\in\R^k$ is then a coisotropic submanifold which is foliated by isotropic leaves that are tangent to the span of the Hamiltonian vector fields $X_{K_1},\dots,X_{K_k}$. Let $\langle\cdot,\cdot\rangle$ be the Euclidean scalar product on $\R^k$, take $\Lambda=\R^k$ and define
$$
H(\theta,x,\lambda)=\rho(\theta)\langle \lambda,\cK(x)\rangle + F(\theta,x)
$$
with $\rho$ and $F$ as above. The equations for a critical point $(\gamma,\lambda)$ are equivalent to $x:=\gamma(\frac 1 2)\in \cK^{-1}(0)$ and $\gamma(0)=\gamma(1)=\phi^1_F(x)\in L_x$, with $L_x$ the isotropic leaf through $x$. One calls $x$ a \emph{leafwise coisotropic intersection of the flow $\phi^1_F$}.


\subsection{Floer homology for families~\cite{Hutchings03}} This construction generalizes the setup that we consider in this paper to nontrivial fibrations. Assume $\pi:E\to \Lambda$ is a symplectic fibration endowed with an exact $2$-form $\Omega=d\Theta\in\Omega^2(E;\R)$ which restricts to a symplectic form in the fibers (this is called a \emph{coupling form}). Let $H:S^1\times E\to \R$ be a Hamiltonian and let $\cL_\Lambda E$ denote the space of loops $\gamma$ in $E$ such that $\pi\circ \gamma$ is constant. One considers the action functional
$$
A_H:\cL_\Lambda E\to \R, \qquad \gamma\mapsto -\int_\gamma\Theta-\int_{S^1}H(\theta,\gamma(\theta))\, d\theta.
$$
The critical points of $A_H$ are the basis for the \emph{Floer homology groups of the family $(E,\Lambda)$}~\cite{Hutchings03}. Since the fibration is locally trivial and the critical points of $A_H$ are contained in a fiber, the definition of the index that we give in this paper applies also to this more general setup.

\section{The parametrized Robbin-Salamon index} \label{app:RS}

We summarize in this appendix the properties of the
Robbin-Salamon index on paths with values in the subgroup $\cS_{n,m}\subset \Sp(2n+2m)$ defined in~\S\ref{sec:paramRS}. We also prove a result (Proposition~\ref{prop:RSstratum}) which is used in the proof of our Main Theorem.

We recall that $\cS_{n,m}$ consists of matrices of the form
$$
M=M(\Psi,X,E)=\left(\begin{array}{ccc}
\Psi & \Psi X & 0 \\
0 & \one & 0 \\
X^TJ_0 & E+\frac 1 2 X^T J_0 X & \one
\end{array}\right),
$$
with $\Psi\in\Sp(2n)$,
$X\in\mathrm{Mat}_{2n,m}(\R)$, and $E\in \mathrm{Mat}_m(\R)$
symmetric. We have denoted by $J_0:=\left(\begin{array}{cc} 0 &
    -\one \\ \one & 0 \end{array}\right)$ the standard complex
structure on $\R^{2n}$, so that $\Psi\in\Sp(2n)$ if and only if
$\Psi^T J_0\Psi=J_0$. The standard complex structure on $\R^{2n}\times
\R^{2m}$ is
$$
\widetilde J_0:=\left(
  \begin{array}{ccc}
J_0 & 0 & 0 \\
0 & 0 & -\one \\
0 & \one & 0
  \end{array}\right),
$$
and we have $M^T\widetilde J_0 M=\widetilde J_0$. Note that we have
natural embeddings (which respect the group structure)
$$
\cS_{n,m}\times \cS_{n',m'}\hookrightarrow \cS_{n+n',m+m'}
$$
which associate to $M=M(\Psi,X,E)\in\cS_{n,m}$ and
$M'=M(\Psi',X',E')\in\cS_{n',m'}$ the matrix
$$
M\oplus M':=M(\Psi\oplus \Psi',X\oplus X',E\oplus
E')\in\cS_{n+n',m+m'}.
$$
The space $\cS_{n,m}$ is stratified as $\coprod_{k=0}^{2n+m}
\cS^k_{n,m}$, with
$$
\cS^k_{n,m}:=\big\{M\in\cS_{n,m} \ : \ \dim\, \ker \, (M-\one) = m+k. \big\}.
$$

The following are general properties of the Robbin-Salamon index $\mu=\mu_{RS}$ defined on paths with values in $\Sp(2n+2m)$~\cite[Theorem~4.1 and Theorem~4.7]{RS}.
\begin{description}
\item[(Homotopy)] If $M,M':[a,b]\to \cS_{n,m}$ are homotopic with
  fixed endpoints then
$$
\mu(M)=\mu(M');
$$
\item[(Catenation)] For any $c\in [a,b]$ we have
$$
\mu(M)=\mu(M|_{[a,c]}) + \mu(M|_{[c,b]});
$$
\item[(Naturality)] For any path $P:[a,b]\to \Sp(2n)\times \Sp(2m)$ of
the form
\begin{equation} \label{eq:Ptheta}
P(\theta)=\left(\begin{array}{ccc}
\Phi(\theta) & 0 & 0 \\
0 & A(\theta) & 0 \\
0 & 0 & A(\theta)
  \end{array}\right)
\end{equation}
(with $\Phi(\theta)\in\Sp(2n)$ and $A(\theta)\in\OO(m)$, hence $PMP^{-1}\in\cS_{n,m}$), we have
$$
\mu(PMP^{-1})=\mu(M);
$$
\item[(Product)] For any $M\in \cS_{n,m}$ and $M'\in \cS_{n',m'}$ we
  have
$$
\mu(M\oplus M')=\mu(M) + \mu(M');
$$
\item[(Zero)] For any path $M:[a,b]\to \cS^k_{n,m}$ we have
$$
\mu(M)=0;
$$
\item[(Integrality)] Given a path $M:[a,b]\to \cS_{n,m}$ with $M(a)\in
  \cS^{k_a}_{n,m}$, $M(b)\in \cS^{k_b}_{n,m}$, we have
$$
\mu(M)+\frac {k_a-k_b} 2 \in \Z;
$$
\end{description}

The next statement summarizes properties that are specific to the index function restricted to paths with values in $\cS_{n,m}$. 

\begin{proposition} \label{prop:mu}
The Robbin-Salamon index $\mu=\mu_{RS}$ defined on paths $M:[a,b]\to
\cS_{n,m}$, $M(\theta)=M(\Psi(\theta),X(\theta),E(\theta))$ has the
following properties.
\begin{description}
\item[(Loop)] For any loop $P:[a,b]\to \Sp(2n)\times \Sp(2m)$ of the
form~\eqref{eq:Ptheta}, we have
$$
\mu(PM)=\mu(M)+ 2\mu(\Phi);
$$
\item[(Splitting)] Given $M=M(\Psi,0,E):[a,b]\to \cS_{n,m}$, we have
$$
\mu(M)=\mu(\Psi) + \frac 1 2 \mathrm{sign}\, E(b) - \frac 1 2
\mathrm{sign}\, E(a);
$$
\item[(Signature)] Given symmetric matrices $E\in \R^{m\times m}$
  and $S\in\R^{2n\times 2n}$ with $\| S \| <2\pi$, we have
$$
\mu\big\{M(\exp(J_0St), 0,tE),\ t\in[0,1]\big\} = \frac 1 2 \mathrm{sign}(S) +
\frac 1 2 \mathrm{sign}(E);
$$
\item[(Determinant)] Given a path $M=M(\Psi,X,E):[a,b]\to \cS_{n,m}$ with
  $M(a)=\one$ and $M(b)\in\cS^0_{n,m}$, we have
$$
(-1)^{n+\frac m 2 - \mu(M)}=\mathrm{sign}\, \det \,
\left(\begin{array}{cc}
\Psi-\one & \Psi X \\
X^TJ_0 & E+\frac 1 2 X^TJ_0 X
\end{array}\right).
$$
We have denoted for simplicity $\Psi=\Psi(b)$, $X=X(b)$, $E=E(b)$.
\item[(Involution)] For any $M=M(\Psi,X,E):[a,b]\to\cS_{n,m}$ we have
$$
\mu(M(\Psi,X,E))=\mu(M(\Psi,-X,E))
$$
and
$$
\mu(M(\Psi^{-1},\Psi X,-E))=\mu(M(\Psi^T,J_0\Psi X,-E))=-\mu(M(\Psi,X,E)).
$$
\end{description}
\end{proposition}

\begin{proof}


To prove the \emph{(Loop)} property we use the equality
$$
\mu(PM)=\mu(M)+2\mu(P)=\mu(M)+2\mu(\Phi)+
2\mu\left(\begin{array}{cc} A& 0 \\ 0 & A \end{array}\right).
$$
Since $\pi_1(\OO(m))=\Z/2\Z$ and $\pi_1(\Sp(2m))=\Z$, the last term
vanishes.

The \emph{(Splitting)} property follows from the \emph{(Product)}
property and the normalization axiom for the Robbin-Salamon index of a
symplectic shear.

The \emph{(Signature)} property follows from the
\emph{(Splitting)} property and from the identity
$\mu_{RS}(\exp(J_0St))=\frac 1 2
\mathrm{sign}(S)$~\cite[Theorem~3.3.(iv)]{SZ92}.


We prove the \emph{(Involution)} property. The first identity
$\mu(M(\Psi^{-1},Y,-E))=\mu(M(\Psi^T,J_0Y,-E))$ follows from the
\emph{(Naturality)} axiom by conjugating with the constant path
$J_0\oplus \one_{2m}$. The identity
$\mu(M(\Psi,X,E))=\mu(M(\Psi,-X,E))$ follows by conjugating twice
with $J_0\oplus \one_{2m}$. 
Finally, using~\eqref{eq:inverse} we obtain 
$-\mu(M(\Psi,X,E))=\mu(M(\Psi,X,E)^{-1})=\mu(M(\Psi^{-1},-\Psi X,-E))=\mu(M(\Psi^{-1},\Psi X,-E))$.

It remains to prove the \emph{(Determinant)} property. Given a path
$N:[0,1]\to \Sp(2n+2m)$ satisfying $N(0)=\one$ and $\det\,
(N(1)-\one)\neq 0$, we have~\cite[Theorem~3.3.(iii)]{SZ92}
$$
(-1)^{n+m-\mu_{RS}(N)}=\mathrm{sign}\, \det \, (N(1)-\one).
$$
We construct such a path $N:[a,b+\eps]\to \Sp(2n+2m)$ by catenating
$M=M(\Psi(\theta),X(\theta),E(\theta))$ with the path
$M':[b,b+\eps]\to \Sp(2n+2m)$ given by
$$
M'(b+\theta):=\left(\begin{array}{ccc}
\Psi & \Psi X & \theta \Psi X \\
0 & \one & \theta\one \\
X^TJ_0 & E+\frac 1 2 X^T J_0 X & \one+\theta (E + \frac
1 2 X^TJ_0 X)
\end{array}\right).
$$
We have denoted for simplicity $\Psi:=\Psi(b)$, $X:=X(b)$, and $E:=
E(b)$. Since $M(b)\in\cS^0_{n,m}$, the path $M'$ has a single crossing
at $b$ and the kernel of
$M'(b)-\one=M(b)-\one$ is $\{0\}\oplus \{0\}\oplus \R^m$. The crossing
form at $b$ is $-\one_m$, so that $\mu_{RS}(M')=-\frac m 2$. Thus
$\mu_{RS}(N)=\mu(M)-\frac m 2$. On the other hand
$$
\det\,(N(b+\eps)-\one)=\eps^m(-1)^m\det\,\left(\begin{array}{cc}
\Psi-\one & \Psi X \\
X^TJ_0 & E+\frac 1 2 X^TJ_0 X
\end{array}\right).
$$
This implies the desired statement.
\end{proof}

\begin{example}
The index $\mu(M(\Psi,X,E))$ depends in an essential way on $X$, as
the following example shows. Given $a,b\in\R$, let
$$
\Psi:=\left(\begin{array}{cc} 2 & 0 \\ 0 & \frac 1 2
\end{array}\right), \qquad
X_{a,b}:=\left(\begin{array}{c} a \\ b \end{array}\right), \qquad E:=1.
$$
We denote $M_{a,b}:=M(\Psi,X_{a,b},E)\in\cS_{1,1}$.
It follows from the
\emph{(Determinant)} property that a path in $\cS_{1,1}$ starting at
$\one$ and ending at $M_{0,0}$ has an index in $\frac 1 2
+ 2\Z$, whereas a path in $\cS_{1,1}$ starting at
$\one$ and ending at $M_{1,1}$ has an index in  $\frac 1 2
+ 2\Z+1$ (the value of the relevant determinant is $-\frac 1 2 +\frac
3 2 ab$).
\end{example}

For the rest of this Appendix we place ourselves in $\R^{2N}$ equipped with the standard symplectic form $\omega_0$ and the standard complex structure $J_0$. The next Proposition is relevant for the parametrized Robbin-Salamon index when applied with $N=n+m$ and $E(t)\equiv \{0\}\oplus \{0\}\oplus \R^m$. We recall that, given a path of symplectic matrices $M:[0,1]\to \Sp(2N)$, the crossing form at a point $t\in[0,1]$ is the quadratic form $\Gamma(M,t)$ on $\ker\, (M(t)-\one)$ given by $\Gamma(M,t)(v)=\langle v,-J_0\dot M(t) M(t)^{-1} v\rangle$.

\begin{proposition} \label{prop:RSstratum}
Let $M:[0,1]\to \Sp(2N)$ be a $C^1$-path of symplectic
  matrices with the following property: there exists a continuous family of vector spaces
  $t\mapsto E(t)\subset \R^{2N}$ such
  that $E(t)\subset \ker\,(M(t)-\one)$ and the crossing form
  $\Gamma(M,t)$ induces a nondegenerate quadratic form on $\ker\,(M(t)-\one)/E(t)$. Assume $\omega_0$ has constant rank on $E(t)$. Then
$$
\mu_{RS}(M)=\frac 1 2 \mathrm{sign}\, \Gamma(M,0) + \sum_{t:\dim
  \ker(M(t)-\one)/E(t)>0} \mathrm{sign}\, \Gamma(M,t) + \frac 1 2
\mathrm{sign}\, \Gamma(M,1).
$$
\end{proposition}

\begin{proof} Let us first assume that the rank of $\omega_0$ is constant equal to $0$ on $E(t)$, i.e. $E(t)$ is isotropic. Let us decompose $\R^{2N}=E(t)\oplus J_0 E(t)\oplus F(t)$, where $F(t)$ is the symplectic orthogonal of $E(t)\oplus J_0 E(t)$. Given $\eps>0$ we denote by $\beta_\eps:[0,1]\to[0,\eps]$ a smoothing of the function
$$
t\mapsto\left\{\begin{array}{ll} t, & 0\le t\le \eps,\\
\eps, & \eps\le t\le 1-\eps,\\
1-t, & 1-\eps\le t\le 1.
\end{array}\right.
$$
We define an element $\Phi^0_\eps(t)\in\Sp(2N)$ which has the following matrix form with respect to the splitting $E(t)\oplus J_0 E(t)\oplus F(t)$:
$$
\Phi^0_\eps(t)=\left(\begin{array}{ccc} \one & 0 & 0 \\ \beta_\eps(t) & \one & 0 \\ 0 & 0 & \one \end{array}\right).
$$
We define $\tM(t):=M(t)\Phi^0_\eps(t)$, and we have $\mu_{RS}(\tM)=\mu_{RS}(M)$ since these paths are homotopic with fixed endpoints. We claim that the following equality holds for all $t\in]0,1[$:
\begin{equation} \label{eq:isotropic}
\ker\, (\tM(t)-\one) = \ker\, (M(t)-\one) \cap (J_0 E(t)\oplus F(t)).
\end{equation}
That $\ker\, (M(t)-\one) \cap (J_0 E(t)\oplus F(t))\subset \ker\, (\tM(t)-\one)$ follows from the fact that
$\Phi^0_\eps(t)$ acts by the identity on $J_0 E(t)\oplus F(t)$. Conversely, let $v=v_1+v_2\in\ker\, (\tM(t)-\one)$, with $v_1\in E(t)$ and $v_2\in J_0 E(t)\oplus F(t)$. The identity $\tM(t)v=v$ is equivalent to $M(t)(v_1+\beta_\eps(t)J_0v_1+v_2)=v_1+v_2$, hence to $(M(t)-\one)v_2=-\beta_\eps(t)M(t)J_0v_1$. Using that $M(t)v_1=v_1$ we obtain
$$
0 = \omega_0(v_1,(M(t)-\one)v_2) = -\beta_\eps(t)\omega_0(v_1,M(t)J_0v_1)=-\beta_\eps(t)\omega_0(v_1,J_0v_1).
$$
Since $\beta_\eps(t)\neq 0$, this implies $v_1=0$, so that $v=v_2\in J_0 E(t)\oplus F(t)$
and $(M(t)-\one)v_2=(\tM(t)-\one)v_2=0$, as desired.

Since the restrictions of $M(t)$ and $\tM(t)$ to $J_0E(t)\oplus F(t)$ are the same, it follows that the crossing form $\Gamma(\tM,t)$ coincides with $\Gamma(M,t)$ on $\ker\, (\tM(t)-\one)$ for $t\in]0,1[$. On the other hand, a straightforward computation shows that
\begin{eqnarray} \label{eq:E01}
\mathrm{sign}\,\Gamma(\tM,0)&=&\mathrm{sign}\,\Gamma(M,0) + \dim\, E(0),\\
\mathrm{sign}\,\Gamma(\tM,1)&=&\mathrm{sign}\,\Gamma(M,1) - \dim\, E(1). \nonumber
\end{eqnarray}
Thus, the contributions at the endpoints compensate each other, and the conclusion follows using the definition of the Robbin-Salamon index via crossing forms.

We now assume that the rank of $\omega_0$ on $E(t)$ is equal to $\dim\, E(t)$, i.e. $E(t)$ symplectic. Let us decompose $\R^{2N}=E(t)\oplus F(t)$, where $F(t)$ is the symplectic orthogonal of $E(t)$. Let $J(t)$ be a continuous family of complex structures on $E(t)$ which are compatible with $\omega_0$. For $\eps>0$ we define a path $\Phi^1_\eps:[0,1]\to \Sp(2N)$ whose matrix with respect to the decomposition $E(t)\oplus F(t)$ is
$$
\Phi^1_\eps(t):=\left(\begin{array}{cc} \exp(J(t)\beta_\eps(t)) & 0 \\ 0 & \one \end{array}\right).
$$
We denote $\tM(t):=M(t)\Phi^1_\eps(t)$, so that we have $\mu_{RS}(\tM)=\mu_{RS}(M)$. We claim that
\begin{equation} \label{eq:symplectic}
\ker\,(\tM(t)-\one)=\ker \,(M(t)-\one) \cap F(t)
\end{equation}
for all $t\in]0,1[$, whenever $0<\eps<\pi$. That $\ker \,(M(t)-\one) \cap F(t)\subset \ker\,(\tM(t)-\one)$ follows from the fact that $\Phi^1_\eps(t)$ acts as the identity on $F(t)$. Conversely, let $v=v_1+v_2\in\ker \,(\tM(t)-\one)$ such that $v_1\in E(t)$ and $v_2\in F(t)$. The relation $\tM(t)v=v$ is equivalent to $(M(t)-\one)v_2=(\one- \exp(J(t)\beta_\eps(t)))v_1$. Then
\begin{eqnarray*}
0&=&\omega_0(v_1,(M(t)-\one)v_2)\\
&=&\omega_0(v_1,(\one- \exp(J(t)\beta_\eps(t)))v_1)\\
&=&-\sin(\beta_\eps(t))\omega_0(v_1,J(t)v_1).
\end{eqnarray*}
Since $\sin(\beta_\eps(t))\neq 0$, we obtain $v_1=0$ and the claim follows.

Since the restrictions of $M(t)$ and $\tM(t)$ to $F(t)$ are the same, it follows that the crossing form $\Gamma(\tM,t)$ coincides with $\Gamma(M,t)$ on $\ker\, (\tM(t)-\one)$ for $t\in]0,1[$. On the other hand, a straightforward computation shows that equations~\eqref{eq:E01} still hold, and the conclusion follows.

Finally, we assume that the rank of $\omega_0$ on $E(t)$ lies strictly between $0$ and $\dim\, E(t)$. We choose a continuous splitting $E(t)=E_1(t)\oplus E_0(t)$ with $E_0(t):=E(t)\cap E(t)^{\omega_0}$ isotropic and $E_1(t)=E_0(t)^\perp$ symplectic. Here $E(t)^{\omega_0}$ denotes the symplectic orthogonal of $E(t)$, and $E_0(t)^\perp$ denotes the Euclidean orthogonal of $E_0(t)$ in $E(t)$. We decompose $\R^{2N}=E_1(t)\oplus E_0(t)\oplus J_0E_0(t)\oplus F(t)$, such that $F(t)$ is the symplectic orthogonal of $E_1(t)\oplus E_0(t)\oplus J_0E_0(t)$. Given $0<\eps<\pi$ we define as above two paths $\Phi^0_\eps(t)$  acting as the identity on $E_1(t)\oplus F(t)$, and $\Phi^1_\eps(t)$ acting as the identity on $E_0(t)\oplus J_0E_0(t)\oplus F(t)$. We denote $\tM(t):=M(t)\Phi^0_\eps(t)\Phi^1_\eps(t)$, so that $\mu_{RS}(\tM)=\mu_{RS}(M)$. One proves as above that the crossings of $\tM$ and $M$ on $]0,1[$ are the same, with the same crossing forms on $\ker\, (\tM(t)-\one)$, and moreover equations~\eqref{eq:E01} still hold. This finishes the proof.
\end{proof}

\begin{remark}
The crossing form $\Gamma(M,t)$ vanishes identically on
$E(t)$. Indeed, given a path $v(t)\in
E(t)$ we have $M(t)v(t)=v(t)$ and $\dot M(t) v(t)+M(t)\dot v(t)=\dot
v(t)$. Dropping the $t$-variable for clarity, we have
\begin{eqnarray*}
  \Gamma(M,t)(v(t)) & = & \langle  v,-J_0\dot M M^{-1}v\rangle \\
&=& \langle  v,-J_0(\dot v-M\dot v)\rangle \\
&=& -\langle v,J_0\dot v\rangle + \langle v,(M^{-1})^TJ_0\dot v\rangle
\\
&=& -\langle v,J_0\dot v\rangle + \langle M^{-1}v,J_0\dot v\rangle \ =
\ 0.
\end{eqnarray*}
\end{remark}

\section{Grading in Rabinowitz-Floer homology} 

We give in this section a sample computation of the index within the setup of~\S\ref{sec:RFH}. The index of a critical point of $A_H$ is defined as the Robbin-Salamon index of a corresponding $1$-periodic orbit for the Hamiltonian $\tH:W\times T^*\R\to\R$, $\tH(x,\lambda,p)=\lambda K(x)$. The flow of the latter is 
$$
\varphi_\tH^\theta(x,\lambda,p) = \big(\varphi_K^{\theta\lambda}(x),\lambda,p-\theta K(x)\big),
$$
and its linearization is 
$$
d\varphi_\tH^\theta(x,\lambda,p)=\left(\begin{array}{ccc} 
d\varphi_K^{\theta\lambda}(x) & \theta X_K(\varphi_K^{\theta\lambda}(x)) & 0 \\
0 & 1 & 0 \\
-\theta dK(x) & 0 & 1 
\end{array}\right). 
$$

We shall compute the index under the following simplifying assumptions:
\begin{itemize} 
\item the level set $\Sigma :=K^{-1}(0)$ is regular and of restricted contact type. This means that the restriction to $\Sigma $ of the primitive of the symplectic form is a contact form, which we denote $\oalpha$.  There exists then a neighborhood $\cV$ of $\Sigma $ and a diffeomorphism $\cV\simeq [1-\eps,1+\eps]\times \Sigma $, $\eps>0$ which transforms the symplectic form into $d(r\oalpha)$, $r\in[1-\eps,1+\eps]$. 
\item the Hamiltonian $K$ has the form $K(r,\overline x)=k(r)$ on $\cV$, with $\ox$ denoting a point on $\Sigma $ and $k(1)=0$, $k'(1)\neq 0$. Then $X_K(r,\ox)=-k'(r)R(\ox)$, with $R$ the Reeb vector field on $\Sigma $ defined by $d\oalpha(R,\cdot)=0$ and $\oalpha(R)=1$. Thus
$$
\varphi_\tH^\theta(r,\ox,\lambda,p) = \big(r,\varphi_R^{-\theta\lambda k'(r)}(\ox),\lambda,p-\theta k(r)\big).
$$
\end{itemize}

Let $(\gamma,\lambda)\in\mathrm{Crit}(A_H)$ and choose a symplectic trivialization $\gamma^*TW\simeq \R^{2n}=\R^{2n-2}\oplus\R\oplus\R$ which maps the contact distribution $\xi$ to $\R^{2n-2}$, the vector field $\partial/\partial r$ to the constant vector $(0,1,0)$, and the Reeb vector field to the constant vector $(0,0,1)$. The tangent bundle $T(T^*\R)$ is in turn naturally trivialized as $T^*\R\times (\R\oplus \R)$. When read in these trivializations, the linearization $d\varphi_\tH^\theta$, $\theta\in[0,1]$  determines a path of symplectic matrices of the form 
$\left(\begin{array}{cc}\Phi(\theta) & 0 \\ 0 & M(\theta)\end{array}\right)$, where $\Phi(\theta)\in\mathrm{Sp}(2n-2)$ corresponds to $d\varphi_R^{-\theta\lambda k'(1)}(\gamma(0))\big|_\xi$ and $M(\theta)\in\cS_{1,1}\subset\mathrm{Sp}(4)$ has the form
$$
M(\theta)=\left(\begin{array}{cccc} 1 & 0 & 0 & 0 \\ 
\theta T & 1 & \theta A & 0 \\
0 & 0 & 1 & 0 \\
\theta A & 0 & 0 & 1  
\end{array}\right), \qquad T=-\lambda k''(1), \ A = -k'(1).
$$
The matrix that represents the crossing form is 
$$
-\widetilde J_0\dot M(\theta) M(\theta)^{-1}=\left(\begin{array}{cccc} T & 0 & A & 0 \\
0 & 0 & 0 & 0 \\
A & 0 & 0  & 0 \\
0 & 0 & 0 & 0 
\end{array}\right). 
$$
Using that $A\neq 0$ we obtain that $\ker(M(\theta)-\one)$ is equal to $\R^4$ if $\theta=0$, respectively to $0\oplus \R \oplus 0 \oplus \R$ if $\theta>0$. We can use Proposition~\ref{prop:RSstratum} to compute the Robbin-Salamon index of the path $M(\theta)$, $\theta\in[0,1]$ and we find 
$$
\mu_{RS}(M)= \frac 1 2 \mathrm{sign}\left(\begin{array}{cc} T & A \\ A & 0 \end{array}\right) = 0. 
$$

Given $(\gamma,\lambda)\in\mathrm{Crit}(A_H)$, we denote by $\og(\theta)=\varphi_R^{\theta|\lambda k'(1)|}(\gamma(0))$ the positively parametrized closed Reeb orbit that underlies $\gamma$, and denote $\mu(\og)$ its index. Then   
$$
\mu(\gamma,\lambda)=\mu_{RS}(\Phi(\theta),\ \theta\in[0,1])=\left\{\begin{array}{cl} 
\mathrm{sign}(-k'(1))\mu(\og), & \lambda>0, \\
0, & \lambda=0, \\
-\mathrm{sign}(-k'(1))\mu(\og), & \lambda<0.
\end{array}\right.
$$

This agrees with the Rabinowitz-Floer homology grading in~\cite{Cieliebak-Frauenfelder} up to a global shift of $\frac 1 2$.

\medskip

\noindent {\bf Acknowledgements.} F.B.: Partially supported by ERC Starting Grant StG-239781-ContactMath. A.O.: This material is based upon work supported by the National Science Foundation under agreement No. DMS-0635607. Any opinions, findings and conclusions or recommendations expressed in this material are those of the author(s) and do not necessarily reflect the views of the National Science Foundation. A.O. was also partially supported by ANR project ``Floer Power'' ANR-08-BLAN-0291-03 and ERC Starting Grant StG-259118-STEIN. A.O. is grateful to the organizers of the GESTA 2011 conference in Castro Urdiales for having given him the opportunity to lecture on $S^1$-equivariant symplectic homology.


\bibliographystyle{abbrv}
\bibliography{LS}

\enddocument


\section{Introduction}

\noindent {\sc Topological background.} Given an oriented fibration
$S^1\hookrightarrow M \stackrel \pi \to B$, the homology groups of the
base and total space are related by the {\bf Gysin exact sequence}
\begin{equation} \label{eq:Gysin}
\ldots \to H_k(M) \stackrel {\pi_*} \to H_k(B) \stackrel D \to H_{k-2}(B)
\to H_{k-1}(M) \to \ldots
\end{equation}
Here $D$ is the cap-product with the Euler class of the fibration and
is equal to the differential $d^2$ of the Leray-Serre spectral
sequence~\cite[Example~5.C]{McC}.

A particular case of the above construction is the following. Assume $M$
carries an $S^1$-action and define the $S^1$-equivariant homology
$H_*^{S^1}(M)$ by
$$
H_*^{S^1}(M) := H_*(M_{S^1}), \qquad M_{S^1}:= M\times_{S^1} ES^1,
$$
where $ES^1$ is a contractible space on which $S^1$ acts freely.
Since $S^1$ acts freely on $M\times ES^1$, the
projection $M \times ES^1 \to M_{S^1}$ is an $S^1$-fibration and the
exact sequence~\eqref{eq:Gysin} becomes
\begin{equation} \label{eq:equiGysin}
 \ldots \to H_k(M) \to H_k^{S^1}(M) \stackrel D \to H_{k-2}^{S^1}(M)
\to H_{k-1}(M) \to \ldots
\end{equation}
We call this the {\bf Gysin exact sequence for \boldmath$S^1$-equivariant
homology}.
Two relevant instances of this construction are the following:

(i) If the action of $S^1$ on $M$ is free
then $H_*^{S^1}(M)\simeq H_*(M/S^1)$ and the Gysin exact sequence for
$S^1$-equivariant homology is the Gysin exact sequence for the fibration
$S^1\hookrightarrow M \to M/S^1$.

(ii) We denote $BS^1:=ES^1/S^1$. Taking the model of $ES^1$ to be
$S^\infty := \lim_{N\to \infty} S^{2N+1}$, with $S^{2N+1}$ the unit
sphere in $\C^{N+1}$, we see that $BS^1\simeq \C P^\infty$. Now, if
$S^1$ acts trivially on $M$, then
$H_*^{S^1}(M)\simeq H_*(M)\otimes H_*(BS^1)$ and~\eqref{eq:equiGysin}
becomes
$$
\ldots \stackrel 0 \to H_k(M) \stackrel i \to \bigoplus _{m\ge 0}
H_{k-2m}(M) \stackrel p \to \bigoplus _{m\ge 1} H_{k-2m}(M)
\stackrel 0 \to H_{k-1}(M) \to \ldots
$$
Here $i$ and $p$ are the obvious inclusion and projection.

\medskip

\noindent {\sc Main results.} This paper is concerned with a Floer homology long exact sequence
of Gysin type. Let $(W,\om)$ be a symplectic manifold with contact
type boundary satisfying
\begin{equation} \label{eq:asph}
\int _{T^2} f^*\om =0 \quad \mbox{for all smooth } f:T^2\to W.
\end{equation}
Our main class of examples consists of exact symplectic manifolds. Let
$a$ be a free homotopy class of loops in $W$. One can define in this
situation {\bf symplectic homology groups} $SH_*^a(W)$ and {\bf
\boldmath$S^1$-equivariant symplectic homology groups} $SH_*^{a,S^1}(W)$, as
well as variants $SH_*^+(W)$, $SH_*^{+,S^1}(W)$ truncated in positive
values of the action functional when $a=0$. The original definition is
due to Viterbo~\cite{V} and we refer to \S\ref{sec:S1equivsymplhom} for
the details of the construction. Our first result is the following.

\begin{theorem} \label{thm:SGysin}
The symplectic homology groups fit into an exact sequence of
Gysin type (we allow $a=+$)
\begin{equation} \label{eq:SGysin}
 \ldots \to SH_k^a(W) \to SH_k^{a,S^1}(W) \stackrel D \to
SH_{k-2}^{a,S^1}(W) \to SH_{k-1}^a(W) \to \ldots
\end{equation}
\end{theorem}

As a matter of fact, we prove in~\cite{BO4} that the above Gysin exact sequence for $a=+$ is isomorphic to the long exact sequence of~\cite{BOcont}, relating $SH_*^+(W)$ with the linearized contact homology of the filled contact manifold $\p W$.

In the case $a=0$, the symplectic
homology groups
$$
SH_*(W):= SH_*^0(W), \qquad SH_*^{S^1}(W):= SH_*^{0,S^1}(W)
$$
also fit into tautological long exact sequences~\cite{V}
\begin{equation} \label{eq:taut}
\ldots \to SH_{*+1}^+(W) \to H_{*+n}(W,\p W) \to
SH_*(W) \to SH_*^+(W) \to \ldots \ ,
\end{equation}
\begin{equation} \label{eq:tautS1}
\ldots \to SH_{*+1}^{+,S^1}(W) \to H_{*+n}^{S^1}(W,\p W) \to
SH_*^{S^1}(W) \to SH_*^{+,S^1}(W) \to \ldots
\end{equation}
Here the $S^1$-equivariant homology of the pair $(W,\p W)$ is
considered with respect to the trivial action of $S^1$. Our next
result is that the Gysin exact sequence is compatible with these
tautological exact sequences.

\begin{theorem} \label{thm:grid}
There is a commutative diagram whose rows and columns are,
respectively, the tautological and Gysin exact sequences
\begin{equation} \label{eq:grid}
\xymatrix
@C=10pt
@R=10pt
{
& \vdots \ar[d] & \vdots \ar[d] & \vdots \ar[d] & \vdots \ar[d] & \\
\ldots \ar[r] & SH_{k+1}^+ \ar[r] \ar[d] & H_{k+n} \ar[r] \ar[d] &
SH_k \ar[r] \ar[d] & SH_k^+ \ar[r] \ar[d] & \ldots \\
\ldots \ar[r] & SH_{k+1}^{+,S^1} \ar[r] \ar[d] & H_{k+n}^{S^1} \ar[r] \ar[d] &
SH_k^{S^1} \ar[r] \ar[d] & SH_k^{+,S^1} \ar[r] \ar[d] & \ldots \\
\ldots \ar[r] & SH_{k-1}^{+,S^1} \ar[r] \ar[d] & H_{k+n-2}^{S^1}
\ar[r] \ar[d] & SH_{k-2}^{S^1} \ar[r] \ar[d] & SH_{k-2}^{+,S^1} \ar[r]
\ar[d] & \ldots \\
\ldots \ar[r] & SH_k^+ \ar[r] \ar[d] & H_{k+n-1} \ar[r] \ar[d] &
SH_{k-1} \ar[r] \ar[d] & SH_{k-1}^+ \ar[r] \ar[d] & \ldots \\
& \vdots & \vdots & \vdots & \vdots &
}
\end{equation}
\end{theorem}

\medskip

\noindent {\sc Examples.} We discuss the consequences of our main theorems for two important classes of symplectic manifolds.

\smallskip \noindent {\it Cotangent bundles.} Let $L$ be a closed oriented Riemannian manifold, and denote by $\Lambda L$ the free loop space of $L$. We consider  the symplectic manifold $W=DT^*L=\{p\in T^*L\, : \, \|p\|\le 1\}$. It was proved by Viterbo~\cite{Vcotangent} that
$$
SH_*(DT^*L)\simeq H_*(\Lambda L), \qquad SH_*^{S^1}(DT^*L)\simeq H_*^{S^1}(\Lambda L).
$$
Alternative proofs for the first isomorphism are due to Abbondandolo and Schwarz~\cite{AS1}, respectively to Salamon and Weber~\cite{SW}.
Our proof of Theorem~\ref{thm:SGysin} can be combined with the methods of~\cite{AS1} in order to prove that the long exact sequence~\eqref{eq:SGysin} is isomorphic to the Gysin sequence for $\Lambda L$, namely
\begin{equation} \label{eq:Gysinloop}
\xymatrix
@C=20pt
{
\dots \ar[r] & H_*(\Lambda L) \ar[r]^-E & H_*^{S^1}(\Lambda L) \ar[r]^-D &
H_{*-2}^{S^1}(\Lambda L) \ar[r]^-M & H_{*-1}(\Lambda L) \ar[r] & \dots
}
\end{equation}
Similarly, for $a=+$, we obtain the Gysin sequence of the pair $(\Lambda^0L,L)$, where $\Lambda^0L$ is the component of free contractible loops in $L$.

\smallskip \noindent {\it Subcritical Stein manifolds.} A subcritical Stein manifold is a complex manifold $(W,J)$, of complex dimension $n$, endowed with a pluri-subharmonic function $\phi:W\to \R$, satisfying the following conditions: (i) the boundary $\p W$ is a regular level set of $\phi$ along which $\vec\nabla \phi$ points outwards; (ii) $\phi$ is Morse and the index of all its critical points is strictly smaller than $n$. The complex structure $J$ is compatible with the natural symplectic form $\omega_\phi:=-d(d\phi\circ J)$.

It was proved by Cieliebak~\cite{C} that $SH_*(W)=0$. His proof can be adapted in a straightforward way in order to show that $SH_*^{S^1}(W)=0$. However, this fact follows also from Theorem~\ref{thm:SGysin} in the case $c_1(W)=0$.

\begin{corollary} \label{cor:subcrit}
Assume $W$ is a subcritical Stein manifold with $c_1(W)=0$. Then we have $SH_*^{S^1}(W)=0$ and there is an isomorphism of exact sequences
$$
{\scriptsize
\xymatrix
@C=15pt
{
\dots \ar[r] & SH_*^+(W) \ar[r] \ar[d]_\simeq & SH_*^{+,S^1}(W) \ar[r]^-D \ar[d]_\simeq &
SH_{*-2}^{+,S^1}(W) \ar[r] \ar[d]_\simeq & SH_{*-1}^+(W) \ar[r] \ar[d]_\simeq & \dots \\
\dots \ar[r]^-0 & H_{*+n-1}(W,\p W) \ar[r] & H_{*+n-1}^{S^1}(W,\p W) \ar[r] & H_{*+n-3}^{S^1}(W,\p W) \ar[r]^-0 & H_{*+n-2}(W,\p W) \ar[r] & \dots
}
}
$$
\end{corollary}

\begin{proof}
Applying Theorem~\ref{thm:SGysin} we obtain that $SH_k^{S^1}(W)\simeq SH_{k-2}^{S^1}(W)$ for all $k\in \Z$.
It was proved by M.-L.~Yau~\cite{Y} that one can choose the plurisubharmonic function $\phi$ so that the Conley-Zehnder indices of all closed characteristics on $\p W$ are positive. It follows from the definition of $S^1$-equivariant symplectic homology that the underlying chain complex is zero if the degree is small enough (one can use "split" Hamiltonians as in the proof of Lemma~\ref{lem:minus}). Reasoning by induction, it follows that $SH_*^{S^1}(W)=0$. The isomorphism of exact sequences follows immediately from Theorem~\ref{thm:grid}, since the columns involving $SH_*$ and $SH_*^{S^1}$ vanish identically.
\end{proof}

\medskip

\noindent {\sc Algebraic Weinstein conjecture.} Following Viterbo~\cite{V}, we say that $W$ satisfies the \emph{Strong Algebraic Weinstein Conjecture (SAWC)} if the map
$$
H_{2n}(W,\p W)\to SH_n(W)
$$
vanishes. Let $\mu_{2n}\in H_{2n}(W,\p W)$ be the fundamental class and $u_k$ be a generator of $H_{2k}(BS^1)$, $k\ge 0$.
We say that $W$ satisfies the \emph{Strong Equivariant Algebraic Weinstein Conjecture (EWC)} if, for all $k\ge 0$, the element
$\mu_{2n}\otimes u_k$ lies in the kernel of the map
$$
H_{2n+2k}^{S^1}(W,\p W)\to SH_{n+2k}^{S^1}(W).
$$
Our next result clarifies the relationship between $SAWC$ and $EWC$, which are the two key notions in Viterbo's fundamental paper~\cite{V}.
\begin{corollary} \label{cor:WC} $SAWC\Longrightarrow EWC$.
\end{corollary}

\begin{proof}
We first note that $SAWC$ is equivalent to the vanishing of $SH_*(W)$. This follows from the fact that $SH_*(W)$ is a ring with unit~\cite{McLean}, and the unit is the image of the fundamental class $\mu_{2n}$ under the map $H_{2n}(W,\p W)\to SH_n(W)$~\cite{Seidel-biased}.

We now consider the top middle square in the commutative diagram~\eqref{eq:grid} of Theorem~\ref{thm:grid}. Since $\mu_{2n}\otimes u_0$ is the image of $\mu_{2n}$ under the injection $H_{2n}\to H_{2n}^{S^1}$, it follows that $\mu_{2n}\otimes u_0$ is in the kernel of $H_{2n}^{S^1}\to SH_n^{S^1}$. We now prove by induction that $\mu_{2n}\otimes u_k$ is in the kernel of $H_{2n+2k}^{S^1}\to SH_{n+2k}^{S^1}$. This follows from the middle square in the commutative diagram~\eqref{eq:grid}, using that $\mu_{2n}\otimes u_{k+1}$ is sent to $\mu_{2n}\otimes u_k$ by the map $H_{2n+2k+2}^{S^1}\to H_{2n+2k}^{S^1}$, and the fact that $SH_{n+2k+2}^{S^1}\to SH_{n+2k}^{S^1}$ is an isomorphism.
\end{proof}

\begin{remark}
 The same argument as above shows that, under the assumption $SAWC$, the maps $H_{k+n}^{S^1}\to SH_k^{S^1}$ vanish for all $k\in\Z$.
\end{remark}

\medskip

\noindent {\sc The parametrized Robbin-Salamon index.} Given a smooth manifold $X$ with an $S^1$-action, the $S^1$-equivariant homology $H_*^{S^1}(X)$ can be realized as $S^1$-invariant Morse homology on $X\times ES^1$, which in turn can be approximated by $X\times S^{2N+1}$, $N\to \infty$. In analogy, $S^1$-equivariant symplectic homology is defined as an $S^1$-invariant Floer theory for a parametrized action functional $\cA_H:C^\infty(S^1,W)\times S^{2N+1}\to \R$ which, on contractible loops, is of the form
$$
\cA_H(\gamma,\lambda):=-\int_{D^2}\sigma^*\omega - \int_{S^1}H(\theta,\gamma(\theta),\lambda)\,d\theta,
$$
where $\sigma:D^2\to W$ is a capping disc for $\gamma$. Here $H:S^1\times W\times
S^{2N+1}\to\R$ is an $S^1$-invariant family of Hamiltonians, meaning that $H(\theta+\tau,x,\tau\lambda)=H(\theta,x,\lambda)$ for all $\tau\in S^1$. This condition ensures that $\cA_H$ is invariant with respect to the diagonal action of $S^1$ on $C^\infty(S^1,W)\times S^{2N+1}$, given by $\tau\cdot(\gamma,\lambda):=(\gamma(\cdot-\tau),\tau\lambda)$.

The corresponding chain complex is generated by $S^1$-orbits of critical points of $\cA_H$, which are pairs $(\gamma,\lambda)$ such that $\gamma$ is a $1$-periodic orbit of the Hamiltonian vector field of $H(\cdot,\cdot,\lambda)$, and $\int_{S^1}\frac {\p H} {\p \lambda}
(\theta,\gamma(\theta),\lambda)\, d\theta=0$ (see~\S\ref{sec:action}). In order to associate a grading to these generators, we define in~\S\ref{sec:RS} a parametrized version of the Robbin-Salamon index.

The construction is valid for an arbitrary finite-dimensional smooth parameter space $\Lambda$. Given a parametrized Hamiltonian $H:S^1\times W\times
\Lambda\to\R$, we extend it to $\tH:S^1\times  W\times
T^*\Lambda\to\R$ by the formula
$$
\tH(t,x,(\lambda,p)):=H(t,x,\lambda).
$$
We then have $X_\tH=X_{H_\lambda}+\frac {\p H} {\p \lambda}\frac \p
{\p p}$, with $H_\lambda:=H(\cdot,\cdot,\lambda)$. A
$1$-periodic orbit of $X_\tH$ has the form
$(\gamma(\cdot),p(\cdot),\lambda)$, with $\gamma$ a $1$-periodic orbit of
$X_{H_\lambda}$ and $p(t)=p(0)+\int_0^t\frac {\p H} {\p \lambda}
(\theta,\gamma(\theta),\lambda)\, d\theta$. The closing condition
$p(1)=p(0)$ is equivalent to the condition
$\int_0^1\frac {\p H} {\p \lambda}
(\theta,\gamma(\theta),\lambda)\, d\theta=0$, while $p(0)\in T^*_\lambda
\Lambda$ can be chosen arbitrarily. Thus (nondegenerate) critical
points of the parametrized action functional for $H$ are in one-to-one
bijective correspondence with (Morse-Bott) families of $1$-periodic
orbits of $\tH$, of dimension
$\dim\,\Lambda=\dim\,T^*_\lambda\Lambda$.

\begin{definition} \label{defi:RS} The \emph{parametrized Robbin-Salamon
  index} $\mu(\gamma,\lambda)$ of a critical point
$(\gamma(\cdot),\lambda)$ for the parametrized Hamiltonian $H$ is
the Robbin-Salamon index~\cite{RS} of one of the corresponding $1$-periodic orbits
$(\gamma(\cdot),p(\cdot),\lambda)$ of $\tH$.
\end{definition}

The construction given in~\S\ref{sec:RS} is phrased directly in terms of the Hamiltonian $H$, rather than in terms of $\widetilde H$ (see~\eqref{eq:defiRS}). The properties of the parametrized Robbin-Salamon index are proved in Appendix~\ref{app:RS}, and they amount to the study of paths of symplectic matrices of a special form. The dimension of the moduli spaces of trajectories connecting a pair of critical points is given by the difference of the indices of these critical points (see Theorem~\ref{thm:index}).

\medskip

\noindent {\sc Ramifications.} We now present several directions of investigation which are related to the present paper.

\smallskip

\noindent {\it Algebraic structures.} The Gysin exact sequence~\eqref{eq:SGysin} can be used to define algebraic operations
in ($S^1$-equivariant) symplectic homology.

As already mentioned in the proof of Corollary~\ref{cor:WC}, symplectic homology $SH_*(W)$ is a unitary ring, with the
pair-of-pants product. This is described by
Seidel~\cite{Seidel-biased}, and was used in a crucial way by McLean~\cite{McLean}
in his construction of exotic affine $\R^{2n}$'s. We denote the pair of pants
product by
$$
\bullet:SH_k(W)\otimes SH_\ell(W)\longrightarrow SH_{k+\ell-n}(W).
$$

Let us write the Gysin exact sequence~\eqref{eq:SGysin} as
$$
\xymatrix
@C=20pt
{
\dots \ar[r] & SH_*(W) \ar[r]^E & SH_*^{S^1}(W) \ar[r]^D &
SH_{*-2}^{S^1}(W) \ar[r]^M & SH_{*-1}(W) \ar[r] & \dots
}
$$
The notation is motivated by the isomorphism with the exact sequence~\eqref{eq:Gysinloop} in the case $W=DT^*L$.
The letters $M$ and $E$ stand for ``mark'' and ``erase'', in the terminology of
Chas and Sullivan~\cite{CS}. It was proved by Abbondandolo and Schwarz~\cite{AS2} that, in the case $W=DT^*L$, the pair-of-pants product
is identified with the Chas-Sullivan loop product~\cite{CS}.

Inspired by Chas and Sullivan~\cite{CS}, we formulate the following definitions and claims, which we will prove in a forthcoming paper.
\begin{itemize}
\item[---] The map
$$
\Delta:SH_*(W)\to SH_{*+1}(W), \qquad \Delta:=M\circ E
$$
is a \emph{Batalin-Vilkovisky (BV) operator}, in the sense that
$\Delta^2=0$, and
$$
\{\cdot,\cdot\} : SH_k(W)\otimes SH_\ell(W)\to SH_{k+\ell-n+1}(W),
$$
$$
\{a,b\} \  := \  \pm \, \Delta(a\bullet b) \pm  a \bullet \Delta(b) \pm
b\bullet \Delta(a)
$$
is a bracket on $SH_*(W)$ (called \emph{the loop bracket}).
\item[---] The map
$$
[\cdot,\cdot]:SH_k^{S^1}(W)\otimes SH_\ell^{S^1}(W)\to SH_{k+\ell-n+2}^{S^1}(W)
$$
$$
[a,b]:=\pm\, E(M(a)\bullet M(b))
$$
is a bracket on $SH_*^{S^1}(W)$ (called \emph{the string bracket}).
\end{itemize}
We give a chain-level description of $\Delta$ in Remark~\ref{rmk:Delta}. The above claims are analogous to Theorems~4.7, 5.4, and 6.1 of~\cite{CS}. The string bracket can be further generalized as follows. Any operation
$$
\widetilde \sigma:SH_*^{\otimes k}\to SH_*, k\ge 2
$$
yields an operation
$$
\sigma:=E\circ \widetilde \sigma \circ M^{\otimes k}:(SH_*^{S^1})^{\otimes k}\to SH_*^{S^1}.
$$
One particular case is $\widetilde \sigma:=\bullet ^{\otimes k-1}$, $k\ge 2$, which yields higher-order operations on $SH_*^{S^1}$ analogous to the ones of~\cite[Theorem~6.2]{CS}.

The range of applications of such operations depends on their explicit
knowledge in particular situations (e.g. cotangent
bundles). However, the Chas-Sullivan string
operations are only beginning to be
understood by topologists (see the work of Felix, Thomas, and
Vigu\'e-Poirrier~\cite{Felix-Thomas-Vigue,Felix-Thomas-Vigue-2}).

 It should also be possible to describe these operations directly in terms of
holomorphic curves. Such a construction is sketched by Seidel in~\cite{Seidel-biased}.

\smallskip

\noindent {\it Relation to Hochschild and cyclic homology.}
Paul Seidel has conjectured in~\cite{Se} that,
given an exact Lefschetz fibration $E\to D$ over the disc, the
symplectic homology of $E$ is isomorphic to the Hochschild homology of
a certain $A_\infty$-category $\cC$ built from the vanishing cycles of
$E$:
$$
SH_*(E) \simeq HH_*(\cC).
$$
It is implicit in~\cite{Se} that there is an equivariant version of
this conjectural isomorphism, namely that the $S^1$-equivariant symplectic
homology of $E$ is isomorphic to the cyclic homology of $\cC$:
$$
SH_*^{S^1}(E)\simeq HC_*(\cC).
$$
On the other hand, Hochschild and cyclic homology are related by the
Connes exact sequence
\begin{equation} \label{eq:Connes}
 \ldots \to HH_k(\cC) \to HC_k(\cC) \stackrel D \to HC_{k-2}(\cC) \to
HH_{k-1}(\cC) \to \ldots
\end{equation}
We conjecture that the two previous isomorphisms are such that the
Gysin exact sequence~\eqref{eq:SGysin} and the Connes exact
sequence~\eqref{eq:Connes} are isomorphic. This fits with the general
philosophy that the Gysin exact sequence for $S^1$-equivariant
homology of certain topological spaces is isomorphic to
the Connes exact sequence of suitable algebras (a good reference is
Loday's book~\cite{Lo}, in particular~\cite[Theorem~7.2.3]{Lo}).

\smallskip

\noindent {\it Relation to Givental's point of view.} Given a closed symplectic manifold $X$, Givental defined in~\cite{Gi} a $D$-module structure on
$H^*(X;\C)\otimes \Lambda_{Nov}\otimes \C[\hbar]$, where $\Lambda_{Nov}$ is a suitable Novikov ring and $\hbar$ is the generator of $H^*(BS^1)$. He interprets this
as being the $S^1$-equivariant Floer cohomology of $X$. Our construction of $S^1$-equivariant Floer homology provides an interpretation of the underlying homology group as the homology of a Floer-type complex. We expect that the $D$-module structure can also be defined within our setup.

\medskip

 \noindent {\sc Structure of the paper.} In~\S\ref{sec:symplhom} we briefly recall the construction of symplectic homology. In~\S\ref{sec:param}
 we introduce a new variant of it, which we call ``parametrized symplectic homology''. It corresponds to families of Hamiltonians, indexed by a finite dimensional parameter space.
 In order to assign a grading to the generators of the underlying chain complex, under the sole nondegeneracy assumption and regardless of the specific form of the Hamiltonian, we define the ``parametrized Robbin-Salamon index" in~\S\ref{sec:RS}. Section~\ref{sec:S1} is devoted to the $S^1$-equivariant theory. We recall in~\S\ref{sec:S1equivhom} the Borel construction and its interpretation in Morse homology. We define $S^1$-equivariant symplectic homology in~\S\ref{sec:S1equivsymplhom}, following Viterbo~\cite[\S5]{V}. We prove Theorems~\ref{thm:SGysin} and~\ref{thm:grid} in~\S\ref{sec:MBparam}, using a Morse-Bott construction and a spectral sequence argument. In~\S\ref{sec:continuation} we use similar techniques to study continuation maps. We prove in Appendix~\ref{app:RS} some important properties of the parametrized Robbin-Salamon index.

 \medskip

 \noindent {\sc Acknowledgements.} The authors were partially supported by the Minist\`ere Belge
des Affaires \'etrang\`eres and the Minist\`ere
Fran\c{c}ais des Affaires \'etrang\`eres et europ\'eennes, through the
programme PHC--Tournesol Fran\c{c}ais. F.B. was partially supported by the Fonds National de la Recherche Scientifique (Belgium). Both authors were partially
supported by ANR project ``Floer Power'' ANR-08-BLAN-0291-03 (France).


\section{Symplectic homology} \label{sec:symplhom}

We briefly recall in this section the definition of symplectic
homology, and we refer to~\cite{BOauto} for full details. In the sequel
$(W,\om)$ denotes a compact symplectic manifold with
contact type boundary $M:=\p W$. This means that there exists a vector
field $X$ defined in a neighbourhood of $M$, transverse and pointing
outwards along $M$, and such that
$$
\cL _X \om = \om.
$$
Such an $X$ is called a {\bf Liouville vector field}. The $1$-form
$\alpha:=(\iota_X\om)|_M$ is a contact form on $M$ and is called the
{\bf Liouville \boldmath$1$-form}. We denote by
$\xi:=\ker \alpha$ the contact structure defined by $\alpha$, and
we note that the isotopy class of $\xi$ is uniquely determined by
$\om$. The {\bf Reeb vector field} $R_\alpha$ is defined by
the conditions $\ker \, \om|_M = \langle R_\alpha \rangle$ and
$\alpha(R_\alpha)=1$. We denote by $\phi_\alpha$ the flow of
$R_\alpha$. The {\bf action spectrum} of $(M,\alpha)$ is
defined by
$$
\textrm{Spec}(M,\alpha) := \{ T \in \R^+\, | \, \textrm{ there is a
   closed } R_\alpha\textrm{-orbit of period } T\}.
$$

Let $\phi$ be the flow of $X$. We parametrize a neighbourhood $U$ of $M$ by
$$
G: M \times [-\delta, 0] \to U, \qquad (p,t) \mapsto \phi^t(p).
$$
Then $d(e^t\alpha)$ is a symplectic form on $M\times \R^+$ and
$G$ satisfies $G^*\om = d(e^t \alpha)$.
We denote by
$$
\widehat W : = W \ \bigcup _{G} \ M\times \R^+
$$
the {\bf symplectic completion of \boldmath$W$} and endow it with the
symplectic form
$$
\widehat \om : =
\left\{\begin{array}{ll}
\om & \textrm{ on } W, \\
d(e^t \alpha) & \textrm{ on } M\times \R^+.
\end{array} \right.
$$

Given a time-dependent Hamiltonian $H :S^1\times \widehat W \to \R$
we define the
{\bf Hamiltonian vector field} $X^\theta_H$ by
$$
\widehat \om (X^\theta_H,\cdot) = d H_\theta, \qquad \theta\in S^1 = \R/\Z,
$$
where $H_\theta:=H(\theta,\cdot)$. We denote by $\phi_H$ the flow of
$X_H^\theta$, defined by $\phi_H^0=\textrm{Id}$ and
$$
  \frac d {d\theta} \phi_H^\theta (x) = X^\theta_H(\phi_H^\theta(x)),
\qquad \theta\in \R.
$$
We denote by $\cP(H)$ the set of $1$-periodic orbits of $X^\theta_H$,
and we denote by $\cP^a(H)\subset \cP(H)$ the set of $1$-periodic
orbits in the free homotopy class $a$.

We define the class $\cH$ of {\bf admissible Hamiltonians} to consist of
smooth functions $H:S^1\times \widehat W\to \R$ satisfying the following
conditions:
\begin{itemize}
\item $H<0$ on $W$;
\item there exists $t_0\ge 0$ such that $H(\theta,p,t)=\beta e^t
  +\beta'$ for $t\ge t_0$, with $0<\beta\notin
  \mathrm{Spec}(M,\alpha)$ and $\beta'\in\R$.
\end{itemize}

We denote by $\cH_{\textrm{reg}}\subset \cH$ the dense set of Hamiltonians $H$
  such that all elements of $\cP(H)$ are nondegenerate, i.e. the Poincar\'e
  return map has no eigenvalues equal to $1$.
Let $a$ be a free homotopy class of loops in $W$.
The {\bf symplectic homology groups} of $(W,\om)$ are defined by
\begin{equation*}
   SH_*^a(W,\om) := \lim _{\substack{ \longrightarrow \\ H\in \cH_{\textrm{reg}}}
    } SH_*^a(H,J).
\end{equation*}
Here $J$ is an almost complex structure on $\widehat W$ which is
compatible with $\widehat \om$, convex and invariant under translation
in the $t$-variable outside a compact set, and regular for
$H$ (in particular one must allow $J$ to depend on $\theta$).
We denote by $SH_*^a(H,J)$ the Floer homology groups of the pair
$(H,J)$ in the free homotopy class $a$ and with coefficients in the
Novikov ring $\Lambda_\om$. We assume throughout this paper that $W$
satisfies condition~\eqref{eq:asph}, so that the energy of a Floer
trajectory does not depend on its homology class, but only on its
endpoints. We refer to \cite{BOauto} for the details of the construction
and in particular for the definition of the coefficient ring
$\Lambda_\om$. Throughout this paper the Novikov ring is understood to
be defined over $\Q$.

For the trivial homotopy class $a=0$ we denote
the symplectic homology groups by $SH_*(W,\om)$. The {\bf
reduced Hamiltonian action functional} is
$$
\cA_H^0 : C^\infty_{\textrm{contr}}(S^1,\widehat W) \to \R,
$$
$$
\cA_H^0(\gamma) := -\int_{D^2} \sigma^*\widehat \om - \int_{S^1}
H(\theta,\gamma(\theta)) \, d\theta.
$$
Here $C^\infty_{\textrm{contr}}(S^1,\widehat W)$ denotes the space of
smooth contractible loops in $\widehat W$ and $\sigma:D^2\to \widehat
W$ is a smooth extension of $\gamma$. Note that $\cA_H^0$ is
well-defined thanks to condition~(\ref{eq:asph}) and
is decreasing along Floer trajectories.

We now consider a special cofinal class of Hamiltonians $\cH'\subset
\cH$, consisting of elements $H\in \cH'$ which satisfy the following
conditions:
\begin{itemize}
\item there exists $t_0\ge 0$ such that $H(\theta,p,t)=\beta e^t
  +\beta'$ for $t\ge t_0$, with $0<\beta\notin
  \mathrm{Spec}(M,\alpha)$ and $\beta'\in\R$;
\item $H<0$ and $C^2$-small on $W$;
\item $H(\theta,p,t)$ is $C^2$-close to an increasing function of $t$
  on $S^1\times M \times [0,t_0]$.
\end{itemize}
The last condition implies that, in the region $M\times [0,t_0]$, each
$1$-periodic orbit of $H$ is $C^1$-close to a closed characteristic on
some level $M\times\{t\}$.

Given $H\in\cH'_\reg:=\cH_\reg\cap \cH'$, a regular almost complex
structure $J$, and a choice of $\epsilon>0$ small enough, we define
the chain complexes
\begin{equation} \label{eq:SC-}
SC_*^-(H,J) := \bigoplus _{\substack{ \gamma \in
     \cP^0(H) \\ \cA_H^0(\gamma) \le \epsilon }} \Lambda_\om
\langle \gamma \rangle \ \subset SC_*(H,J)
\end{equation}
and
$$
SC_*^+(H,J) := SC_*(H,J) / SC_*^-(H,J).
$$
The differential on $SC_*^\pm(H,J)$ is induced by $\p$. The groups
$$
SH_*^\pm(H,J) := H_*(SC_*^\pm(H),\p)
$$
do not depend on $J$, nor on $\epsilon$, and we define
$$
SH_*^\pm(W,\om) := \lim _{\substack{ \to \\ H\in \cH'_{\textrm{reg}}}} SH_*^\pm(H).
$$
We call $SH_*^+(W,\om)$ the {\bf positive symplectic homology group}
of $(W,\om)$.

\begin{remark} {\rm
Condition~\eqref{eq:asph} can be
replaced in the case of contractible orbits by the weaker {\bf
symplectic asphericity} condition $\langle \om,\pi_2(W)\rangle =0$.
}
\end{remark}

Let us assume now that $W$ has {\bf positive contact
   type} boundary~\cite[\S5.4]{O1}. This means that every positively oriented
closed characteristic $\gamma$ on $M$ which is contractible in $W$ has
positive action $\cA_\om(\gamma)$ bounded away from zero, where
$$
\cA_\om(\gamma) := \int_{D^2} \sigma^*\om
$$
for some extension $\sigma:D^2\to W$ of $\gamma$. This condition is
automatically satisfied if the boundary $M$ is of restricted contact
type, i.e. the vector field $X$ is globally defined on $W$. Under the
positive contact type assumption we have~\cite{V}
$$
SH_*^-(W,\om) = H_{*+n}(W,\p W;\Lambda_\om), \qquad n=\frac 1 2 \dim
\, W,
$$
and the short exact sequence of complexes $SC_*^-(H) \to SC_*(H)
\to SC_*^+(H)$ induces the tautological long exact
sequence~\eqref{eq:taut}.


\section{Parametrized symplectic homology} \label{sec:param}

We introduce in this section a new variant of Floer
homology, which we call ``parametrized Floer homology''.
In the sequel $\Lambda$ is a finite dimensional closed
manifold of dimension $m$, which we call ``parameter space''. The
elements of $\Lambda$ are denoted by $\lambda$. When the
parameter space is $S^{2N+1}$, the parametrized symplectic homology
groups will be the abutment of the spectral sequence which gives rise
to the Gysin exact sequence~\eqref{eq:SGysin}.

\subsection{The parametrized Floer equation} \label{sec:action}
For each free homotopy class $a$ in $W$, we fix a reference loop
$l_a:S^1\to \widehat W$ such that $[l_a]=a$. If $a$ is the trivial
homotopy class, we choose $l_a$ to be a constant loop.
Recall that free homotopy classes of loops in $\widehat W$ are in
one-to-one correspondence with conjugacy classes in $\pi_1(\widehat
W)$. As a consequence, the inverse $a^{-1}$ of a free homotopy class
is well-defined. We require that $l_{a^{-1}}$ coincides with the loop
$l_a$ with the opposite orientation.

We define the set $\cH_\Lambda$ of
{\bf admissible Hamiltonian families} to consist of elements
$H\in C^\infty(S^1\times \widehat W\times \Lambda,\R)$ which satisfy
the following conditions:
\begin{itemize}
\item $H<0$ on $S^1\times W\times \Lambda$;
\item there exists $t_0\ge 0$ such that $H(\theta,p,t,\lambda)=\beta
  e^t +\beta'(\lambda)$ for $t\ge t_0$, with $0<\beta\notin
  \mathrm{Spec}(M,\alpha)$ and $\beta'\in C^\infty(\Lambda,\R)$.
\end{itemize}

Let $H:S^1 \times \widehat W \times \Lambda \to
\R$ be an admissible Hamiltonian family denoted by
$H(\theta,x,\lambda)=H_\lambda(\theta,x)$. This defines a family of
action functionals
$$
\cA : C^\infty(S^1,\widehat W)\times \Lambda \to \R,
$$
$$
\cA(\gamma,\lambda) = \cA_\lambda(\gamma) := -\int_{[0,1]\times S^1} \sigma^*
\om - \int_{S^1} H_\lambda(\theta,\gamma(\theta)) d\theta,
$$
where $\sigma:[0,1]\times S^1\to \widehat W$ is a smooth homotopy from
$l_{[\gamma]}$ to $\gamma$. The functional $\cA$ is well-defined due
to our standing assumption~\eqref{eq:asph}.

The differential of $\cA$ is given by
\begin{equation} \label{eq:dA}
d\cA(\gamma,\lambda) \cdot (\zeta,\ell)=
\int_{S^1}\om(\dot\gamma(\theta)-X_{H_\lambda}(\gamma(\theta)),\zeta(\theta))
d\theta
-
\int_{S^1} \frac {\p H} {\p \lambda} (\theta,\gamma(\theta),\lambda) d\theta
\cdot \ell
\end{equation}
and therefore $(\gamma,\lambda)$ is a critical point of $\cA$ if and
only if
\begin{equation} \label{eq:periodicpar}
 \gamma\in\cP(H_\lambda) \quad \mbox{and} \quad
 \int_{S^1} \frac {\p H} {\p \lambda} (\theta,\gamma(\theta),\lambda)\,
d\theta =0.
\end{equation}
We denote by $\cP(H)$ the set of critical points of $\cA$ consisting
of pairs $(\gamma,\lambda)$ satisfying~\eqref{eq:periodicpar}. We
denote by $\cP^a(H)$ the set of pairs $(\gamma,\lambda)\in\cP(H)$ such
that $\gamma$ lies in the free homotopy class $a$.

\begin{remark} {\rm
 Equation~\eqref{eq:periodicpar} can be interpreted as follows. Every
loop $\gamma:S^1\to\widehat W$ determines a function
\begin{equation} \label{eq:Fgamma}
F_\gamma:\Lambda \to \R, \qquad \lambda \mapsto \int_{S^1}
H(\theta,\gamma(\theta),\lambda) \, d\theta.
\end{equation}
A pair $(\gamma,\lambda)$ belongs therefore to $\cP(H)$ if and only if
$$
\gamma\in \cP(H_\lambda) \quad \mbox{ and } \quad \lambda\in
\textrm{Crit}(F_\gamma).
$$
}
\end{remark}

Let $J=(J_\lambda^\theta)$, $\lambda\in\Lambda$, $\theta\in S^1$ be a
family of $\theta$-dependent compatible
almost complex structures on $\widehat W$
which, at infinity, are invariant under translations in the
$t$-variable and satisfy the relations
\begin{equation} \label{eq:standardJ}
J_\lambda^\theta \xi=\xi, \qquad J_\lambda^\theta (\frac \partial
{\partial t}) =R_\alpha.
\end{equation}
Such an {\bf admissible family of almost complex
  structures} $J$ induces a
  family of $L^2$-metrics on the space $C^\infty(S^1,\widehat W)$, parametrized
  by $\Lambda$ and defined by
$$
\langle \zeta,\eta\rangle_\lambda := \int_{S^1}
\om(\zeta(\theta),J_\lambda^\theta\eta(\theta)) d\theta, \quad \zeta,\eta\in
T_\gamma C^\infty(S^1,\widehat
W)=\Gamma(\gamma^*T\widehat W).
$$
Such a metric can be coupled with any metric $g$ on $\Lambda$ and
gives rise to a metric on
$C^\infty(S^1,\widehat W)\times \Lambda$ acting
at a point $(\gamma,\lambda)$ by
$$
\langle(\zeta,\ell), (\eta,k)\rangle_{J,g}:= \langle \zeta,\eta\rangle_\lambda +
g(\ell,k), \qquad (\zeta,\ell),(\eta,k)\in \Gamma(\gamma^*T\widehat
W)\oplus T_\lambda\Lambda.
$$
We denote by $\cJ_\Lambda$ the set of pairs $(J,g)$ consisting of an
admissible almost complex structure $J$ on $\widehat W$ and of a
Riemannian metric $g$ on $\Lambda$.

The {\bf parametrized Floer equation} is the gradient equation for
$\cA$ with respect to such a metric
$\langle\cdot,\cdot\rangle_{J,g}$. More precisely, given
$\op:=(\og,\olambda),\up:=(\ug,\ulambda)\in \cP(H)$
we denote by
$$
\widehat \cM(\op,\up;H,J,g)
$$
the {\bf space of parametrized Floer trajectories}, consisting of
pairs $(u,\lambda)$ with
$$
u:\R\times S^1 \to \widehat W, \qquad \lambda:\R\to \Lambda,
$$
satisfying
\begin{eqnarray}
\label{eq:Floer1par}
 \p_s u + J_{\lambda(s)}^\theta (\p_\theta u -
X_{H_{\lambda(s)}}^\theta (u)) & = & 0, \\
\label{eq:Floer2par}
 \dot \lambda (s) - \int_{S^1} \vec \nabla_\lambda
H(\theta,u(s,\theta),\lambda(s)) d\theta & = & 0,
\end{eqnarray}
and
\begin{equation} \label{eq:asymptoticpar}
 \lim_{s\to -\infty} (u(s,\cdot),\lambda(s)) = (\og,\olambda), \quad
 \lim_{s\to +\infty} (u(s,\cdot),\lambda(s)) = (\ug,\ulambda).
\end{equation}
Here and in the sequel we use the notation $\vec \nabla$ for a
gradient vector field, whereas $\nabla$ will denote a
covariant derivative.

\begin{remark}{\rm Equation~\eqref{eq:Floer2par} is equivalent to
\begin{equation} \label{eq:Floer2parbis}
 \dot \lambda(s) - \vec \nabla F_{u(s,\cdot)}(\lambda(s))=0,
\end{equation}
where $F_{u(s,\cdot)}$ is defined by~\eqref{eq:Fgamma}. Thus, the
parametrized Floer equation is a system involving a Floer equation and
a finite-dimensional gradient equation.
}
\end{remark}

The additive group $\R$ acts on $\widehat \cM(\op,\up;H,J,g)$
by reparametrization in the $s$-variable and we denote by
$$
\cM(\op,\up;H,J,g) := \widehat \cM(\op,\up;H,J,g)/\R
$$
the {\bf moduli space of parametrized Floer trajectories}.

Let us fix $p\ge 2$. The linearization of the
equations~(\ref{eq:Floer1par}-\ref{eq:Floer2par}) gives rise to the
operator
$$
D_{(u,\lambda)} : W^{1,p}(u^*T\widehat W) \oplus
W^{1,p}(\lambda^*T\Lambda) \to
L^p(u^*T\widehat W) \oplus L^p(\lambda^*T\Lambda),
$$
$$
D_{(u,\lambda)} (\zeta,\ell) :=
\left(\begin{array}{c}
D_u\zeta + (D_\lambda J\cdot \ell)(\p_\theta u - X_{H_\lambda}(u)) -
J_\lambda (D_\lambda X_{H_\lambda}\cdot \ell) \\
\nabla_s \ell - \nabla_\ell \int_{S^1} \vec \nabla_\lambda H
(\theta,u,\lambda) d\theta
- \int_{S^1} \nabla_\zeta \vec \nabla_\lambda H(\theta,u,\lambda) d\theta
\end{array}\right),
$$
where
$$
D_u : W^{1,p}(u^*T\widehat W) \to L^p(u^*T\widehat W)
$$
is the usual Floer operator given by
$$
D_u\zeta := \nabla_s \zeta + J_\lambda \nabla_\theta \zeta -
J_\lambda \nabla_\zeta X_{H_\lambda} + \nabla_\zeta J_\lambda (\p_\theta u -
X_{H_\lambda}).
$$

The Hessian of $\cA$ at a critical point $p=(\gamma,\lambda)$ is given
by the formula
\begin{eqnarray} \label{eq:d2A}
\lefteqn{d^2\cA(\gamma,\lambda)\big((\zeta,\ell),(\eta,k)\big)} \\
&=& \int_{S^1} \omega(\nabla_\theta\eta-\nabla_\eta
X_{H_\lambda},\zeta) d\theta - \int_{S^1}\eta(\frac {\partial H}
{\partial \lambda}\cdot \ell) d\theta \nonumber \\
&& - \ \int_{S^1} k(dH_\lambda\cdot \zeta) d\theta - \int_{S^1} \frac
{\partial^2 H} {\partial \lambda^2}(\ell,k) d\theta \nonumber \\
&=& d^2\cA_{H_\lambda}(\gamma)(\zeta,\eta) - \int_{S^1}\eta(\frac {\partial H}
{\partial \lambda}\cdot \ell) d\theta
- \int_{S^1} k(dH_\lambda\cdot \zeta) d\theta
- d^2 F_\gamma(\lambda)(\ell,k). \nonumber
\end{eqnarray}

We define the asymptotic operator at a critical point
$(\gamma,\lambda)$ by
$$
D_{(\gamma,\lambda)} : H^1(S^1,\gamma^*T\widehat W) \times T_\lambda
\Lambda \to L^2(S^1,\gamma^*T\widehat W) \times T_\lambda
\Lambda,
$$
\begin{equation}  \label{eq:Dasy}
D_{(\gamma,\lambda)}(\zeta,\ell) = \left(\begin{array}{c}
J_\lambda(\nabla_\theta \zeta - \nabla_\zeta X_{H_\lambda} -
(D_\lambda X_{H_\lambda})\cdot \ell) \\
-\int_{S^1} \nabla_\zeta \frac {\partial H} {\partial \lambda} d\theta
- \int_{S^1} \nabla_\ell \frac {\partial H} {\partial \lambda} d\theta
\end{array}\right).
\end{equation}
Note that $D_{(\gamma,\lambda)}$ is obtained from $D_{(u,\lambda)}$
for $(u(s,\theta),\lambda(s))\equiv (\gamma(\theta),\lambda)$ and
$(\zeta(s,\theta),\ell(s)) \equiv (\zeta(\theta),\ell)$.

We say that a critical point $(\gamma,\lambda)$ is {\bf nondegenerate} if
the Hessian $d^2\cA(\gamma,\lambda)$ has trivial kernel. In~\cite[Lemma~2.3]{BOtrans} we proved that this condition is equivalent to the injectivity of the asymptotic
operator $D_{(\gamma,\lambda)}$. Since the latter is self-adjoint, this condition is also equivalent to
its surjectivity.

\begin{remark}
  We note that nondegeneracy of a critical point $(\gamma,\lambda)$
  does not imply that $\gamma$ is a nondegenerate orbit of
  $H_\lambda$, nor that $\lambda$ is a nondegenerate critical point of
  $F_\gamma$. This situation is already present in Morse theory, as
  the following example shows. We consider the Morse function
  $f:\R\times \R\to\R$, $(x,y)\mapsto xy$. Then $(x_0,y_0)=(0,0)$ is a
  nondegenerate critical point, but the restrictions of $f$ to
  $\R\times \{0\}$ and $\{0\}\times \R$ are constant, hence $x_0=0$
  and $y_0=0$ are degenerate critical points.
\end{remark}

An admissible Hamiltonian family $H$ is called {\bf nondegenerate} if $\cP(H)$
consists of nondegenerate elements.
We denote the set of nondegenerate and admissible Hamiltonian families by
$\cH_{\Lambda,\textrm{reg}}\subset \cH_\Lambda$. By~\cite[Proposition~2.5]{BOtrans}, the
set $\cH_{\Lambda,\textrm{reg}}$ is of the second Baire category in
$\cH_\Lambda$. Moreover, if $H\in\cH_{\Lambda,\textrm{reg}}$ the set
$\cP(H)$ is discrete.

We denote
\begin{eqnarray*}
\cW^{1,p} & := & W^{1,p}(\R\times S^1,u^*T\widehat W) \oplus
W^{1,p}(\R,\lambda^*T\Lambda), \\
\cL^p & := & L^p(\R\times S^1,u^*T\widehat W) \oplus
L^p(\R,\lambda^*T\Lambda).
\end{eqnarray*}
Let $(\og,\olambda),(\ug,\ulambda)\in\cP(H)$ be nondegenerate. We proved in~\cite[Theorem~2.6]{BOtrans} that, given any $(u,\lambda)\in
\widehat \cM((\og,\olambda),(\ug,\ulambda); H,J,g)$, the operator
$$
D_{(u,\lambda)} : \cW^{1,p}\to \cL^p
$$
is Fredholm for $1<p<\infty$.

\begin{remark} We can choose a unitary trivialization of
$u^*T\widehat W$
and a trivialization of $\lambda^*T\Lambda$ in which $D_{(u,\lambda)}$ has
the form
\begin{equation} \label{eq:Dtriv}
D_{(u,\lambda)}\left(\begin{array}{c}
\zeta \\ \ell
\end{array}\right)
:=
\Bigg[\left(\begin{array}{cc}
\partial_s +J_0\partial_\theta & 0 \\ 0 & \frac d {ds}
\end{array}\right)
+ N
\Bigg]
\left(\begin{array}{c}
\zeta \\ \ell
\end{array}\right),
\end{equation}
with $N:\R\times S^1\to \mathrm{Mat}_{2n+m}(\R)$ pointwise
bounded and $\lim_{s\to\pm\infty}N(s,\theta)$
symmetric.
\end{remark}

Let $H\in\cH_{\Lambda,\mathrm{reg}}$. A pair $(J,g)\in\cJ_\Lambda$ is
  called {\bf regular for \boldmath$H$} if the operator $D_{(u,\lambda)}$ is
  surjective for any solution $(u,\lambda)$
  of~(\ref{eq:Floer1par}-\ref{eq:asymptoticpar}). We denote the space
  of such pairs by $\cJ_{\Lambda,\mathrm{reg}}(H)$. We proved in~\cite[Theorem~4.1]{BOtrans} that there exists a subset of second Baire category
$\cH\cJ_{\Lambda,\mathrm{reg}}\subset
\cH_{\Lambda,\mathrm{reg}}\times \cJ_\Lambda$ such that $H\in\cH_{\Lambda,\mathrm{reg}}$ and $(J,g)\in
\cJ_{\Lambda,\mathrm{reg}}(H)$ whenever $(H,J,g)\in
\cH\cJ_{\Lambda,\mathrm{reg}}$.

As a consequence, whenever $(H,J,g)\in\cH\cJ_{\Lambda,\mathrm{reg}}$, the moduli spaces of parametrized Floer trajectories
$\cM(\op,\up;H,J,g)$ are smooth manifolds, for all $\op,\up\in\cP(H)$. The local dimension at $(u,\lambda)\in \cM(\op,\up;H,J,g)$ is equal to
$\mathrm{ind}\,D_{(u,\lambda)}-1$.
The purpose of the next section is to compute this Fredholm index.

\subsection{The parametrized Robbin-Salamon index} \label{sec:RS}

Recall that, for each free homotopy class $a$ in $\widehat W$, we have chosen
in Section~\ref{sec:action} a reference loop $l_a$ such that
$[l_a]=a$. We now choose a symplectic trivialization
$$
\Phi^1_a:S^1\times \R^{2n} \to l_a^*T\widehat W
$$
for each free homotopy class $a$. If $a$ is the trivial homotopy
class we choose the trivialization to be constant.

For each $p=(\gamma,\lambda)\in\cP(H)$ we choose a smooth
homotopy $\sigma_p:[0,1]\times S^1\to\widehat W$ such that
$\sigma_p(0,\cdot)=l_{[\gamma]}$ and $\sigma_p(1,\cdot)=\gamma$.
This gives rise to a unique (up to homotopy) symplectic trivialization
$$
\Phi^1_p: [0,1]\times S^1\times \R^{2n} \to \sigma^*_pT\widehat W
$$
such that $\Phi^1_p=\Phi^1_{[\gamma]}$ on $\{0\} \times S^1 \times
\R^{2n}$. Moreover, we fix an isometry $\Phi^2_p:\R^m\to
T_\lambda\Lambda$.

We define the subgroup $\cS_{n,m}\subset \Sp(2n+2m)$ to consist of
matrices of the form
$$
M=M(\Psi,X,E)=\left(\begin{array}{ccc}
\Psi & \Psi X & 0 \\
0 & \one & 0 \\
X^TJ_0 & E+\frac 1 2 X^T J_0 X & \one
\end{array}\right),
$$
with $\Psi\in\Sp(2n)$,
$X\in\mathrm{Mat}_{2n,m}(\R)$, and $E\in \mathrm{Mat}_m(\R)$
symmetric. Here we have denoted $J_0:=\left(\begin{array}{cc} 0 &
    -\one \\ \one & 0 \end{array}\right)$ the standard complex
structure on $\R^{2n}$, and the elements $\Psi\in\Sp(2n)$ are
characterized by the condition $\Psi^T J_0\Psi=J_0$.
Similarly, we denote the standard complex structure on $\R^{2n}\times \R^{2m}$ by
$$
\widetilde J_0:=\left(
  \begin{array}{ccc}
J_0 & 0 & 0 \\
0 & 0 & -\one \\
0 & \one & 0
  \end{array}\right),
$$
and the elements $\widetilde \Psi\in\Sp(2n+2m)$ are characterized by
the condition $\widetilde \Psi^T \widetilde J_0\widetilde
\Psi=\widetilde J_0$.
That $\cS_{n,m}$
is a subgroup follows from the relations
\begin{eqnarray*}
\lefteqn{M(\Psi_1,X_1,E_1)\cdot M(\Psi_2,X_2,E_2)}\\
&=& M(\Psi_1\Psi_2,X_2+\Psi_2^{-1}X_1, E_1+E_2+\mathrm{Sym}(X_1^TJ_0\Psi_2X_2))
\end{eqnarray*}
and
\begin{equation} \label{eq:inverse}
M(\Psi,X,E)^{-1}=M(\Psi^{-1},-\Psi X,-E).
\end{equation}
Here we have denoted by
$$
\mathrm{Sym}(P):=(P^T+P)/2
$$
the symmetric part of a square matrix $P$.

To an element $p=(\gamma,\lambda)\in\cP(H)$ equipped with a unitary
trivialization of $\gamma^*T\widehat W$ and an isometry
$T_\lambda\Lambda\equiv\R^m$, we will now associate a path
$$
M(\theta)=M(\Psi(\theta),X(\theta),E(\theta)), \qquad \theta\in[0,1],
$$
with $M(0)=\one=M(\one,0,0)$.
In the given trivialization of $T(\widehat W \times  \Lambda)$ along $\gamma$,
the linearization of the flow $\Phi^\theta$ of the Hamiltonian vector
field $X^\theta_H$ has the form
\begin{eqnarray}
T_{(\gamma(0),\lambda)} (\widehat W \times \Lambda) &\to&
T_{(\gamma(\theta),\lambda)} (\widehat W \times \Lambda), \nonumber \\
(\zeta_0, l) &\mapsto& (\Psi(\theta) \zeta_0 + \Psi(\theta) X(\theta)
l, l) . \label{eq:linflow}
\end{eqnarray}
This uniquely defines the matrices $\Psi(\theta)$ and $X(\theta)$. The
matrices $E(\theta)$ are defined to be the symmetric part of the
endomorphisms
\begin{eqnarray}
T_\lambda \Lambda &\to& T_\lambda \Lambda, \nonumber \\
l &\mapsto& -\frac{d}{d\lambda} \int_0^\theta \vec \nabla_\lambda H
(\tau, \Phi^\tau(\gamma(0),\lambda), \lambda) \, d\tau \cdot l . \label{eq:Easy}
\end{eqnarray}

We define the {\bf parametrized
  Robbin-Salamon index} $\mu(p)$ of $p$ with respect to the given
trivialization as the Robbin-Salamon index~\cite{RS} of the
path $M:[0,1]\to\cS_{n,m}\subset \Sp(2n+2m)$:
\begin{equation} \label{eq:defiRS}
\mu(p):=\mu_{RS}(M(\cdot)).
\end{equation}
Note that the path $M$ corresponds to the linearized Hamiltonian flow of $\widetilde H:S^1\times \widehat W\times T^*\Lambda\to \R$ defined by
$\widetilde H(\theta,x,(\lambda,p)):=H(\theta,x,\lambda)$ (see Definition~\ref{defi:RS} in the Introduction).
We refer to Appendix~\ref{app:RS} for a summary of the properties of
the parametrized Robbin-Salamon index.

\begin{theorem} \label{thm:index}
Assume $(\og,\olambda),(\ug,\ulambda)\in\cP(H)$ are
  nondegenerate and fix $1<p<\infty$. For any $(u,\lambda)\in
\widehat \cM((\og,\olambda),(\ug,\ulambda); H,J,g)$ the index of the
Fredholm operator $D_{(u,\lambda)}:\cW^{1,p}\to\cL^p$ is
$$
\ind \, D_{(u,\lambda)} = -\mu(\og, \olambda) + \mu(\ug, \ulambda) .
$$
\end{theorem}

In the above statement, it is understood that the trivialization used
to define $\mu(\og,\olambda)$ is obtained from the trivialization used
to define $\mu(\ug,\ulambda)$ by continuation along the map $u$. We
denote by
$$
\overline M:=M(\overline \Psi,\overline X,\overline E), \qquad
\underline M:=M(\underline \Psi,\underline X,\underline E)
$$
the paths (based at $\one$) in $\cS_{n,m}$ used to define $\mu(\og,\olambda)$,
respectively $\mu(\ug,\ulambda)$.

\begin{proof} By elliptic regularity for $D_{(u,\lambda)}$ (and its
  formal adjoint), it is enough to prove the statement for
  $p=2$. Indeed, the kernel and cokernel of $D_{(u,\lambda)}$ are
  spanned by smooth elements, so that the index does not depend on
  $p$. Given a unitary trivialization of $u^*T\widehat W$ and
  an orthogonal trivialization of $\lambda^* T\Lambda$ we can write
  $D_{(u,\lambda)}$ as
$$
D_{(u,\lambda)}:H^1(\R\times S^1,\R^{2n}) \times H^1(\R,\R^m) \to
L^2(\R\times S^1,\R^{2n})\times L^2(\R,\R^m),
$$
$$
D_{(u,\lambda)}(\zeta,\ell) = \left(\begin{array}{c} \p_s\zeta \\
    \p_s\ell \end{array}\right) +
A(s)\left(\begin{array}{c}\zeta\\\ell\end{array}\right),
$$
where $A(s):H^1(S^1,\R^{2n}) \times \R^m \to
L^2(S^1,\R^{2n})\times \R^m$ has the property that $A(s)\to A^\pm$,
$s\to\pm\infty$ and $A^\pm$ coincide through the given trivializations
with the asymptotic operators $D_{(\ug,\ulambda)}$ and
$D_{(\og,\olambda)}$, which are bijective in view of our standing nondegeneracy assumption.

In order to compute the Fredholm index of the operator $D_{(u,\lambda)}$, we use
the spectral flow of the family of self-adjoint operators $A(s)$, $s
\in \R$~\cite[Theorem~A]{RSspec}.
In view of \eqref{eq:Dasy}, these operators can be written in the
given trivializations of $T\widehat W$ and $T\Lambda$ along $u$ and
$\lambda$ as
\begin{equation} \label{eq:A}
A(s)(\zeta,l) = \left( \begin{array}{c}
J_0 \partial_\theta \zeta(\theta) + S(s,\theta) \zeta(\theta)
+ C(s,\theta)^T l \\
\int_{S^1} C(s,\theta) \zeta(\theta) \, d\theta + \int_{S^1} D(s,\theta) \, d\theta \ l
\end{array} \right) ,
\end{equation}
where $S(s,\theta) = S(s,\theta)^T$ and $D(s,\theta) = D(s,\theta)^T$
are symmetric matrices.

Let us compute the kernel of the operator $A(s)$ for a fixed value of
$s\in\R$. We define
$$
\Psi:\R\times [0,1]\to\Sp(2n)
$$
by $\dot \Psi(s,\theta) = J_0 S(s,\theta) \Psi(s,\theta)$ and
$\Psi(s,0)=\one$, so that
$$
\lim_{s\to-\infty} \Psi(s,\cdot)=\overline \Psi(\cdot), \qquad
\lim_{s\to\infty} \Psi(s,\cdot)=\underline \Psi(\cdot).
$$

For $(\zeta, l) \in \ker A(s)$, we write $\zeta(\theta) =
\Psi(s,\theta) \eta(\theta)$ for some smooth
function $\eta : [0,1] \to \R^{2n}$. Substituting this in the first
component of $A(s)(\zeta, l)$, we obtain
\begin{equation}  \label{eq:diffeta}
\dot \eta(\theta) = \Psi(s,\theta)^{-1} J_0 C(s,\theta)^T l .
\end{equation}

We define $X : \R\times [0,1] \to \Mat_{2n,m}(\R)$ by
\begin{equation}\label{eq:X}
\dot X(s,\theta) = \Psi(s,\theta)^{-1} J_0 C(s,\theta)^T
\end{equation}
and $X(s,0) = 0$. The solution
of \eqref{eq:diffeta} is then $\eta(\theta) = X(s,\theta) l + \eta(0)$, so that
\begin{equation} \label{eq:zeta}
\zeta(\theta) = \Psi(s,\theta) \zeta_0 + \Psi(s,\theta) X(s,\theta) l ,
\end{equation}
with $\zeta_0 = \zeta(0) = \eta(0)$. Comparing~\eqref{eq:zeta}
with~\eqref{eq:linflow} we see that
$$
\lim_{s\to-\infty} X(s,\cdot)=\overline X(\cdot), \qquad
\lim_{s\to\infty} X(s,\cdot)=\underline X(\cdot).
$$
The solution $\zeta(\theta)$ given by~\eqref{eq:zeta} descends to $S^1 = \R/\Z$
if and only if
\begin{equation}  \label{eq:degen1}
\zeta_0 = \Psi(s,1) \zeta_0 + \Psi(s,1) X(s,1) l .
\end{equation}
Substituting the expression~\eqref{eq:zeta} for $\zeta(\theta)$ in the
second component of\break $A(s)(\zeta, l)$, we obtain
\begin{equation}  \label{eq:degen2}
\int_0^1 C(s,\theta) \Psi(s,\theta)  d\theta \, \zeta_0
+ \int_0^1 \big(C(s,\theta) \Psi(s,\theta) X(s,\theta) + D(s,\theta)\big) d\theta \, l = 0 .
\end{equation}
We now notice that, for any $\theta\in [0,1]$, we have
\begin{equation}  \label{eq:CT=B}
\int_0^\theta C(\tau) \Psi(\tau) \, d\tau = \int_0^\theta \dot
X(\tau)^T \Psi(\tau)^T J_0 \Psi(\tau)\, d\tau = X(\theta)^T J_0.
\end{equation}

We define
$$
E : \R\times [0,1] \to \Mat_m(\R)
$$
by
\begin{equation}  \label{eq:E}
E(s,\theta) = \int_0^\theta \big( C(s,\tau) \Psi(s,\tau) X(s,\tau) +
D(s,\tau) \big)d\tau  - \frac 12 X(s,\theta)^T J_0 X(s,\theta).
\end{equation}
We claim that $\frac 1 2 X(s,\theta)^T J_0 X(s,\theta)$ is the
anti-symmetric part of the matrix $\int_0^\theta
C(s,\tau)\Psi(s,\tau)X(s,\tau)d\tau$, so that $E(s,\theta)$ is
symmetric. Omitting the $s$-variable for clarity, and using that
$C(\tau)\Psi(\tau)=\dot X(\tau)^TJ_0$, we obtain
\begin{eqnarray*}
\lefteqn{\hspace{-1.5cm}\int_0^\theta C(\tau)\Psi(\tau)X(\tau)d\tau - \int_0^\theta X(\tau)^T
\Psi(\tau)^TC(\tau)^Td\tau}\\
& = & \int_0^\theta \dot
X(\tau)^TJ_0X(\tau)d\tau+\int_0^\theta X(\tau)^TJ_0\dot X(\tau)d\tau
\\
& = & X(\theta)^TJ_0X(\theta).
\end{eqnarray*}
It follows that $E(s,\theta)$ is the symmetric part of
$\int_0^\theta
\big(C(s,\tau)\Psi(s,\tau)X(s,\tau)+D(s,\tau)\big)d\tau$. Comparing
this with~\eqref{eq:Easy}, it follows that
$$
\lim_{s\to-\infty} E(s,\cdot)=\overline E(\cdot), \qquad
\lim_{s\to\infty} E(s,\cdot)=\underline E(\cdot).
$$

With our new notations in place, we see that~\eqref{eq:degen1}
and~\eqref{eq:degen2} are equivalent to the $(2n+m) \times (2n+m)$
system of linear equations
\begin{equation}  \label{eq:degen}
\left( \begin{array}{cc}
\Psi(s,1) - \one & \Psi(s,1) X(s,1) \\
X(s,1)^TJ_0 & E(s,1) + \frac 1 2 X(s,1)^TJ_0X(s,1)
\end{array} \right)
\left( \begin{array}{c}
\zeta_0 \\ l
\end{array} \right) =
\left( \begin{array}{c}
0 \\ 0
\end{array} \right) .
\end{equation}
The solutions of the system~\eqref{eq:degen} are in bijective correspondence
with the elements $(\zeta, l) \in \ker A(s)$ through equation~\eqref{eq:zeta}.
On the other hand, it follows from the definition of $\cS_{n,m}$ that
solutions of~\eqref{eq:degen} are in bijective correspondence with elements
$$
(\zeta_0,l,0)\in \ker\, \big(M(\Psi(s,1), X(s,1)
, E(s,1)) - \one\big).
$$
Since $(0,0,v)\in \ker\, \big(M(\Psi(s,1), X(s,1)
, E(s,1)) - \one\big)$ for all $v\in \R^m$, we infer that $\ker A(s)
\neq 0$ if and only if
\begin{equation}
\dim\ker \, \big( M(\Psi(s,1), X(s,1) , E(s,1)) - \one \big) >m.
\end{equation}

\begin{remark} {\it
To each operator $A(s)$ of the form \eqref{eq:A} we associated
a path of matrices $M : [0,1] \to \cS_{n,m}$,
$M(\theta) = M(\Psi(\theta), X(\theta), E(\theta))$  such that $M(0)
=\one$. Conversely, any such path $M$ determines a unique
operator $A(s)$ of the form \eqref{eq:A} by the formulas
\begin{eqnarray*}
S(\theta) &=& -J_0 \dot \Psi(\theta) \Psi(\theta)^{-1} \\
C(\theta) &=& \dot X(\theta)^T \Psi(\theta)^T J_0  \\
D(\theta) &=& \dot E(\theta) + \mathrm{Sym} \left( X(\theta)^T J_0 \dot X(\theta) \right).
\end{eqnarray*}}
\end{remark}

We now compute the crossing form $\Gamma(A,s)$ on $\ker A(s)$ for the spectral flow of the family
of operators $A(s)$, $s \in \R$. Recall that it is defined by $\Gamma(A,s)(\zeta,l) =
\langle (\zeta, l) , \frac{d}{ds} A(s) (\zeta,l) \rangle$ for all $(\zeta,l) \in \ker A(s)$.
The operator $\frac{d}{ds} A(s)$ is given by
$$
\frac{d}{ds} A(s) (\zeta, l) = \left( \begin{array}{c}
\partial_s S(s,\theta) \zeta(\theta) + \partial_s C(s,\theta)^T l \\
\int_{S^1} \partial_s C(s,\theta) \zeta(\theta)  d\theta
+ \int_{S^1} \partial_s D(s,\theta) d\theta \ l
\end{array} \right) .
$$
Since $(\zeta, l) \in \ker A(s)$, we have
$\zeta(\theta)=\Psi(s,\theta)\zeta_0+\Psi(s,\theta)X(s,\theta)l$. We obtain
\begin{eqnarray}
\lefteqn{\Gamma(A,s)(\zeta,l)} \nonumber \\
&=&\int_{S^1} \left\langle \zeta(\theta) , \partial_s S(s,\theta) \zeta(\theta)
+ \partial_s C(s,\theta)^T l \right\rangle \, d\theta \nonumber \\
&&+ \left\langle l, \int_{S^1} \partial_s C(s,\theta) \zeta(\theta) \, d\theta
+ \int_{S^1} \partial_s D(s,\theta) \, d\theta \ l \right\rangle \nonumber \\
&=& \int_0^1  (\zeta_0 + X (s,\theta) l)^T \Psi(s,\theta)^T
\partial_s S(s,\theta) \Psi(s,\theta) (\zeta_0 +
X(s,\theta) l) \, d\theta \label{eq:first-term} \\
&& + \int_0^1  (\zeta_0 + X (s,\theta) l)^T \Psi(s,\theta)^T
\partial_s C(s,\theta)^T l  \, d\theta \nonumber \\
&& + l^T \int_0^1 \partial_s C(s,\theta)
\Psi(s,\theta) (\zeta_0 + X(s,\theta) l) \, d\theta  \nonumber \\
&& + l^T \int_{S^1} \partial_s D(s,\theta) d\theta \, l.    \label{eq:cross}
\end{eqnarray}
Let us define symmetric matrices $\widehat S(s,\theta)$ by
$\partial_s \Psi(s,\theta) = J_0 \widehat S(s,\theta)
\Psi(s,\theta)$. The condition $\Psi(s,0)=\one$ implies $\widehat
S(s,0)=0$. We claim that (see
also~\cite[proof of Lemma~2.6]{S})
\begin{equation} \label{eq:stheta}
\Psi(s,\theta)^T \partial_s S(s,\theta) \Psi(s,\theta)
= \frac{\partial}{\partial\theta} \left( \Psi(s,\theta)^T \widehat S(s,\theta) \Psi(s,\theta) \right).
\end{equation}
Dropping the $(s,\theta)$ variables for clarity, we have~\cite{S}
\begin{eqnarray*}
\frac{\partial}{\partial\theta}  \left( \Psi^T \widehat S \Psi \right)
 &=&  \Psi^T S^T (- J_0)  \widehat S \Psi
+ \Psi^T \frac{\partial}{\partial\theta} ( \widehat S  \Psi ) \\
 &=& - \Psi^T S \partial_s \Psi
+ \Psi^T \frac{\partial}{\partial\theta} \left( - J_0 \partial_s \Psi \right)  \\
 &=& - \Psi^T S \partial_s \Psi  - \Psi^T J_0 \partial_s
 \frac{\partial}{\partial\theta} \Psi  \\
 &=& - \Psi^T S \partial_s \Psi  - \Psi^T J_0 \partial_s
 (J_0 S \Psi)  \\
 &=& \Psi^T \partial_s S \Psi .
\end{eqnarray*}
Using~\eqref{eq:stheta}, the term~\eqref{eq:first-term} becomes
\begin{eqnarray*}
\lefteqn{ \int_0^1  (\zeta_0 + X (s,\theta) l)^T
\frac{\partial}{\partial\theta}  \left( \Psi(s,\theta)^T \widehat S(s,\theta) \Psi(s,\theta) \right)
(\zeta_0 + X(s,\theta) l) \, d\theta } \\
&=& (\zeta_0 + X(s,1) l)^T  \Psi(s,1)^T \widehat S(s,1) \Psi(s,1)  (\zeta_0 + X (s,1) l) \\
&& - l^T  \int_0^1 \frac{\partial}{\partial\theta} X(s,\theta)^T
\Psi(s,\theta)^T \widehat S(s,\theta) \Psi(s,\theta) (\zeta_0 + X(s,\theta) l) \, d\theta  \\
&& -  \int_0^1 (\zeta_0 + X (s,\theta) l)^T  \Psi(s,\theta)^T \widehat S(s,\theta) \Psi(s,\theta)
\frac{\partial}{\partial\theta} X(s,\theta)  \, d\theta \ l \\
&=& \zeta_0^T    \widehat S(s,1)   \zeta_0
 + l^T  \int_0^1 C(s,\theta) J_0
\widehat S(s,\theta) \Psi(s,\theta) (\zeta_0 + X(s,\theta) l) \, d\theta  \\
&& -  \int_0^1 (\zeta_0 + X (s,\theta) l)^T  \Psi(s,\theta)^T \widehat S(s,\theta)
J_0 C(s,\theta)^T  \, d\theta \ l \\
&=& \zeta_0^T  \widehat S(s,1) \zeta_0
 + l^T  \int_0^1 C(s,\theta) \frac{\partial}{\partial
   s}\Psi(s,\theta) (\zeta_0 + X(s,\theta) l) \, d\theta  \\
&& +  \int_0^1 (\zeta_0 + X (s,\theta) l)^T  \partial_s \Psi(s,\theta)^T
C(s,\theta)^T  \, d\theta \ l .
\end{eqnarray*}
The second equality uses~\eqref{eq:degen1}. Thus, equation \eqref{eq:cross} becomes
\begin{eqnarray*}
\Gamma(A,s)(\zeta,l) &=&
 \zeta_0^T \widehat S(s,1) \zeta_0
  + l^T  \int_0^1  \partial_s \big(C(s,\theta) \Psi(s,\theta) \big)
 (\zeta_0 + X(s,\theta) l) d\theta  \\
&& +  \int_0^1 (\zeta_0 + X (s,\theta) l)^T  \frac{\partial }{\partial s} \big(\Psi(s,\theta)^T
C(s,\theta)^T \big) d\theta \, l \\
&& + l^T \int_{S^1} \partial_s D(s,\theta) d\theta \, l  .
\end{eqnarray*}
We claim that the matrix of the quadratic form $\Gamma(A,s)$ acting on
the vector space of elements $(\zeta_0,l)$
satisfying~\eqref{eq:degen1} is given by
\begin{equation} \label{eq:crossA2}
\left( \begin{array}{cc}
\widehat S(s,1) & - J_0\partial_s X(s,1) \\
\partial_s X(s,1)^T J_0 &
\partial_s E(s,1) -\mathrm{Sym}\big(X^T(s,1) J_0 \partial_s X(s,1)\big)
\end{array} \right) .
\end{equation}
For the anti-diagonal terms, the claim follows
from~\eqref{eq:CT=B}. For the term in the lower right corner, we compute
\begin{eqnarray}
 \frac{\partial} {\partial s} E(s,1)
 &=&   \int_0^1 \partial_s D(s,\theta) +
 \frac {\partial} {\partial s}\mathrm{Sym}(C(s,\theta) \Psi(s,\theta)
 X(s,\theta)) \, d\theta \nonumber \\
  &=& \int_0^1 \partial_s D(s,\theta)
  + \mathrm{Sym}\big(\partial_s (C(s,\theta)
 \Psi(s,\theta)) X(s,\theta)\big)\, d\theta \nonumber \\
&& \hspace{1.8cm}  + \ \int_0^1 \mathrm{Sym}\big(C(s,\theta) \Psi(s,\theta)
 \partial_s X(s,\theta)\big) \, d\theta, \label{eq:last-term}
\end{eqnarray}
while, in view of $C\Psi=\dot X^TJ_0$, the term~\eqref{eq:last-term} becomes
\begin{eqnarray*}
\lefteqn{\int_0^1 \mathrm{Sym}\big(\dot X(s,\theta)^TJ_0
 \partial_s X(s,\theta)\big) \, d\theta} \\
&= & \hspace{-2mm}\mathrm{Sym}\big(X(s,1)^TJ_0\frac \partial {\partial
 s}X(s,1)\big) \!-\!\!
 \int_0^1 \!\mathrm{Sym}\big(X(s,\theta)^TJ_0
 \partial_s \dot X(s,\theta)\big) \, d\theta \\
&=& \hspace{-2mm} \mathrm{Sym}\big(X(s,1)^TJ_0\frac \partial {\partial
 s}X(s,1)\big) \!+\!\!
 \int_0^1 \!\mathrm{Sym}\big(X(s,\theta)^T
 \partial_s (\Psi(s,\theta)^T C(s,\theta)^T)\big) \, d\theta.
\end{eqnarray*}

Let us now compute the crossing form $\Gamma(M,s)$ for the
Robbin-Salamon index of the path
$$
s\mapsto M(s,1)=M(\Psi(s,1),X(s,1),E(s,1)).
$$
By definition, the crossing form is $\Gamma(M,s)(\zeta_0,l,v)=\langle
(\zeta_0,l,v), Q(s)(\zeta_0,l,v)\rangle$, with $Q(s):=-\widetilde J_0
\frac {\partial} {\partial s} M(s,1)M(s,1)^{-1}$.
Using~\eqref{eq:inverse} and the definition of $\widehat
S(s,1)=-J_0\frac {\partial} {\partial s} \Psi(s,1) \Psi(s,1)^{-1}$, a
straightforward computation shows that $Q(s)$ is given by
$$
\left(
  \begin{array}{ccc}
\widehat S(s,1) & -J_0\Psi(s,1)\frac {\partial} {\partial s}X(s,1) & 0
\\
\frac {\partial} {\partial s}X(s,1)^T\Psi(s,1)^TJ_0 & \frac {\partial}
{\partial s}E(s,1) + \mathrm{Sym}\big(X(s,1)^TJ_0\frac {\partial}
{\partial s}X(s,1)\big) & 0 \\
0 & 0 & 0
  \end{array}\right).
$$

The key observation now is that, for any $(\zeta_0,l,0)\in \ker \,
(M(s,1)-\one)$, we have
$$
\Gamma(M,s)(\zeta_0,l,0) =\Gamma(A,s)(\zeta,l),
$$
with $\zeta(\theta)=\Psi(s,\theta)\zeta_0 +
\Psi(s,\theta)X(s,\theta)l$. This is seen by a straightforward
computation, substituting
$\zeta_0=\Psi(s,1)\zeta_0+\Psi(s,1)X(s,1)l$ in the non-diagonal terms
of $\Gamma(M,s)(\zeta_0,l,0)$.

By Proposition~\ref{prop:RSstratum} in Appendix~\ref{app:RS} (applied with
$E(s)\equiv \{0\}\oplus\{0\}\oplus\R^m$), it follows that the spectral
flow of $A(s)$ coincides with the Robbin-Salamon index of the
degenerate path $s\mapsto M(s,1)$. Thus
 $$
 \ind \, D_{(u,\lambda)} = \mu_{RS}\left( M(\Psi(s,1), X(s,1), E(s,1)), s \in \R \right) .
 $$
 By the \emph{(Homotopy)} and \emph{(Catenation)} axioms for the
 Robbin-Salamon index~\cite{RS}, and using that
 $\lim_{s\to-\infty}M(s,\theta)=\overline M(\theta)$ and
 $\lim_{s\to\infty}M(s,\theta)=\underline M(\theta)$, we obtain
\begin{eqnarray*}
 \ind \, D_{(u,\lambda)}
 &=& \mu_{RS}\left( M(\underline \Psi(\theta), \underline X(\theta),
 \underline E(\theta)), \theta \in [0,1] \right) \\
&& -  \mu_{RS}\left( M(\overline \Psi(\theta), \overline X(\theta),
 \overline E(\theta)), \theta \in [0,1] \right) \\
&=& \mu(\ug, \ulambda) - \mu(\og, \olambda).
\end{eqnarray*}
\end{proof}

\subsection{The parametrized chain complex}
Given $H\in\cH_{\Lambda,\mathrm{reg}}$, $(J,g)\in\Jreg(H)$, and a free
homotopy class $a$ in $\widehat W$, we
define $SC^{a,\Lambda}_*(H,J,g)$ as a chain complex whose underlying
$\Lambda_\omega$-module is
$$
SC^{a,\Lambda}_*(H,J,g):=\bigoplus _{p\in\cP^a(H)}
\Lambda_\omega\langle p \rangle.
$$
We define the degree of a generator $p\in\cP(H)$ in terms of the
parametrized Robbin-Salamon index by
$$
|p|:= -\mu(p)+\frac m 2\in \Z.
$$
The fact that the grading is integral follows
from~\cite[Theorem~4.7]{RS} (we recall that $m=\dim\,\Lambda$). We
define $|p\, e^A|:=|p|-2\langle
c_1(T\widehat W),A\rangle$, where $c_1(T\widehat W)$ is computed with
respect to a compatible almost complex structure.

Recall that, for each $p=(\gamma,\lambda)\in\cP(H)$, we have chosen a
cylinder $\sigma_p:[0,1]\times S^1\to\widehat W$ such that
$\sigma_p(0,\cdot)=l_{[\gamma]}$ and $\sigma_p(1,\cdot)=\gamma$. We
define $\overline \sigma_p(s,\theta):=\sigma_p(1-s,\theta)$. Given
$\op=(\og,\olambda),\up=(\ug,\ulambda)\in\cP(H)$ we define
$$
\cM^A(\op,\up;H,J,g)\subset \cM(\op,\up;H,J,g)
$$
to consist of trajectories $(u,\lambda)$ such that
$[\sigma_{\op}\#u\#\overline \sigma_{\up}]=A\in H_2(\widehat
W;\Z)$. It follows from Theorem~\ref{thm:index} that
$$
\dim\, \cM^A(\op,\up;H,J,g) = |\op| - |\up\, e^A| -1.
$$

Let $\op:=(\og,\olambda),\up:=(\ug,\ulambda)\in\cP(H)$. Whenever
$|\op|-|\up \, e^A|=1$, one can associate to each
element $(u,\lambda)\in \cM^A(\op,\up;H,J,g)$ a sign $\eps(u,\lambda)$
via the coherent orientations recipe of Floer and Hofer~\cite{FH}. As
in their construction, since the asymptotics are fixed, the relevant
spaces of Fredholm operators are contractible, and the corresponding
determinant line bundles are trivial. Hence the moduli spaces of
parametrized Floer trajectories are orientable. Since our moduli
spaces are modeled on $\R$ as gradient trajectories, we can use the
algorithm in~\cite{FH} to construct a set of orientations which is
coherent with respect to the gluing operation. More precisely, one
chooses an element $p\in\cP(H)$, and for each $p\neq \up\in\cP(H)$ one
chooses arbitrary orientations of the spaces of operators
$\cO(p,\up)$ asymptotic to $D_p$ at $-\infty$ and to $D_{\up}$ at
$+\infty$. These determine orientations of $\cO(\up,p)$ by requiring
that the glued orientation on $\cO(p,p)$ be the one determined by the
canonical orientation of the constant operator $D_p$. We obtain
orientations on $\cO(\op,\up)$ by requiring that the glued orientation
with $\cO(p,\op)$ and $\cO(\up,p)$ be the canonical one on
$\cO(p,p)$.

We define a differential $\partial$ on $SC^{a,\Lambda}_*(H,J,g)$ by
$$
\partial\op:=\sum_{|\op|-|\up \, e^A|=1} \Big( \sum_{(u,\lambda)\in
  \cM^A(\op,\up;H,J,g)} \eps(u,\lambda)\Big) \ \up \, e^A.
$$
This expression is well-defined by standard compactness
arguments~\cite{HS,S}. More precisely, for each $A\in H_2(W;\Z)$
satisfying $|\op|-|\up\, e^A|=1$ the
set $\cM^A(\op,\up;H,J,g)$ is finite, and for each $c>0$ the number of
$A\in H_2(W;\Z)$ such that $\omega(A)\le c$ and
$\cM^A(\op,\up;H,J,g)\neq\emptyset$ is finite.

It follows from standard compactness and gluing arguments~\cite{F,S} that
$\partial^2=0$. Compactness is established in three steps. Firstly, one
obtains a uniform $C^0$-bound on the $\widehat W$-component
of parametrized Floer trajectories using the
maximum principle and the fact that $\partial^2 H/\partial s \partial
t=0$ outside a compact set~\cite[Lemma~1.5]{O}. Secondly, one proves
that the $\Lambda$-component converges by applying the
Arzel\'a-Ascoli theorem. Thirdly, the $\widehat W$-component converges
by Floer-Gromov compactness because it satisfies an $s$-dependent
Floer equation. Gluing involves exactly the same kind of estimates as
in Floer theory.

We denote the resulting homology groups by
$SH^{a,\Lambda}_*(H,J,g)$. As for usual symplectic homology, we obtain by
passing to the direct limit {\bf parametrized symplectic homology groups}
$$
SH^{a,\Lambda}_*(W) := \lim _{\stackrel \longrightarrow
  {H\in\cH_{\Lambda,\mathrm{reg}}}} SH^{a,\Lambda}_*(H,J,g).
$$

\begin{proposition}[K\"unneth formula] \label{prop:Kunneth}
 The following isomorphism holds with field coefficients
\begin{equation}
 SH^{a,\Lambda}_*(W) \simeq SH^a_*(W) \otimes H_*(\Lambda).
\end{equation}
\end{proposition}

\begin{proof} We use Hamiltonians of the form
$$
H_\lambda(\theta,x):= K(\theta,x) + f(\lambda).
$$
Here $f:\Lambda\to \R$ is a Morse function and $K$ is an admissible
Hamiltonian having nondegenerate orbits. We choose a generic
admissible almost complex structure $J$ on $W$ and a generic
Riemannian metric $g$ on $\Lambda$.

The critical points of the parametrized action functional are of the
form $(\gamma,\lambda)$, $\gamma\in \cP(K)$,
$\lambda\in \mathrm{Crit}(f)$. The
properties of the parametrized Robbin-Salamon index described in
Appendix~\ref{app:RS} imply that
$$
\mu(\gamma,\lambda)=\mu_{RS}(\gamma) + \mathrm{ind}_f(\lambda)-\frac m
2,
$$
where $\mathrm{ind}_f(\lambda)$ denotes the Morse index of
$\lambda\in\mathrm{Crit}(f)$. It follows that
$$
|(\gamma,\lambda)|=-\mu_{RS}(\gamma) + m - \mathrm{ind}_f(\lambda) =
-\mu_{RS}(\gamma) + \mathrm{ind}_{-f}(\lambda).
$$
The parametrized Floer equation is
split and has the form
$$
\left\{\begin{array}{rcl} \dbar_J u & = & J X_{H_\lambda} = JX_K, \\
\dot \lambda(s) & = & \int_{S^1} \vec \nabla_\lambda H(\theta,
u(s,\theta), \lambda(s)) d\theta = \vec \nabla f(\lambda(s)).
\end{array}\right.
$$
This follows from the obvious identities
$$
X_{H_\lambda}(\theta,x,\lambda) \equiv X_K(\theta,x),
\qquad
\vec \nabla_\lambda H(\theta,x,\lambda) \equiv \vec \nabla
f(\lambda).
$$
We obtain an isomorphism of complexes
$$
SC^{a,\Lambda}_*(H,J,g) \simeq SC^a_*(K,J) \otimes C_*(-f,g),
$$
where $SC^a_*(K,J)$ denotes the Floer complex for $(K,J)$ in the free
homotopy class $a$ (graded by
$-\mu_{RS}(\gamma)$) and $C_*(-f,g)$ denotes the Morse complex for
$(-f,g)$ (graded by $\mathrm{ind}_{-f}(\lambda)$). Since we use
field coefficients the conclusion follows by the algebraic K\"unneth
theorem.
\end{proof}

\begin{remark}{\bf (Naturality)} \label{rmk:SHincl}
 An embedding of parameter spaces $\iota:\Lambda \hookrightarrow \Lambda'$ induces
 a natural map $S\iota_*:SH^{a,\Lambda}_*(W)\to SH^{a,\Lambda'}_*(W)$ which is equal
 to $\mathrm{Id}\otimes \iota_*$ via the K\"unneth isomorphism. This
 can be seen by using a Hamiltonian
 $K(\theta,x)+f(\lambda)$ on $S^1\times \widehat
 W\times \Lambda$ as in the proof of Proposition~\ref{prop:Kunneth} above,
 and a Hamiltonian $K(\theta,x)+\widetilde f(\lambda')$ on $S^1\times \widehat
 W\times \Lambda'$, where $\widetilde f=f+|y|^2$ in a tubular
 neighbourhood of $\Lambda\subset \Lambda'$ and $y$ is the normal
 coordinate.
\end{remark}


\section{\boldmath$S^1$-equivariant theories} \label{sec:S1}

In \S\ref{sec:S1equivhom} we
give a Morse theoretic presentation of $S^1$-equivariant homology of a
manifold carrying an $S^1$-action. This serves as a motivation for
\S\ref{sec:S1equivsymplhom}, where we give the definition of the
$S^1$-equivariant symplectic homology groups $SH_*^{S^1}(W)$
following Viterbo~\cite[\S5]{V}.
We adopt a slightly more general setting and define groups
$SH_*^{a,S^1}(W)$ corresponding to nontrivial free homotopy classes of
loops in $W$.

\subsection[$S^1$-equivariant homology and Morse theory]{\boldmath$S^1$-equivariant homology and Morse theory}
  \label{sec:S1equivhom}

In this section $M$ denotes a finite-dimensional smooth manifold
carrying a smooth action of $S^1$. Our aim is to give a description
of
$$
H_*^{S^1}(M):=H_*(M\times_{S^1} ES^1)
$$
in terms of Morse homology. We recall that
$\displaystyle ES^1=\lim_{\stackrel \longrightarrow N} S^{2N+1}$ and
therefore $\displaystyle M\times_{S^1} ES^1 = \lim _{\stackrel
  \longrightarrow N} \,
M\times_{S^1} S^{2N+1}$. We denote
$$
M_{S^1} := M\times _{S^1} ES^1, \quad M_{S^1}^{(N)} := M\times _{S^1}
S^{2N+1}.
$$

The first observation is
that, given a positive integer $k$, the homology groups
$H_*(M_{S^1}^{(N)})$ stabilize in degree $*\le k$ for $N$ large
enough. Indeed, the equivariant inclusion $S^{2N+1}\hookrightarrow
S^{2N+3}$ induces an inclusion of fibrations
$$
\xymatrix
@R=15pt
{M\ \ar@{^(->}[r] & M_{S^1}^{(N)} \ar[r] \ar[d] & \C P^N \ar[d] \\
M\ \ar@{^(->}[r] & M_{S^1}^{(N+1)} \ar[r]  & \C P^{N+1}
}
$$
This induces in turn a morphism between the associated Leray-Serre
spectral sequences which is an isomorphism on the $E^2$-page in total
degree less than $N$. Functoriality of the Leray-Serre spectral
sequence implies that, for $N$ large enough (and determined by $k$),
the above inclusion induces isomorphisms $H_*(M_{S^1}^{(N)}) \stackrel
\sim \to H_*(M_{S^1}^{(N+1)})$, $*\le k$ (see for
example~\cite[Theorem~3.5]{McC}).

We can give a description of $H_*(M_{S^1}^{(N)})$ in terms of
Morse-Bott functions on $M\times S^{2N+1}$ as follows.
We choose a function $a :M \times S^{2N+1} \to \R$ which is
$S^1$-invariant, i.e.
$$
a(\tau x,\tau\lambda)=a(x,\lambda), \quad \tau \in S^1, \
(x,\lambda)\in M\times S^{2N+1},
$$
and which has only Morse-Bott circles of critical
points, i.e. the induced function $\underline a:
M_{S^1}^{(N)}\to \R$ is Morse. We denote by $S_p$, $p\in
\textrm{Crit}(a)$ these circles
of critical points and by $[p]\in M_{S^1}^{(N)}$ the
nondegenerate critical point of $\underline a$ corresponding to $S_p$,
so that $S_p = S_{\tau\cdot p}$ and $[p]=[\tau\cdot p]$,
$\tau\in S^1$. We denote by
$$
\ind(S_p) = \ind([p])
$$
the Morse-Bott index of $S_p$.

We choose
a generic $S^1$-invariant metric $g$ on $M\times S^{2N+1}$
such that the gradient flow of $a$ has the Thom-Smale transversality
property, i.e.
$$
W^u(S_p) \pitchfork W^s(S_q), \quad p,q\in \textrm{Crit}(a).
$$
This is equivalent to asking that the gradient flow of
$\underline a$ with respect to the induced metric $\underline g$ on
$M_{S^1}^{(N)}$ satisfies $W^u([p]) \pitchfork W^s([q])$, $[p],[q]\in
\textrm{Crit}(\underline a)$.
Given $\op,\up\in \textrm{Crit}(a)$ we denote by
$$
\widehat \cM(S_\op,S_\up; a,g)
$$
the {\bf space of gradient trajectories} consisting of maps
$v=(u,\lambda):\R\to M\times S^{2N+1}$ which satisfy
\begin{equation}
\dot v = -\vec \nabla a (v)
\quad \Leftrightarrow \quad
\left\{\begin{array}{rcl}
 \dot u & = & - \vec \nabla _x  a (u,\lambda), \\
 \dot \lambda & = & - \vec \nabla _\lambda  a(u,\lambda),
\end{array}\right.
\end{equation}
and
\begin{equation}
\left\{\begin{array}{rcl}
\displaystyle
\lim_{s\to -\infty} v(s) & \in & S_\op, \\
\displaystyle
\lim_{s\to \infty} v(s) & \in & S_\up,
\end{array}\right.
\quad \Leftrightarrow \quad
\left\{\begin{array}{rcl}
\displaystyle
\lim_{s\to -\infty} (u(s),\lambda(s)) & = & (\overline x,\overline
\lambda)\in S_\op, \\
\displaystyle
\lim_{s\to \infty} (u(s),\lambda(s)) & = & (\underline x,\underline
\lambda)\in S_\up.
\end{array}\right.
\end{equation}
Here the gradient $\vec \nabla$ is considered with respect to the metric
$g$ and $\vec \nabla_x, \vec \nabla_\lambda$ are its components along $TM$
and $TS^{2N+1}$ respectively.
Under the transversality assumption for
the metric $g$ the space of gradient trajectories is a smooth manifold
of dimension
$$
\dim \, \widehat \cM(S_\op,S_\up; a,g) =
\ind(S_\op) - \ind(S_\up) +1.
$$
It carries a natural action of $\R$ by reparametrization and we
denote by
$$
\cM(S_\op,S_\up; a,g) := \widehat \cM(S_\op,S_\up; a,g) /\R
$$
the {\bf moduli space of gradient trajectories}. In our setting the
moduli space carries an action of $S^1$ and the quotient
$$
\cM_{S^1}(S_\op,S_\up; a,g) := \cM(S_\op,S_\up; a,g)/S^1
$$
is a smooth manifold of dimension
$$
\dim \, \cM_{S^1}(S_\op,S_\up; a,g)= \ind(S_\op)-\ind(S_\up)-1.
$$

The bundle with fiber $TW^u(\tau\cdot p)$, $\tau\in S^1$ over
$S_p$ is orientable since $W^u(S_p):=\bigcup_{\tau\in S^1}
W^u(\tau\cdot p)$ carries an action of $S^1$. We choose for each
$S_p$ an orientation of this bundle, which amounts to choosing an
orientation of $W^u(S_p)$. Since each $S_p$ inherits a natural
orientation from $S^1$, this determines a coorientation of the bundle
with fiber $TW^s(\tau\cdot p)$, $\tau\in S^1$ over
$S_p$ and therefore a coorientation of $W^s(S_p):=\bigcup_{\tau\in S^1}
W^s(\tau\cdot p)$. We get orientations on $\widehat
\cM(S_\op,S_\up;a,g)$ and, after quotienting out $\R$ and $S^1$, we
get orientations on $\cM_{S^1}(S_\op,S_\up;a,g)$, $\op,\up\in
\textrm{Crit}(a)$. In particular, if $\ind(\op)-\ind(\up)=1$ the
moduli space $\cM_{S^1}(S_\op,S_\up;a,g)$ is zero-dimensional and each
element $[v]$ inherits a sign $\epsilon([v])$.

We define the {\bf \boldmath$S^1$-equivariant Morse complex} by
$$
C_k^{S^1} (a,g) := \bigoplus _{\ind(S_p)=k} \Z \langle S_p \rangle,
$$
with the {\bf \boldmath$S^1$-equivariant Morse differential}
$$
d^{S^1} : C_k^{S^1} \to C_{k-1} ^{S^1},
$$
$$
d^{S^1}\langle S_\op \rangle := \sum_{\ind(S_\op)-\ind(S_\up)=1 \ }
\sum_{\ [v]\in \cM_{S^1}(S_\op,S_\up;a,g)} \epsilon([v]) \langle S_\up
\rangle.
$$

Since the elements of $\cM_{S^1}(S_\op,S_\up; a,g)$ are in one-to-one
correspondence with elements of the moduli space
$\cM([\op],[\up];\underline a, \underline g)$ of gradient trajectories
of $\underline a$ with respect to the metric $\underline g$ on
$M_{S^1}^{(N)}$, and since the rule for obtaining signs on
$\cM_{S^1}(S_\op,S_\up;a,g)$ if $\ind(S_\op)-\ind(S_\up)=1$
induces the usual Morse homology rule for signs on
$\cM([\op],[\up];\underline a, \underline g)$, we infer that the
complex $(C_*^{S^1},d^{S^1})$ is tautologically isomorphic with the
Morse complex of the pair $(\underline a,\underline g)$. Therefore
$$
H_k(C_*^{S^1},d^{S^1}) \simeq H_k(M_{S^1}^{(N)}), \quad k\in \N
$$
and, for $N$ large enough (depending on $k$), we have
$$
H_k(C_*^{S^1},d^{S^1}) \simeq H_k^{S^1}(M).
$$

\begin{remark} {\rm
The previous construction admits an obvious reformulation for any
manifold $P$ endowed with a free $S^1$-action: the homology of the
quotient $P/S^1$ can be described in terms of Morse-Bott data on $P$
alone.
}
\end{remark}

\subsection[$S^1$-equivariant symplectic homology]{\boldmath$S^1$-equivariant symplectic homology}
\label{sec:S1equivsymplhom}

In this section we give the definition of $S^1$-equivariant symplectic
homology following Viterbo~\cite{V}. Our treatment parallels the
finite dimensional case as presented in~\S\ref{sec:S1equivhom}. We obtain the definition of
$S^1$-equivariant symplectic homology as a variant of parametrized
symplectic homology with $\Lambda=S^{2N+1}$.

The space of smooth loops $\gamma:S^1\to \widehat W$ carries an action
of $S^1$ given by
$$
(\tau \cdot \gamma)(\cdot) := \gamma(\cdot -\tau), \quad \tau \in S^1.
$$

Let $H:S^1 \times \widehat W \times S^{2N+1} \to \R$ be a family of
Hamiltonian functions denoted by
$H(\theta,x,\lambda)=H_\lambda(\theta,x)$. This defines a family of
action functionals
$$
\cA : C^\infty(S^1,\widehat W)\times S^{2N+1} \to \R,
$$
$$
\cA(\gamma,\lambda) = \cA_\lambda(\gamma) := -\int_{[0,1]\times S^1} \sigma^*
\om - \int_{S^1} H_\lambda(\theta,\gamma(\theta)) d\theta,
$$
where $\sigma:[0,1]\times S^1\to \widehat W$ is a smooth homotopy from
$l_{[\gamma]}$ to $\gamma$, and $l_{[\gamma]}$ is a fixed
representative of the free homotopy class of $\gamma$.

\begin{lemma}
 The family $\cA$ is invariant with respect to the diagonal action
of $S^1$ if and only if the family of Hamiltonians satisfies
\begin{equation} \label{eq:H}
H_{\tau\lambda}(\theta+\tau,\cdot)=H_\lambda(\theta,\cdot) +
r(\theta,\tau,\lambda)
\end{equation}
for some function $r:S^1\times S^1\times S^{2N+1} \to \R$ such that
\begin{equation} \label{eq:h1}
\int_{S^1} r(\theta,\tau,\lambda) d\theta =0 \mbox{ for all } \tau\in
S^1,\ \lambda\in S^{2N+1}
\end{equation}
and
\begin{equation} \label{eq:h2}
r(\theta,1,\lambda) = 0, \qquad
r(\theta+\tau,-\tau,\tau\lambda)=-r(\theta,\tau,\lambda).
\end{equation}
\end{lemma}

\begin{proof}
The nontrivial implication is that invariance of $\cA$ implies the
desired condition on $H$. We thus assume that $\cA$ is
invariant, i.e. $\cA_{\tau\lambda}
(\tau\gamma) = \cA_\lambda(\gamma)$ for all loops $\gamma$. This
is equivalent to the equality
$$
\int_{S^1} H_{\tau\lambda}(\theta+\tau,\gamma(\theta))d\theta
= \int_{S^1} H_\lambda(\theta,\gamma(\theta)) d\theta, \ \forall \ \gamma
$$
and, denoting
$F(\theta,\tau,\lambda,x):=H_{\tau\lambda}(\theta+\tau,x)-H_\lambda(\theta,x)$, we
obtain
$$
\int_{S^1} F(\theta,\tau,\lambda,\gamma(\theta))d\theta=0, \ \forall \ \gamma,\tau,\lambda.
$$
By letting $\gamma$ vary in the neighbourhood of the constant loop at
some $x\in \widehat W$ we see that
we must have $\int_{S^1} D_x F(\theta,\tau,\lambda,x)\cdot \zeta(\theta) d\theta =0$ for
all loops $\zeta$ of tangent vectors at $x$. It follows that $D_x
F(\theta,\tau,\lambda,x)=0$ for all $\theta\in S^1$ and, since $x$ was chosen
arbitrarily, we get $F(\theta,\tau,\lambda,x)=r(\theta,\tau,\lambda)$
with $\int_{S^1} r(\theta,\tau,\lambda) d\theta =0$. This
shows~\eqref{eq:h1}, whereas~\eqref{eq:h2} is straightforward.
\end{proof}

\begin{remark} {\rm
 Condition~\eqref{eq:H} holds for example if $r\equiv 0$,
i.e. if the family $H$ satisfies
\begin{equation} \label{eq:HS1inv}
H_{\tau\lambda}(\theta+\tau,\cdot)=H_\lambda(\theta,\cdot).
\end{equation}
In particular one can choose the family $H$ to be given by a
single autonomous Hamiltonian $H(\theta,x,\lambda)=H(x)$.
}
\end{remark}

We denote by $\cH^{S^1}_N\subset \cH_{S^{2N+1}}$ the set of admissible
Hamiltonian families $H:S^1\times \widehat W\times S^{2N+1}\to \R$
satisfying condition~\eqref{eq:HS1inv}. It follows from the
definitions that there exists $t_0\ge 0$ such that, for $t\ge t_0$, we
have $H(\theta,p,t,\lambda)=\beta e^t +\beta'(\lambda)$, with
$0<\beta\notin\mathrm{Spec}(M,\alpha)$, and $\beta'\in
C^\infty(S^{2N+1},\R)$ invariant under the action of $S^1$.

The differential of $\cA$ is given by~\eqref{eq:dA} and critical
points of $\cA$ satisfy~\eqref{eq:periodicpar}. Since $\cA$ is
$S^1$-invariant, the set of critical points of $\cA$ is
$S^1$-invariant as well,
i.e. if $(\gamma,\lambda)\in\cP(H)$, then $(\tau\gamma,\tau\lambda)\in
\cP(H)$ for all $\tau\in S^1$. Given
$p:=(\gamma,\lambda)\in \cP(H)$ we denote
$$
S_p=S_{(\gamma,\lambda)}:= \{(\tau\gamma,\tau\lambda) \ : \
\tau\in S^1\} \subset \cP(H),
$$
so that $S_p=S_{\tau \cdot p}$, $\tau\in S^1$. We shall refer to $S_p$
as an {\bf \boldmath$S^1$-orbit of critical points} (of $\cA$).

An {\bf admissible family of almost complex structures}
$J=(J_\lambda^\theta)$ (in the sense of Section~\ref{sec:param}) is
called {\bf \boldmath$S^1$-invariant} if it satisfies the condition
\begin{equation} \label{eq:J}
J_{\tau\lambda}^{\theta+\tau}=J_\lambda^\theta, \qquad \theta\in S^1,\
\tau\in S^1, \ \lambda\in S^{2N+1}.
\end{equation}
Such a $J^\theta$ induces an $S^{2N+1}$-family of $L^2$-metrics on
$C^\infty(S^1,\widehat W)$ defined by
$$
\langle \zeta,\eta\rangle_\lambda := \int_{S^1}
\om(\zeta(\theta),J_\lambda^\theta\eta(\theta)) d\theta, \quad \zeta,\eta\in
T_\gamma C^\infty(S^1,\widehat
W)=\Gamma(\gamma^*T\widehat W).
$$
Condition~\eqref{eq:J} ensures that, when coupled with an
$S^1$-invariant metric $g$ on $S^{2N+1}$, this family gives
rise to an $S^1$-invariant metric on
$C^\infty(S^1,\widehat W) \times S^{2N+1}$. We denote by $\cJ_N^{S^1}$
the set of pairs $(J,g)$ consisting of an $S^1$-invariant admissible
family of almost complex structures $J$ on $\widehat W$ and of an
$S^1$-invariant Riemannian metric $g$ on $S^{2N+1}$.

Given
$H\in\cH^{S^1}_N$, $(J,g)\in\cJ^{S^1}_N$, and
$\op:=(\og,\olambda),\up:=(\ug,\ulambda)\in \cP(H)$, we denote by
$$
\widehat \cM(S_\op,S_\up;H,J,g)
$$
the {\bf space of \boldmath$S^1$-equivariant Floer trajectories}, consisting of
pairs $(u,\lambda)$ with
$$
u:\R\times S^1 \to \widehat W, \qquad \lambda:\R\to S^{2N+1},
$$
satisfying
\begin{eqnarray}
\label{eq:Floer1}
 \p_s u + J_{\lambda(s)}^\theta \p_\theta u -
J_{\lambda(s)}^\theta X_{H_{\lambda(s)}}^\theta (u) & = & 0, \\
\label{eq:Floer2}
 \dot \lambda (s) - \int_{S^1} \vec \nabla_\lambda
H(\theta,u(s,\theta),\lambda(s)) d\theta & = & 0,
\end{eqnarray}
and
\begin{equation} \label{eq:asymptotic}
 \lim_{s\to -\infty} (u(s,\cdot),\lambda(s)) \in S_\op, \quad
 \lim_{s\to +\infty} (u(s,\cdot),\lambda(s)) \in S_\up.
\end{equation}

The additive group $\R$ acts on $\widehat \cM(S_\op,S_\up;H,J,g)$
by reparametrization in the $s$-variable. We denote by
$$
\cM(S_\op,S_\up;H,J,g) := \widehat \cM(S_\op,S_\up;H,J,g)/\R
$$
the {\bf moduli space of \boldmath$S^1$-equivariant Floer trajectories}.
This space is endowed with natural evaluation maps
$$
\oev:\cM(S_\op,S_\up;H,J,g)\to S_\op, \qquad
\uev:\cM(S_\op,S_\up;H,J,g)\to S_\up.
$$

An $S^1$-orbit of critical points $S_p\subset \cP(H)$ is called {\bf
  nondegenerate} if the Hessian $d^2\cA(\gamma,\lambda)$ has a
  $1$-dimensional kernel $V_p$ for some (and hence any) $(\gamma,\lambda)\in
  S_p$. It follows from~\cite[Lemma~2.3]{BOtrans}
  that nondegeneracy is equivalent to the fact
  that the kernel of the asymptotic operator $D_p$ is also
  $1$-dimensional and equal to $V_p$. In both cases, a generator of
  $V_p$ is given by the infinitesimal generator of the $S^1$-action.

We define the set $\cH^{S^1}_{N,\reg}\subset \cH^{S^1}_N$ to consist
of elements $H$ such that, for any $p\in\cP(H)$, the $S^1$-orbit $S_p$
is nondegenerate. We proved in~\cite[Proposition~5.1]{BOtrans} that
 the set $\cH^{S^1}_{N,\reg}$ is of the second Baire category in
 $\cH^{S^1}_N$. Moreover, if $H\in \cH^{S^1}_{N,\reg}$, each
 $S^1$-orbit $S_p\subset C^\infty(S^1,\widehat W)\times S^{2N+1}$ is
 isolated.

Let $d>0$ be small enough (for a fixed $H\in\cH^{S^1}_{N,\reg}$, one
can take $d>0$ to be smaller than the minimal spectral gap of the asymptotic
operators $D_p$, $p\in\cP(H)$), and fix $1<p<\infty$. Given
$\op,\up\in \cP(H)$ and $(u,\lambda)\in \widehat
\cM(S_\op,S_\up;H,J,g)$, we define
\begin{eqnarray*}
  \cW^{1,p,d} & := & W^{1,p}(u^*T\widehat
  W;e^{d|s|}dsd\theta) \oplus W^{1,p}(\lambda^*
  TS^{2N+1};e^{d|s|}ds)\oplus V_{\op}\oplus V_{\up}, \\
 \cL^{p,d} & := & L^p(u^*T\widehat
  W;e^{d|s|}dsd\theta) \oplus L^p(\lambda^*
  TS^{2N+1};e^{d|s|}ds).
\end{eqnarray*}
Here we identify $V_{\op}$, $V_{\up}$ with the $1$-dimensional spaces
generated by the sections $\beta(s)(\dot\og,X_{\olambda})$, respectively $\beta(-s)(\dot\ug,X_{\ulambda})$
 of $u^*T\widehat W\oplus \lambda^*TS^{2N+1}$. For this identification, we denote by $X_{\olambda}$, $X_{\ulambda}$ the values of the infinitesimal generator of the $S^1$-action on $S^{2N+1}$
 at the points $\olambda$, respectively $\ulambda$, and choose a cut-off function $\beta:\R\to [0,1]$ which is equal to $1$ near $-\infty$, and vanishes near $+\infty$.

\begin{proposition} \label{prop:indexMB}
Assume $S_\op,S_\up\subset \cP(H)$ are nondegenerate. For any
$(u,\lambda)\in \widehat \cM(S_\op,S_\up;H,J,g)$ the operator
$$
D_{(u,\lambda)}: \cW^{1,p,d} \to \cL^{p,d}
$$
is Fredholm of index
$$
\ind\, D_{(u,\lambda)} = -\mu(\op) +\mu(\up) + 1.
$$
\end{proposition}

In the above statement, it is understood that the trivialization used
to define $\mu(\op)$ is obtained from the trivialization used
to define $\mu(\up)$ by continuation along the map $u$.

\begin{proof}
The Fredholm property was proved in~\cite[Proposition~5.2]{BOtrans} as follows.
Let $\cW^{1,p}$ and $\cL^p$ be defined as $\cW^{1,p,d}$ and $\cL^{p,d}$ above, with $d=0$ and without taking into account the direct summands $V_{\op}$, $V_{\up}$. Let $\widetilde D_{(u,\lambda)}:\cW^{1,p}\to \cL^p$ be the operator obtained by conjugating with $e^{\frac d p |s|}$ the restriction of $D_{(u,\lambda)}$ to
$W^{1,p}(u^*T\widehat   W;e^{d|s|}dsd\theta) \oplus W^{1,p}(\lambda^* TS^{2N+1};e^{d|s|}ds)$.
Then $\widetilde D_{(u,\lambda)}$ has nondegenerate asymptotics, hence it is Fredholm by~\cite[Theorem~2.6]{BOtrans}. Since the restriction of $D_{(u,\lambda)}$ to a codimension $2$ subspace is conjugate to $\widetilde D_{(u,\lambda)}$, it follows that $D_{(u,\lambda)}$ is Fredholm as well.

The asymptotic operator at $-\infty$ is $\widetilde D_{\op}=D_{\op}+\frac d p \one$, and the asymptotic operator at $+\infty$ is $\widetilde D_{\up}=D_{\up}-\frac d p \one$.  The parametrized Robbin-Salamon indices after perturbation are given by $\mu(\op)+\frac 12$, respectively $\mu(\up)-\frac 12$ (this can be seen for example using a Taylor expansion at order $1$ in $\eps$ for a perturbation of the asymptotic operators of the form $\widetilde D_p=D_p+\eps\one$).  Using Theorem~\ref{thm:index} we obtain
{\begin{eqnarray*}
\ind \, D_{(u,\lambda)} & = & \ind \, \widetilde D_{(u,\lambda)} + 2 \\
& = & -(\mu(\op)+\frac 1 2) + (\mu(\up)-\frac 1 2) + 2 \\
& = & -\mu(\op) +\mu(\up) + 1.
\end{eqnarray*}}
\end{proof}

Let $H\in\cH^{S^1}_{N,\reg}$. A pair $(J,g)\in \cJ^{S^1}_N$ is called
{\bf regular for \boldmath$H$} if the operator $D_{(u,\lambda)}$ is surjective
for any $\op,\up\in\cP(H)$ and any $(u,\lambda)\in \widehat
\cM(\op,\up;H,J,g)$. We denote the set of such regular pairs by
$\cJ^{S^1}_{N,\reg}(H)$.

We defined in~\cite[\S7]{BOtrans} two special classes $\cH_*\cJ'\subset \cH\cJ'$ in $\cH_N^{S^1}\times\cJ_N^{S^1}$. We proved in~\cite[Theorem~7.4]{BOtrans} that
there exists an open subset $\cH\cJ'_{\mathrm{reg}}\subset \cH\cJ'$
which is dense in a neighbourhood of $\cH_*\cJ'\subset \cH\cJ'$, and
which consists of triples $(H,J,g)$ such that
$$
H\in\cH^{S^1}_{N,\mathrm{reg}},\qquad (J,g)\in\cJ^{S^1}_{N,\mathrm{reg}}(H).
$$

Let $(H,J,g)\in \cH\cJ'_{\reg}$.
Recall that, for each $p=(\gamma,\lambda)\in\cP(H)$, we have chosen a
cylinder $\sigma_p:[0,1]\times S^1\to\widehat W$ such that
$\sigma_p(0,\cdot)=l_{[\gamma]}$ and $\sigma_p(1,\cdot)=\gamma$. We
define $\overline \sigma_p(s,\theta):=\sigma_p(1-s,\theta)$. Given
$\op=(\og,\olambda),\up=(\ug,\ulambda)\in\cP(H)$ we define
$$
\cM^A(S_\op,S_\up;H,J,g)\subset \cM(S_\op,S_\up;H,J,g)
$$
to consist of trajectories $(u,\lambda)$ such that
$[\sigma_{\op}\#u\#\overline \sigma_{\up}]=A\in H_2(\widehat
W;\Z)$. It follows from Proposition~\ref{prop:indexMB} that
\begin{equation} \label{eq:dimMS}
\dim\, \cM^A(S_\op,S_\up;H,J,g) = -\mu(\op) +\mu(\up) +2\langle
c_1(T\widehat W),A\rangle.
\end{equation}
Since $\cA$ and $(J,g)$ are $S^1$-invariant,
the moduli space $\cM^A(S_\op,S_\up;H,J,g)$
carries a free action of $S^1$ induced by the diagonal
action on $C^\infty(S^1,\widehat W) \times S^{2N+1}$, i.e.
$$
\tau\cdot (u,\lambda) := (u(\cdot,\cdot-\tau),\tau\lambda).
$$
We denote the quotient by
$$
\cM_{S^1}(S_\op,S_\up;H,J,g) := \cM(S_\op,S_\up;H,J,g)/S^1.
$$
This is a smooth manifold of dimension
$$
\dim\, \cM^A_{S^1}(S_\op,S_\up;H,J,g) = -\mu(\op) +\mu(\up) +2\langle
c_1(T\widehat W),A\rangle -1.
$$

An important feature of these moduli spaces is that they admit a system
of coherent orientations in the sense of~\cite{FH}. The difference with respect to
the setup of Floer homology is that the asymptotes for the moduli spaces are not fixed, but can vary along circles $S_p$, $p=(\gamma,\lambda)\in\cP(H)$. However, if one chooses the trivializations of $\gamma^*T\widehat W \oplus T_\lambda S^{2N+1}$ so that they are invariant under the $S^1$-action,
then the analytical expression of the asymptotic operators $D_p$, $p\in\cP(H)$ only depends on $S_p$.
It then follows from the arguments in~\cite{FH} that the spaces of Fredholm operators of the
form~\eqref{eq:Dtriv} with nondegenerate asymptotics of the form
$D_p$, $p\in\cP(H)$ are contractible, and hence the corresponding determinant line bundles are orientable.
The system of coherent orientations on the moduli spaces $ \cM^A_{S^1}(S_\op,S_\up;H,J,g)$ is obtained by pulling back a system of coherent orientations on these spaces of Fredholm operators, as in~\cite{FH}.

Given a free homotopy class $a$ in $\widehat W$, we define
the {\bf \boldmath$S^1$-equivariant chain complex} $SC^{a,S^1,N}_*(H,J,g)$ as a
chain complex whose underlying $\Lambda_\omega$-module is
\begin{equation} \label{eq:SCS1}
SC^{a,S^1,N}_*(H):=SC^{a,S^1,N}_*(H,J,g):=
\bigoplus_{S_p\subset \cP^a(H)} \Lambda_\omega\langle
S_p\rangle.
\end{equation}
The grading is defined by $|S_p\, e^A| := -\mu(p) +\frac m 2 -2\langle
c_1(T\widehat W),A\rangle$. The {\bf \boldmath$S^1$-equivariant differential}
$\partial^{S^1}:SC^{a,S^1,N}_*(H)\to SC^{a,S^1,N}_{*-1}(H)$ is defined
by
$$
\partial^{S^1}(S_\op):=\sum_{\substack{
  S_\up\subset \cP^a(H) \\
|S_\op| - |S_\up\, e^A|=1}}
\ \sum_{\scriptstyle [u]\in \cM^A_{S^1}(S_\op,S_\up;H,J,g)}
\epsilon([u])S_\up\, e^A.
$$
The sign $\epsilon([u])$ is obtained by comparing the coherent
orientation of the moduli space
$\cM^A_{S^1}(S_\op,S_\up;H,J,g)$ with the orientation induced by the
infinitesimal generator of the $S^1$-action.

\begin{proposition} \label{prop:partialS1}
The map $\partial^{S^1}$ satisfies
$$
\partial^{S^1}\circ \partial^{S^1}=0.
$$
\end{proposition}

The proof of Proposition~\ref{prop:partialS1} is given in
Section~\ref{sec:MBparam}.
We define the {\bf \boldmath$S^1$-equivariant Floer homology groups} by
$$
SH^{a,S^1,N}_*(H,J,g):=H_*(SC^{a,S^1,N}_*(H),\partial^{S^1}).
$$

\begin{proposition} \label{prop:indepJg}
Let $H\in\cH^{S^1}_{N,\reg}$. Given
$(J_1,g_1),(J_2,g_2)\in\cJ^{S^1}_{N,\reg}(H)$, there exists a
canonical isomorphism
$$
SH^{a,S^1,N}_*(H,J_1,g_1) \simeq SH^{a,S^1,N}_*(H,J_2,g_2).
$$
\end{proposition}

We prove Proposition~\ref{prop:indepJg} in
Section~\ref{sec:continuation}. Given $H\in\cH^{S^1}_{N,\reg}$ we shall denote
$SH^{a,S^1,N}_*(H):=SH^{a,S^1,N}_*(H,J,g)$ for $(J,g)\in
\cJ^{S^1}_{N,\reg}(H)$. In analogy with the construction of symplectic
homology, we define
$$
SH^{a,S^1,N}_*(W):=\lim_{\stackrel \longrightarrow {H\in
  \cH^{S^1}_{N,\reg}}} SH^{a,S^1,N}_*(H).
$$
The {\bf \boldmath$S^1$-equivariant symplectic homology groups of \boldmath$W$} are
defined by
$$
SH^{a,S^1}_*(W):=\lim_{\stackrel \longrightarrow N} SH^{a,S^1,N}_*(W).
$$
The direct limit is taken with respect to the embeddings
$S^{2N+1}\hookrightarrow S^{2N+3}$, inducing maps
$SH^{a,S^1,N}_*(W)\to SH^{a,S^1,N+1}_*(W)$ (see
Remark~\ref{rmk:SHincl}).

For the particular case of the trivial homotopy class $a=0$, we denote
the $S^1$-equivariant symplectic homology groups by $SH^{S_1}_*(W)$.
Given $H\in\cH^{S^1}_{N,\reg}$ we define the {\bf parametrized reduced
  action functional} $\cA^0:C^\infty_{\mathrm{contr}}(S^1,\widehat
W)\times S^{2N+1}\to\R$ by
$$
\cA^0(\gamma,\lambda):=-\int_{D^2}\sigma^*\widehat \omega - \int_{S^1}
H(\theta,\gamma(\theta),\lambda)\, d\theta.
$$
Here $\sigma:D^2\to\widehat W$ is a smooth extension of $\gamma$, and
$\cA^0$ is well-defined due to assumption~\eqref{eq:asph}.

Similarly to the case of symplectic homology, we define a special
cofinal class of Hamiltonian families $\cH^{\prime\, S^1}_N\subset
\cH^{S^1}_N$, consisting of elements $H=(H_\lambda)\in\cH^{S^1}_N$
such that $H_\lambda\in \cH'$ for all $\lambda\in S^{2N+1}$ (see
Section~\ref{sec:symplhom} for the definition of the class $\cH'$).

Given $H\in\cH^{\prime\, S^1}_{N,\reg}:=\cH^{\prime\, S^1}_N\cap
\cH^{S^1}_{N,\reg}$, $(J,g)\in\cJ^{S^1}_{N,\reg}(H)$, and $\eps>0$
small enough, we define the chain complexes
$$
SC^{-,S^1,N}_*(H,J,g):=\bigoplus _{\substack{S_p\subset\cP^0(H)
    \\ \cA^0(p)\le \eps)}} \Lambda_\omega \langle S_p\rangle \subset
SC^{S^1,N}_*(H,J,g)
$$
and
$$
SC^{+,S^1,N}_*(H,J,g):=SC^{S^1,N}_*(H,J,g)/SC^{-,S^1,N}_*(H,J,g).
$$
The differential on $SC^{\pm,S^1,N}_*(H,J,g)$ is induced by
$\partial^{S^1}$. The corresponding homology groups
$SH^{\pm,S^1,N}_*(H,J,g)$ do not depend on $(J,g)$ and $\eps$, and we
define
$$
SH^{\pm,S^1,N}_*(W):=\lim_{\stackrel \longrightarrow
  {H\in\cH^{\prime\, S^1}_{N,\reg}}}SH^{\pm,S^1,N}_*(H,J,g).
$$
Passing to the direct limit over $N\to\infty$, we define
$$
SH^{\pm,S^1}_*(W):=\lim_{\stackrel \longrightarrow
  {N}}SH^{\pm,S^1,N}_*(W).
$$
We call $SH^{+,S^1}_*(W)$ the {\bf positive \boldmath$S^1$-equivariant
  symplectic homology group} of $(W,\omega)$. It follows from the
definitions that this fits into the {\bf tautological long exact
  sequence}
$$
\dots \to SH^{+,S^1}_{k+1}(W) \to SH^{-,S^1}_k(W)\to SH^{S^1}_k(W) \to
SH^{+,S^1}_k(W)\to \dots
$$

\begin{lemma} \label{lem:minus} Assume $W$ has positive contact type boundary in the
  sense of Section~\ref{sec:symplhom}. There is a natural isomorphism
$$
SH^{-,S^1}_*(W) \simeq H^{S^1}_{*+n}(W,\partial W;\Lambda_\omega).
$$
Here $H^{S^1}_{*+n}(W,\partial W;\Lambda_\omega)\simeq
H_{*+n}(W,\partial W;\Lambda_\omega)\otimes H_*(\C P^\infty;\Q)$
denotes the $S^1$-equivariant homology of the pair $(W,\partial W)$
with respect to the trivial $S^1$-action.
\end{lemma}

\begin{proof}
We consider a Hamiltonian $H\in\cH^{\prime\, S^1}_{N,\reg}$ which has
the form
$$
H(\theta,x,\lambda)=K(x)+\widetilde f(\lambda)
$$
on $S^1\times W\times S^{2N+1}$, with $K:W\to \R$ a $C^2$-small
function, and $\widetilde f:S^{2N+1}\to \R$ the lift of a Morse
function $f:\C P^N\to\R$. We choose $(J,g)\in
\cJ^{S^1}_{N,\reg}(H)$ such that $J$ is independent of $\theta$ and
$\lambda$ on $W$ (this is possible because $W$ is symplectically
aspherical~\cite{SZ}). The
parametrized Floer equation is split, the Floer complex for $(K,J)$
reduces to the Morse complex, and we have an isomorphism of complexes
$$
SC^{-,S^1,N}_*(H,J,g)=C_{*+n}(K,J;\Lambda_\omega)\otimes
C^{S^1}_*(\widetilde f,g;\Q).
$$
Here $C_*$ denotes the corresponding Morse complexes. Since
$C^{S^1}_*(\widetilde f,g;\Q)$ corresponds to a Morse complex on $\C
P^N$, the conclusion follows.
\end{proof}


\section{Morse-Bott constructions} \label{sec:MB}

\subsection{Morse-Bott complex for parametrized symplectic homology}
  \label{sec:MBparam}

We describe in this section a Morse-Bott construction for parametrized
symplectic homology in the case when $\Lambda=S^{2N+1}$ and the action
functional $\cA:C^\infty(S^1,\widehat W)\times S^{2N+1}\to \R$ is
$S^1$-invariant with respect to the diagonal action of $S^1$. The
situation is analogous to that of Floer homology for an autonomous
Hamiltonian considered in~\cite{BOauto}.

Let $H\in\cH^{S^1}_{N,\reg}$ and $(J,g)\in\cJ^{S^1}_{N,\reg}(H)$ as in
Section~\ref{sec:S1equivsymplhom}. For each $S^1$-orbit of critical
points $S_p\subset \cP(H)$ we choose a perfect Morse function
$f_p:S_p\to\R$. We denote by $m_p$, $M_p$ the minimum, respectively
the maximum of $f_p$. Given $\op,\up\in\cP(H)$,
$Q_\op\in\mathrm{Crit}(f_\op)$, $Q_\up\in\mathrm{Crit}(f_\up)$, and
$m\ge 0$, we denote by
$$
\cM^A_m(Q_\op,Q_\up;H,\{f_p\},J,g)
$$
the union for $p_1,\dots,p_{m-1}\in\cP(H)$ and $A_1+\dots+A_m=A$ of
the fibered products
\begin{eqnarray*}
&&
\hspace{-1cm}W^u(Q_\op)
\times_{\oev}
(\cM^{A_1}(S_{\op}\,,S_{p_1})\!\times\!\R^+)
{_{\varphi_{f_{p_1}}\!\circ\uev}}\!\times
_{\oev}
(\cM^{A_2}(S_{p_1},S_{p_2})\!\times\!\R^+) \\
&&
{_{\varphi_{f_{p_2}}\!\circ\uev}\times_{\oev}} \ldots\,
{_{\varphi_{f_{p_{m-1}}}\!\!\circ\uev}}\!\!\times
_{\oev}
\cM^{A_m}(S_{p_{m-1}},\!S_{\up})
{_{\uev}\times} W^s(Q_\up).
\end{eqnarray*}
It follows from~\cite[Lemma~3.6]{BOauto} that, for a generic choice of
the collection of Morse functions $\{f_p\}$, the previous fibered
product is a smooth manifold of dimension
\begin{eqnarray*}
\lefteqn{\dim \, \cM^A_m(Q_\op,Q_\up;H,\{f_p\},J,g)} \\
& = &
-\mu(\op)+\ind_{f_\op}(Q_\op) + \mu(\up) - \ind_{f_\up}(Q_\up) +
2\langle c_1(T\widehat W),A\rangle -1.
\end{eqnarray*}
We denote
$$
\cM^A(Q_\op,Q_\up;H,\{f_p\},J,g):=\bigcup_{m\ge 0}  \cM^A_m(Q_\op,Q_\up;H,\{f_p\},J,g).
$$

Given a free homotopy class $a$ of loops in $\widehat W$, we define the {\bf
  parametrized Morse-Bott chain complex} $BC^{a,N}_*(H,\{f_p\},J,g)$
as a chain complex whose underlying $\Lambda_\omega$-module is
$$
BC^{a,N}_*(H):=BC^{a,N}_*(H,\{f_p\},J,g) := \bigoplus_{S_p\subset \cP^a(H)} \Lambda_\omega \langle
m_p,M_p\rangle.
$$
The grading is given by
\begin{eqnarray*}
|m_p \, e^A| & := & -\mu(\gamma,\lambda) + 1 - 2\langle c_1(T\widehat
 W),A\rangle, \\
|M_p\, e^A| & := & -\mu(\gamma,\lambda) - 2\langle c_1(T\widehat
 W),A\rangle.
\end{eqnarray*}
The {\bf parametrized Morse-Bott differential}
$$
d:BC^{a,N}_*(H)\to BC^{a,N}_{*-1}(H)
$$
is defined by
\begin{equation} \label{eq:dMB}
dQ_\op := \sum_{\substack{
  \up\in \cP^a(H), Q_\up\in \mathrm{Crit}(f_{\up}) \\
|Q_\op| - |Q_\up\, e^A|=1}}
\ \sum_{\scriptstyle \u\in \cM^A(Q_\op,Q_\up;H,\{f_\gamma\},J,g)}
\epsilon(\u)Q_\up\, e^A, \qquad Q_\op \in {\rm Crit}(f_\op).
\end{equation}
The sign $\eps(\u)$ is determined by the fibered-sum rule from
coherent orientations on the relevant spaces of Fredholm operators, as
explained in~\cite[Section~4.4]{BOauto}.

The Correspondence Theorem~3.7 in~\cite{BOauto} shows that $d^2=0$.
Similarly to the construction of symplectic homology, we define
$$
BH^{a,N}_*(W):=\lim_{\stackrel \longrightarrow {H\in
  \cH^{S^1}_{N,\reg}}} H_*(BC^{a,N}_*(H),d),
$$
where the direct limit is taken with respect to increasing homotopies
of Hamiltonians. It then follows from the Correspondence Theorem~3.7
in~\cite{BOauto} that
$$
BH^{a,N}_*(W)=SH^{a,S^{2N+1}}_*(W).
$$

We now define a filtration on $BC^{a,N}_*(H)$ as follows. Let
$$
B_kC^{a,N}_*(H) := \bigoplus_{
\substack{S_p\subset \cP^a(H) \\ A\in H_2(W;\Z) \\ -\mu(p)-2\langle
  c_1(T\widehat W),A\rangle =k}} \langle m_p\, e^A, \ M_p \, e^A\rangle.
$$

\begin{proposition} The $\Q$-vector spaces
$$
F_\ell BC^{a,S^1}_*(H):= \bigoplus _{k\le \ell} B_kC^{a,N}_*(H),
\qquad \ell\in\Z
$$
form an increasing filtration on $BC^{a,S^1}_*(H)$.
\end{proposition}

\begin{proof} The formula~\eqref{eq:dMB} involves elements $Q_\op$,
  $Q_\up$ satisfying $|Q_\op|-|Q_\up\, e^A|=1$, i.e.
$$
-\mu(\op)+\mathrm{ind}_{f_\op}(Q_\op) + \mu(\up)
-\mathrm{ind}_{f_\up}(Q_\up) + 2\langle c_1(T\widehat W),A\rangle =1.
$$
Since
$\mathrm{ind}_{f_\op}(Q_\op)-\mathrm{ind}_{f_\up}(Q_\up)\in\{-1,0,1\}$,
we obtain that $-\mu(\op)+ \mu(\up)+ 2\langle c_1(T\widehat
W),A\rangle\in\{0,1,2\}$.
\end{proof}

The differential $d$ splits as
$$
d=d^0+d^1+d^2
$$
with $d^r:B_kC^{a,N}_*(H)\to B_{k-r}C^{a,N}_*(H)$. The complex
$BC^{a,N}_*(H)$ admits a bi-grading which, for an element $Q_p\, e^A$
is $(-\mu(p)-2\langle c_1(T\widehat
W),A\rangle,\mathrm{ind}_{f_p}(Q_p))$. The associated
spectral sequence $(E^{a,N;r}_{*,*}(H),\bar d^r)$ is supported in two lines
and converges to $SH^{a,S^{2N+1}}_*(H)$. In particular, it degenerates
at $r=2$ and takes the form of a long exact sequence~\cite{BOcont}
\begin{equation}  \label{eq:seq}
{ 
\xymatrix
@C=15pt
@R=10pt@W=1pt@H=1pt
{
\dots \ar[r] & SH^{a,S^{2N+1}}_k(H) \ar[r]
 & E^{a,N;2}_{k,0}(H) \ar[r]^{\bar d^2} & E^{a,N;2}_{k-2,1}(H) \ar[r]
& SH^{a,S^{2N+1}}_{k-1}(H) \ar[r] & \dots
}
}
\end{equation}
The differentials $\bar d^r$, $r=0,1,2$ are induced by $d^r$. In
particular, $\bar d^1$ satisfies
\begin{equation} \label{eq:bard1}
\bar d^1\circ \bar d^1=0
\end{equation}
on $E^{a,N;1}_{*,*}(H)$.

\begin{proposition}  \label{prop:d0}
 For any $p\in\cP^a(H)$ and any $B\in H_2(W;\Z)$, the differential
 $d^0:B_kC^{a,N}_*(H)\to B_kC^{a,N}_*(H)$ vanishes.
\end{proposition}

\begin{proof} Since $d$ is $\Lambda_\omega$-linear, it is enough to
  prove the statement for $B=0$. By definition $d^0(Q_p)$,
  $Q_p\in\mathrm{Crit}(f_p)$ involves critical points of $f_\up$,
  $\up\in\cP^a(H)$ satisfying $-\mu(p)+\mu(\up) +2\langle
  c_1(T\widehat W),A\rangle=0$. On the other hand, the dimension of
  the moduli spaces $\cM^{A_1}(S_{p_1},S_{p_2};H,J,g)$ is equal to
  $-\mu(p_1)+\mu(p_2)+2\langle c_1(T\widehat W),A_1\rangle$. Since
  these moduli spaces carry a free $S^1$-action, their dimension must
  be at least $1$. This proves that $d^0(Q_p)$ counts only gradient
  trajectories of $f_p$ emanating from $Q_p$. In particular
  $d^0(M_p)=0$, and $d^0(m_p)$ is either $0$ or equal to $\pm 2 M_p$.
That $d^0(m_p)=0$ follows from the existence of a system of coherent orientations on the moduli spaces
of $S^1$-equivariant Floer trajectories, by the same arguments as in~\cite[Lemma~4.28]{BOauto}.
\end{proof}

As a consequence, provided that we work with rational coefficients,
the term $E^{a,N;1}_{*,*}(H)$ can be expressed as
$$
E^{a,N;1}_{*,*}(H) = \bigoplus_{S_p\subset \cP^a(H)}
\Lambda_\omega \langle m_p,M_p\rangle.
$$
Let us denote by $M$ the generator of $H_0(S^1)$ and by $m$ the
generator of $H_1(S^1)$. It follows from the
definition~\eqref{eq:SCS1} of the $S^1$-equivariant chain complex that
there is a natural isomorphism of $\Lambda_\omega$-modules which
preserves the bi-degree
$$
\Phi:E^{a,N;1}_{*,*}(H) \stackrel \sim \to SC^{a,S^1,N}_*(H) \otimes H_*(S^1),
$$
given by
$$
\Phi(m_p):= S_p\otimes m,\qquad \Phi(M_p):=S_p\otimes M.
$$

\begin{proposition} \label{prop:d1}
There is a commutative diagram
$$
\xymatrix
@R=25pt
{E^{a,N;1}_{*,*}(H) \ar[r]^-\Phi \ar[d]_{\bar d^1} & SC^{a,S^1,N}_*(H)
  \otimes H_*(S^1) \ar[d]^{\partial^{S^1}\otimes \mathrm{Id}} \\
E^{a,N;1}_{*,*}(H) \ar[r]_-\Phi & SC^{a,S^1,N}_*(H)
  \otimes H_*(S^1)
}
$$
\end{proposition}

\begin{proof} By definition $\bar d^1(Q_\op)$ involves critical points
  of $f_\up$, $\up\in\cP^a(H)$ such that
  $-\mu(\op)+\mu(\up)+2\langle c_1(T\widehat W),A\rangle=1$. It
  follows from the dimension formula~\eqref{eq:dimMS} that
  $\cM^A(Q_\op,Q_\up;H,\{f_p\},J,g)$ involves exactly one parametrized
  Floer trajectory $u_1\in\cM^A(S_\op,S_\up;H,J,g)$. Since the
  dimension of the moduli space $\cM^A(Q_\op,Q_\up;H,\{f_p\},J,g)$ is
  zero, it follows that either $\oev(u_1)=M_\op$ and $Q_\up=M_\up$, or
  $\uev(u_1)=m_\up$ and $Q_\op=m_\op$.

  Using that the $S^1$-action on $\cM^A(S_\op,S_\up;H,J,g)$ is free,
  we see that the coefficient of $Q_\up\, e^A$ in $\bar d^1(Q_\op)$ is
  given by an algebraic count of connected components of
  $\cM^A(S_\op,S_\up;H,J,g)$. The latter are in bijective
  correspondence with elements of
  $\cM^A_{S^1}(S_\op,S_\up;H,J,g)$, and the signs are the same
  by our convention for orienting the latter moduli space. Thus, the
  coefficient of $Q_\up\, e^A$ in $\bar d^1(Q_\op)$ is equal to the
  coefficient of $S_\up\, e^A$ in $\partial^{S^1}(S_\op)$. This proves
  the Proposition.
\end{proof}

\begin{proof}[Proof of Proposition~\ref{prop:partialS1}]
The claim $\partial^{S^1} \!\!\circ \, \partial^{S^1}=0$ follows directly
from Proposition~\ref{prop:d1}, using that $\bar d^1 \circ \bar
d^1=0$ (see~\eqref{eq:bard1}).
\end{proof}

It follows from Proposition~\ref{prop:d1} that $\Phi$ induces an
isomorphism which respects the bi-degree
\begin{equation} \label{eq:barPhi}
\bar \Phi : E^{a,N;2}_{*,*}(H) \stackrel \sim \to SH^{a,S^1,N}_*(H)
  \otimes H_*(S^1).
\end{equation}

We are now ready to prove Theorem~\ref{thm:SGysin}. We need two
preparatory Lemmas.

\begin{lemma} \label{lem:limit}
 We have $\lim_{N\to \infty} SH_*^{S^{2N+1}}(W) = SH_*(W)$ in
 each degree.
\end{lemma}

\begin{proof} The limit $\lim_{N\to \infty} SH_*^{S^{2N+1}}(W)$ is
 taken with respect to the maps $S\iota_*$ corresponding to the
 inclusions $\iota:S^{2N+1}\hookrightarrow S^{2N+3}$ as in
 Remark~\ref{rmk:SHincl}. Moreover, we saw that, modulo the K\"unneth
 isomorphism $SH_*^{S^{2N+1}}(W)\simeq SH_*(W)\otimes H_*(S^{2N+1})$
 proved in Proposition~\ref{prop:Kunneth}, the map $S\iota_*$ is equal to
 $\mathrm{Id}\otimes \iota_*$. The conclusion follows.
\end{proof}

\begin{lemma} \label{lem:continuation}
 Let $H_s$ be a smooth increasing homotopy from
 $H_0\in\cH^{S^1}_{N,\reg}$ to $H_1 \in\cH^{S^1}_{N,\reg}$. Let
 $(J_i,g_i)\in\cJ^{S^1}_{N,\reg}(H_i)$, $i=0,1$ and $(J_s,g_s)$ a
 regular smooth homotopy in $\cJ^{S^1}_N$ from $(J_0,g_0)$ to
 $(J_1,g_1)$. The induced chain morphism $BC^{a,S^1,N}_*(H_0,J_0,g_0)\to
 BC^{a,S^1,N}_*(H_1,J_1,g_1)$ respects the filtrations.
\end{lemma}

The proof of Lemma~\ref{lem:continuation} is given in
Section~\ref{sec:continuation} below.

\begin{proof}[Proof of Theorem~\ref{thm:SGysin}]
 Using the isomorphism $\bar \Phi$ in~\eqref{eq:barPhi}, the long
 exact sequence~\eqref{eq:seq} becomes
$$
... \to SH^{a,S^{2N+1}}_k\!(H) \to SH^{a,S^1,N}_k\!(H) \to
SH^{a,S^1,N}_{k-2}\!(H) \to SH^{a,S^{2N+1}}_{k-1}\!(H) \to ...
$$
By Lemma~\ref{lem:continuation}, a smooth increasing homotopy of
Hamiltonian families in $\cH^{S^1}_{N,\reg}$ induces a filtered chain
morphism, and therefore a commutative diagram of exact sequences.
Passing to the direct limit over $H\in\cH^{S^1}_{N,\reg}$ and using
that the direct limit functor is exact, we obtain a long exact
sequence
$$
... \to SH^{a,S^{2N+1}}_k\!(W) \to SH^{a,S^1,N}_k\!(W) \to
SH^{a,S^1,N}_{k-2}\!(W) \to SH^{a,S^{2N+1}}_{k-1}\!(W) \to ...
$$
Passing further to the direct limit over $N\to\infty$, and using
Lemma~\ref{lem:limit}, we obtain
$$
... \to SH^a_k(W) \to SH^{a,S^1}_k(W) \to
SH^{a,S^1}_{k-2}(W) \to SH^a_{k-1}(W) \to ...
$$
\end{proof}

We will now prove Theorem~\ref{thm:grid}. We need a preliminary
algebraic lemma involving the cone of a chain morphism. Recall that,
given a chain morphism $f:(A_*,\p_A)\to (A'_*,\p_{A'})$ of degree $0$, we
define the {\bf cone of \boldmath$f$} as the chain complex
$$
\mathcal{C}(f)_*:=A'_{*+1}\oplus A_*,
$$
with differential $\p$ given in matrix form by
$$
\p:=\left(\begin{array}{cc} \p_{A'} & f \\ 0 & \p_A \end{array}\right).
$$
There is a short exact sequence of complexes
\begin{equation} \label{eq:short}
\xymatrix{0\ar[r] & A'_{*+1} \ar[r]^i &  \mathcal{C}(f)_* \ar[r]^p &
  A_* \ar[r] & 0,
}
\end{equation}
with $i$, $p$ the obvious inclusion, respectively projection. The main
property of the cone construction is that the connecting homomorphism
in the homology long exact sequence associated to~\eqref{eq:short} is
precisely $f_*:H_*(A)\to H_*(A')$.

\begin{lemma} \label{lem:grid}
 Let
\begin{equation} \label{eq:fgh}
\xymatrix{
0\ar[r] & A_* \ar[r] \ar[d]^f & B_* \ar[r] \ar[d]^g & C_* \ar[r]
\ar[d]^h & 0 \\
0\ar[r] & A'_* \ar[r]  & B'_* \ar[r]  & C'_* \ar[r]
 & 0
}
\end{equation}
be a morphism of short exact sequences of complexes. This induces the
commutative diagram of homological long exact sequences
\begin{equation} \label{eq:big-grid}
\xymatrix
@C=20pt
@R=20pt
{
 & \vdots \ar[d] & \vdots \ar[d] & \vdots \ar[d] & \vdots \ar[d] & \\
\cdots \ar[r] & H_{*+1}(A') \ar[r] \ar[d] & H_{*+1}(B') \ar[r] \ar[d]
& H_{*+1}(C') \ar[r] \ar[d] & H_*(A') \ar[r] \ar[d] & \cdots \\
\cdots \ar[r] & H_{*+1}(\mathcal{C}(f)) \ar[r] \ar[d] &
H_{*+1}(\mathcal{C}(g)) \ar[r] \ar[d] & H_{*+1}(\mathcal{C}(h))
\ar[r] \ar[d] & H_*(\mathcal{C}(f)) \ar[r] \ar[d] & \cdots \\
\cdots \ar[r] & H_{*+1}(A) \ar[r] \ar[d]^{f_*} & H_{*+1}(B) \ar[r]
\ar[d]^{g_*} & H_{*+1}(C) \ar[r] \ar[d]^{h_*} & H_*(A) \ar[r]
\ar[d]^{f_*} & \cdots \\
\cdots \ar[r] & H_*(A') \ar[r] \ar[d] & H_*(B') \ar[r] \ar[d]
& H_*(C') \ar[r] \ar[d] & H_{*-1}(A') \ar[r] \ar[d] & \cdots \\
& \vdots & \vdots & \vdots & \vdots &
}
\end{equation}
\end{lemma}

\begin{proof}
Applying the cone construction to each column of~\eqref{eq:fgh} we
obtain the short exact sequence of short exact sequences of complexes
\begin{equation} \label{eq:cone-fgh}
\xymatrix
@C=25pt
@R=15pt
{
& 0 \ar[d] & 0 \ar[d] & 0 \ar[d] & \\
0 \ar[r] & A'_{*+1} \ar[d] \ar[r] & B'_{*+1} \ar[d] \ar[r] & C'_{*+1}
\ar[d] \ar[r] & 0 \\
0 \ar[r] & \mathcal{C}(f)_* \ar[d] \ar[r] & \mathcal{C}(g)_* \ar[d]
\ar[r] & \mathcal{C}(h)_* \ar[d] \ar[r] & 0 \\
0 \ar[r] & A_* \ar[d] \ar[r] & B_* \ar[d] \ar[r] & C_*
\ar[d] \ar[r] & 0 \\
& 0 & 0 & 0 &
}
\end{equation}
The lines/columns in~\eqref{eq:big-grid} are obtained as homological long exact
sequences associated to the horizontal/vertical short exact sequences
in~\eqref{eq:cone-fgh}. Commutativity of the horizontal strips
in~\eqref{eq:big-grid} follows from functoriality of the homological
long exact sequence with respect to morphisms of short exact sequences
of complexes. More precisely, for the first two horizontal strips
in~\eqref{eq:big-grid} we use~\eqref{eq:cone-fgh}, and for the third
horizontal strip in~\eqref{eq:big-grid} we use~\eqref{eq:fgh}.
\end{proof}

\begin{proof}[Proof of Theorem~\ref{thm:grid}]
Given $H\in\cH^{\prime\, S^1}_{N,\reg}$, $(J,g)\in
\cJ^{S^1}_{N,\reg}(H)$, and a generic collection of perfect Morse
functions $f_p:S_p\to \R$, $p\in\cP^0(H)$, we denote
$C_*:=BC^N_*(H,\{f_p\},J,g)$ (we recall that we work in the
trivial free homotopy class). Filtering by the action as in the
definition of $SC^{\pm,S^1,N}_*(H,J,g)$ in
Section~\ref{sec:S1equivsymplhom}, we obtain filtered complexes
$C^\pm_*:=BC^{\pm,N}_*(H,\{f_p\},J,g)$. We denote by
$(E^{\pm,N;r}_{*,*},\bar d^r)$ the corresponding spectral sequences,
which degenerate at $r=3$ for dimensional reasons.

Since $d^0=0$, we have a short exact sequence
$0\to (E^{-,1}_{*,*},\bar d) \to
(E^1_{*,*},\bar d) \to (E^{+,1}_{*,*},\bar d) \to 0$ with $\bar d=\bar
d^1+\bar d^2$. This can be rewritten as a morphism of short exact
sequences of chain complexes
\begin{equation} \label{eq:bard2}
\xymatrix{
0 \ar[r] & (E^{-,1}_{*,0},\bar d^1) \ar[r] \ar[d]^{\bar d^2} &
(E^1_{*,0},\bar d^1)
\ar[r] \ar[d]^{\bar d^2} & (E^{+,1}_{*,0},\bar d^1) \ar[r] \ar[d]^{\bar d^2} & 0
\\
0 \ar[r] & (E^{-,1}_{*-2,1},\bar d^1) \ar[r] & (E^1_{*-2,1},\bar d^1)
\ar[r] & (E^{+,1}_{*-2,1},\bar d^1) \ar[r] & 0
}
\end{equation}

We claim that the commutative diagram~\eqref{eq:grid} in the statement
is obtained by applying Lemma~\ref{lem:grid} to~\eqref{eq:bard2}.
This follows from the following three observations. Firstly, the cone
$\mathcal{C}(\bar d^2)$ is canonically identified with
$(C,\bar d)$, respectively
$(C^\pm,\bar d)$, so that its homology is
$H_*(BC^N_*(H),d)$, respectively $H_*(BC^{\pm,N}_*(H),d)$. Secondly,
the homology of $(E^{N;1}_{*,i},\bar d^1)$, $i=0,1$ is isomorphic to
$SH^{S^1,N}_*(H)$, and the homology of $(E^{\pm, N;1}_{*,i},\bar
d^1)$, $i=0,1$ is
isomorphic to $SH^{\pm,S^1,N}_*(H)$. Thirdly, a straightforward
verification shows that, via the above identifications with the cone
$\mathcal{C}(\bar d^2)$, the Gysin exact sequences obtained from the
spectral sequences $E^{N;r}_{*,*}$ and $E^{\pm,N;r}_{*,*}$ coincide
with the homological long exact sequences of the corresponding cone
constructions.

Passing to the direct limit on $H\in\cH^{\prime\, S^1}_{N,\reg}$ and
$N\to\infty$ we obtain the commutative diagram~\eqref{eq:grid}.
\end{proof}

\begin{remark} \label{rmk:Delta} Denoting the maps in the Gysin exact sequence by
$$
\xymatrix
@C=20pt
{
\dots \ar[r] & SH_*(W) \ar[r]^E & SH_*^{S^1}(W) \ar[r]^D &
SH_{*-2}^{S^1}(W) \ar[r]^M & SH_{*-1}(W) \ar[r] & \dots
}
$$
we defined in the Introduction the Batalin-Vilkovisky (BV) operator
$$
\Delta:=M\circ E:SH_*(W)\to SH_{*+1}(W).
$$
The above interpretation of the Gysin exact sequence as the long exact sequence of the cone $C_*:=\cC(\bar d^2)$ allows us to give the following description of $\Delta$ at chain level. We identify $C_*=BC_*^N(H,\{f_p\},J,g)$ with $SC_{*-1}^{S^1}\oplus SC_*^{S^1}:=SC_{*-1}^{S^1,N}(H,J,g)\oplus SC_*^{S^1,N}(H,J,g)$ via
$$
m_p\longmapsto (S_p,0), \qquad M_p\longmapsto (0,S_p).
$$
Via this identification, the map $\Delta$ is induced by the chain map $\bar\Delta :C_*\to C_{*+1}$ given by
$$
\bar \Delta : SC_{*-1}^{S^1}\oplus SC_*^{S^1}\longrightarrow SC_*^{S^1}\oplus SC_{*+1}^{S^1},
$$
$$
(S_p,S_q) \longmapsto (S_q,0).
$$
Indeed, the short exact sequence of the cone $\cC(\bar d^2)=SC_{*-1}^{S^1}\oplus SC_*^{S^1}$ writes
$$
0\to SC_{*-1}^{S^1}\stackrel i \to SC_{*-1}^{S^1}\oplus SC_*^{S^1}\stackrel p \to SC_*^{S^1}\to 0.
$$
The maps $i$ and $p$ are the canonical inclusion and projection. The connecting homomorphism in the homological long exact sequence is the map $D$, so that the maps
$i$ and $p$ induce $M$ and $E$ respectively. Hence the composition $\bar\Delta=i\circ p$ induces $\Delta=M\circ E$.

\end{remark}

\subsection{Filtered continuation maps for parametrized symplectic homology} \label{sec:continuation}

Let $H_s$, $s\in\R$  be a smooth increasing homotopy from $H_-\in\cH^{S^1}_{N,\mathrm{reg}}$ to $H_+\in\cH^{S^1}_{N,\mathrm{reg}}$, such that $H_s\equiv H_-$ for $s<<0$ and $H_s\equiv H_+$ for $s>>0$. Let $(J_\pm,g_\pm)\in\cJ^{S^1}_{N,\reg}(H_\pm)$ and $(J_s,g_s)$, $s\in\R$ a
 regular smooth homotopy in $\cJ^{S^1}_N$ from $(J_-,g_-)$ to
 $(J_+,g_+)$, which is constant near $\pm\infty$. Given
$\op\in\cP(H_-)$ and $\up\in\cP(H_+)$, we define the {\bf moduli space
of \boldmath$s$-dependent \boldmath$S^1$-equivariant Floer
trajectories} $\cM(S_\op,S_\up;H_s,J_s,g_s)$ to consist of pairs
$(u,\lambda)$ with
 $$
 u:\R\times S^1\to \widehat W, \qquad \lambda:\R\to S^{2N+1}
 $$
satisfying
\begin{equation} \label{eq:Floer-u-s}
\p_s u + J^\theta_{s,\lambda(s)}\p_\theta u - J^\theta_{s,\lambda(s)}X^\theta_{H_{s,\lambda(s)}}(u)=0,
\end{equation}
\begin{equation} \label{eq:Floer-lambda-s}
\dot \lambda(s) -\int_{S^1} \vec\nabla_\lambda H_s(\theta,u(s,\theta),\lambda(s))d\theta=0,
\end{equation}
and
\begin{equation} \label{eq:Floer-asy-s}
\lim_{s\to-\infty} (u(s,\cdot),\lambda(s))\in S_{\op}, \qquad \lim_{s\to+\infty} (u(s,\cdot),\lambda(s))\in S_{\up}.
\end{equation}
Due to the $s$-dependence, the additive group $\R$ does not act on the moduli space $\cM(S_\op,S_\up;H_s,J_s,g_s)$.
Recall that, for each $p=(\gamma,\lambda)\in\cP(H_\pm)$, we have chosen a
cylinder $\sigma_p:[0,1]\times S^1\to\widehat W$ such that
$\sigma_p(0,\cdot)=l_{[\gamma]}$ and $\sigma_p(1,\cdot)=\gamma$. We
define $\overline \sigma_p(s,\theta):=\sigma_p(1-s,\theta)$. We define
$$
\cM^A(S_\op,S_\up;H_s,J_s,g_s)\subset \cM(S_\op,S_\up;H_s,J_s,g_s)
$$
to consist of trajectories $(u,\lambda)$ such that
$[\sigma_{\op}\#u\#\overline \sigma_{\up}]=A\in H_2(\widehat
W;\Z)$. It follows from Proposition~\ref{prop:indexMB} that
\begin{equation}  \label{eq:dim-MB-s}
\dim\, \cM^A(S_\op,S_\up;H_s,J_s,g_s) = -\mu(\op) +\mu(\up) +2\langle
c_1(T\widehat W),A\rangle+1.
\end{equation}

For each $S^1$-orbit of critical
points $S_p\subset \cP(H_\pm)$ we choose a perfect Morse function
$f^\pm_p:S_p\to\R$. We denote by $m_p$, $M_p$ the minimum, respectively
the maximum of $f^\pm_p$. Given $\op\in\cP(H_-)$, $\up\in\cP(H_+)$,
$Q_\op\in\mathrm{Crit}(f^-_\op)$, $Q_\up\in\mathrm{Crit}(f^+_\up)$, $A\in H_2(W;\Z)$, and
$m_\pm\ge 0$, we denote by
$$
\cM^A_{m_-,m_+}(Q_\op,Q_\up;H_s,\{f^\pm_p\},J_s,g_s)
$$
the union for $p^-_1,\dots,p^-_{m_-}\in\cP(H_-)$, $p^+_1,\dots,p^+_{m_+}\in\cP(H_+)$, and $A^-_1+\dots+A^-_{m_-}+A^0+A^+_1+\dots+A^+_{m_+}=A$ of
the fibered products
\begin{eqnarray*}
&&
W^u(Q_\op)
\times_{\oev}
(\cM^{A^-_1}(S_{\op}\,,S_{p^-_1};H_-,J_-,g_-)\!\times\!\R^+)\\
&& {_{\varphi_{f^-_{p^-_1}}\!\circ\uev}}\!\times
_{\oev}\dots {_{\varphi_{f^-_{p^-_{m_- -1}}}\!\circ\uev}}\!\times
_{\oev} (\cM^{A^-_{m_-}}(S_{p^-_{m_- -1}},S_{p^-_{m_-}};H_-,J_-,g_-)\!\times\!\R^+) \\
&& {_{\varphi_{f^-_{p^-_{m_-}}}\!\circ\uev}\times_{\oev}}
(\cM^{A^0}(S_{p^-_{m_-}},S_{p^+_1};H_s,J_s,g_s)\!\times\!\R^+)\\
&& {_{\varphi_{f^+_{p^+_1}}\!\circ\uev}}\!\times
_{\oev} (\cM^{A^+_1}(S_{p^+_1},S_{p^+_2};H_+,J_+,g_+)\!\times\!\R^+) \\
&&{_{\varphi_{f^+_{p^+_2}}\!\circ\uev}}\!\times
_{\oev}\dots
{_{\varphi_{f^+_{p^+_{m_+}}}\!\!\circ\uev}}\!\!\times
_{\oev}
\cM^{A^+_{m_+}}(S_{p^+_{m_+}},\!S_{\up})
{_{\uev}\times} W^s(Q_\up).
\end{eqnarray*}
It follows from~\cite[Lemma~3.6]{BOauto} that, for a generic choice of
the collection of Morse functions $\{f^\pm_p\}$, the previous fibered
product is a smooth manifold of dimension
\begin{eqnarray*}
\lefteqn{\dim \, \cM^A_{m_-,m_+}(Q_\op,Q_\up;H_s,\{f^\pm_p\},J_s,g_s)} \\
& = &
-\mu(\op)+\ind_{f^-_\op}(Q_\op) + \mu(\up) - \ind_{f^+_\up}(Q_\up) +
2\langle c_1(T\widehat W),A\rangle.
\end{eqnarray*}
We denote
$$
\cM^A(Q_\op,Q_\up;H_s,\{f^\pm_p\},J_s,g_s):=\bigcup_{m_\pm\ge 0}  \cM^A_{m_-,m_+}(Q_\op,Q_\up;H_s,\{f^\pm_p\},J_s,g_s).
$$
Whenever $\dim\, \cM^A(Q_\op,Q_\up;H_s,\{f^\pm_p\},J_s,g_s)=0$, we can associate a sign $\eps(\u)$ to each of its elements via the choice of coherent orientations and the fibered sum rule~\cite[Section~4.4]{BOauto}. We define the {\bf continuation morphism}
$$
\sigma_{H_+,H_-}:BC_*^{a,N}(H_-)\to BC_*^{a,N}(H_+)
$$
by
$$
\sigma_{H_+,H_-}(Q_\op):=\sum_{\substack{
  \up\in \cP^a(H_+), Q_\up\in \mathrm{Crit}(f^+_{\up}) \\
|Q_\op| - |Q_\up\, e^A|=0}}
\ \sum_{\scriptstyle \u\in \cM^A(Q_\op,Q_\up;H_s,\{f^\pm_\gamma\},J_s,g_s)}
\epsilon(\u)Q_\up\, e^A,
$$
for all $\op\in\cP(H_-)$ and $Q_\op \in {\rm Crit}(f^-_\op)$. In order to emphasize the homotopy used to
define $\sigma_{H_+,H_-}$, we shall sometimes write $\sigma_{H_+,H_-}^{H_s}$.

\begin{proof}[Proof of Lemma~\ref{lem:continuation}]
That the map $\sigma_{H_+,H_-}$ is a chain morphism satisfying $\sigma_{H_+,H_-}\circ d=d\circ \sigma_{H_+,H_-}$ follows from a straightforward generalization of the Correspondence Theorem~3.7 in~\cite{BOauto}.  Via the identification of the parametrized Morse-Bott complexes with the Floer complexes of suitable perturbations of the Hamiltonians $H_\pm$, the morphism $\sigma_{H_+,H_-}$ corresponds to the continuation morphism induced by an increasing homotopy of Hamiltonians.

That $\sigma_{H_+,H_-}$ preserves the filtration follows from the fact that each of the moduli spaces
$\cM^{A^0}(S_{p^-_{m_-}},S_{p^+_1};H_s,J_s,g_s)$, $\cM^{A^-_j}(S_{p^-_{i-1}},S_{p^-_i};H_-,J_-,g_-)$, $1\le i\le m_-$ and
$\cM^{A^+_i}(S_{p^+_i},S_{p^+_{i+1}};H_+,J_+,g_+)$, $1\le i \le m_+$ carries a free $S^1$-action (we denote $p^-_0=\op$, $p^+_{m_+ +1}=\up$). In case they are nonempty, their dimension is therefore at least $1$. It then follows from the dimension formulas~\eqref{eq:dim-MB-s}
and~\eqref{eq:dimMS} that $|\op|-|\up e^A|=-\mu(\op)+\mu(\up)+2\langle c_1(T\widehat W),A\rangle\ge 0$.
\end{proof}

For the next statement it is useful to introduce the following
algebraic concept. Let $(C_*,d_C)$ and $(D_*,d_D)$ be differential
complexes endowed with increasing filtrations $F_\ell C_*$, $F_\ell
D_*$, $\ell\in \Z$. A map
$K:C_*\to D_*$ is said to be {\bf of order \boldmath$k\ge 0$} if
$K(F_\ell C_*)\subset F_{\ell+k}D_*$ (we allow $K$ to shift the
grading). This definition is relevant in the following context. Assume
$f,g:C_*\to D_*$ are filtration preserving chain maps such that
$f-g=d_D \circ K + K\circ d_C$ for a chain homotopy $K:C_*\to D_{*+1}$
of order $k\ge 0$. Then the maps $f_r, g_r$, $r\ge 0$ induced on the
associated spectral sequences are homotopic for $r=k$, and coincide
for $r>k$~\cite[Exercise~3.8, p.87]{McC}.

\begin{proposition} \label{prop:chain-homotopy}
Let $H_-\le H_+$ be Hamiltonians in $\cH^{S^1}_{N,\mathrm{reg}}$ and $H_s^0, H_s^1\in \cH^{S^1}_N$, $s\in\R$ be generic smooth increasing homotopies from $H_-$ to $H_+$, which are constant near $\pm\infty$. Let $(J_\pm,g_\pm)\in\cJ^{S^1}_{N,\reg}(H_\pm)$ and $(J^0_s,g^0_s)$, $(J^1_s,g^1_s)$ be two generic  smooth homotopies in $\cJ^{S^1}_N$ from $(J_-,g_-)$ to
 $(J_+,g_+)$, which are constant near $\pm\infty$. A generic homotopy of homotopies $(H_s^\rho,J^\rho_s,g^\rho_s)$, $\rho\in [0,1]$ induces a map $K:BC_*^{a,N}(H_-)\to BC_{*+1}^{a,N}(H_+)$ of order $1$ such that
$$
\sigma_{H_+,H_-}^{H_s^1} - \sigma_{H_+,H_-}^{H_s^0} = d\circ K+K\circ d.
$$
\end{proposition}

\begin{proof} Given $\op\in\cP(H_-)$, $\up\in \cP(H_+)$, $Q_\op\in\mathrm{Crit}(f^-_\op)$, $Q_\up\in\mathrm{Crit}(f^+_\up)$, and $A\in H_2(W;\Z)$ such that
$$
|Q_\op|-|Q_\up e^A|=0,
$$
we define
$$
\cM^A:=\bigcup_{\rho\in [0,1]} \cM^A(Q_\op,Q_\up;H^\rho_s, \{f^\pm_p\},J^\rho_s, g^\rho_s).
$$
For a generic choice of the triple $(H_s^\rho,J^\rho_s,g^\rho_s)$, $\rho\in [0,1]$, the space $\cM^A$ is a smooth $1$-dimensional manifold. Its boundary splits as
$$
\p \cM^A= \p^0\cM^A \cup \p^1\cM^A \cup \p^{int}\cM^A.
$$
Here $\p^i\cM^A=\cM^A(Q_\op,Q_\up;H^i_s, \{f^\pm_p\},J^i_s, g^i_s)$, $i=0,1$ and $\p^{int}\cM^A$ corresponds to degeneracies at some point $\rho\in]0,1[$, namely
\begin{eqnarray*}
\lefteqn{\p^{int}\cM^A}\\
 &\hspace{-.3cm} = & \!\!\bigcup_{\rho\in ]0,1[} \cM^B(Q_\op,Q_{p_-};H_-,\{f^-_p\},J_-,g_-) \times \cM^{A-B}(Q_{p_-},Q_\up;H^\rho_s, \{f^\pm_p\},J^\rho_s, g^\rho_s) \\
& \hspace{-.3cm}& \hspace{-.5cm}\cup \bigcup_{\rho\in ]0,1[}\!\!\cM^{A-B}(Q_\op,Q_{p_+};H^\rho_s, \{f^\pm_p\},J^\rho_s, g^\rho_s) \times \cM^B(Q_{p_+},Q_\up;H_+,\{f^+_p\},J_+,g_+).
\end{eqnarray*}
Here the union is taken over $B\in H_2(W;\Z)$, $p_\pm\in\cP(H_\pm)$, $Q_{p_\pm}\in\mathrm{Crit}(f^\pm_{p_\pm})$ such that $|Q_{p_-}|-|Q_\up e^{A-B}|=-1$ and $|Q_\op|-|Q_{p_+} e^{A-B}|=-1$. For a generic choice of the triple $(H_s^\rho,J^\rho_s,g^\rho_s)$, $\rho\in [0,1]$, there are only a finite number of values of $\rho$ involved in the above union. The elements of $\p^{int}\cM^A$ correspond to the breaking of a gradient trajectory involved in one of the fiber products defining $\cM^A(Q_\op,Q_\up;H^\rho_s, \{f^\pm_p\},J^\rho_s, g^\rho_s)$, as $\rho$ converges to some $\rho_0\in]0,1[$. There are yet two other types of degeneracy in $\cM^A$, which compensate each other: the length of a gradient trajectory in a fibered product as above can shrink to zero, or a Floer trajectory can break at a point $Q\in S_p\setminus \mathrm{Crit}(f^\pm_p)$, for some $p\in\cP(H_\pm)$.

We define $K:BC_*^{a,N}(H_-)\to BC_{*+1}^{a,N}(H_+)$ by
$$
K(Q_\op)=\sum_{\rho\in]0,1[} \ \sum_{|Q_\op|-|Q_\up e^A|=-1} \ \sum_{\u\in\cM^A(Q_\op,Q_\up;H^\rho_s, \{f^\pm_p\},J^\rho_s, g^\rho_s)} \eps(\u)Q_\up e^A.
$$
The above description of $\p\cM^A$ shows that we have indeed $\sigma_{H_+,H_-}^{H_s^1} - \sigma_{H_+,H_-}^{H_s^0} = d\circ K+K\circ d$. That the chain homotopy $K$ is of order $1$ means that it satisfies $K(F_\ell B_*^{a,N}(H_-))\subset F_{\ell+1} B_{*+1}^{a,N}(H_+)$. This follows from the fact that each family of moduli spaces $\bigcup_{\rho\in]0,1[} \cM^{A^0}(Q_{p_-},Q_{p_+};H^\rho_s, \{f^\pm_p\},J^\rho_s, g^\rho_s)$
carries a free $S^1$-action. In case it is nonempty, its dimension must therefore be at least $1$. On the other hand, it follows from~\eqref{eq:dim-MB-s} that this dimension is equal to $|p_-|-|p_+ e^{A^0}|+2$, so that $|p_-|-|p_+ e^{A^0}|\ge -1$. A similar argument shows that, for the moduli spaces $\cM^{A^\pm}(Q_{p^\pm_0},Q_{p^\pm_1};H_\pm,\{f^\pm_p\},J_\pm,g_\pm)$ appearing in the definition of $K$, we must have $|p^\pm_0|-|p^\pm_1 e^{A^\pm}|\ge 0$. Thus, for the fibered products appearing in the definition of $K$ we have $|\op|-|\up e^A|\ge -1$.
\end{proof}

\begin{proposition} \label{prop:composition}
Let $H_0\le H_1\le H_2$ be three Hamiltonians in $\cH^{S^1}_{N,\mathrm{reg}}$, and let $H_s^{01}, H_s^{12}\in  \cH^{S^1}_N$, $s\in\R$ be two generic smooth increasing homotopies from $H_0$ to $H_1$, respectively from $H_1$ to $H_2$, which are constant near $\pm\infty$. Let $(J_i,g_i)\in\cJ^{S^1}_{N,\reg}(H_i)$, $i=0,1,2$ and $(J^{01}_s,g^{01}_s)$, $(J^{12}_s,g^{12}_s)$ be two generic  smooth homotopies in $\cJ^{S^1}_N$ from $(J_0,g_0)$ to
 $(J_1,g_1)$, respectively from $(J_1,g_1)$ to $(J_2,g_2)$, which are constant near $\pm\infty$.
For $R>0$ sufficiently large we denote
$$
H^{02,R}_s:=\left\{\begin{array}{ll}
H^{01}_{s+R}, & s\le 0, \\
H^{12}_{s-R}, & s\ge 0.
\end{array}\right.
$$
We define the homotopies $J^{02,R}_s, g^{02,R}_s$ in a similar way. There exists a map $K:BC_*^{a,N}(H_0)\to BC_{*+1}^{a,N}(H_2)$ of order $1$ such that
$$
\sigma_{H_2,H_1}^{H_s^{12}}\circ \sigma_{H_1,H_0}^{H_s^{01}} - \sigma_{H_2,H_0}^{H^{02,R}_s} = d\circ K+K\circ d.
$$
\end{proposition}

\begin{proof} Let $\{f^i_p\}$, $i=0,1,2$ be three generic collections of perfect Morse functions on $S_p$, for $p\in\cP(H_i)$ respectively. Given $\op\in\cP(H_0)$, $\up\in\cP(H_2)$, $Q_\op\in\mathrm{Crit}(f^0_\op)$, $Q_\up\in\mathrm{Crit}(f^2_\up)$, $A\in H_2(W;\Z)$ such that $|Q_\op|-|Q_\up e^A|=0$, we define for $R_0>0$ sufficiently large the family of moduli spaces
$$
\cM^A_1:=\bigcup_{R\ge R_0} \cM^A(Q_\op,Q_\up;H^{02,R}_s, \{f^0_p,f^2_p\},J^{02,R}_s,g^{02,R}_s).
$$
For a generic choice of the homotopies, this is a smooth $1$-dimensional manifold. Its boundary splits as
$$
\p\cM^A_1=\p^{R_0}\cM^A_1 \cup \p^\infty\cM^A_1\cup \p^{int}\cM^A_1.
$$
Here $\p^{R_0}\cM^A_1=\cM^A(Q_\op,Q_\up;H^{02,R_0}_s, \{f^0_p,f^2_p\},J^{02,R_0}_s,g^{02,R_0}_s)$.
We now describe $\p^\infty\cM^A_1$, which corresponds to degenerations as $R\to\infty$. Let $p\in\cP(H_1)$, $m_0\ge 0$, $B\in H_2(W;\Z)$, and define
$\cM^B_{m_0}(Q_\op,S_p;H^{01}_s, \{f^0_p\},J^{01}_s,g^{01}_s)$ as the union for
$p^0_1,\dots,p^0_{m_0}\in\cP(H_0)$ and $A^0_1+\dots+A^0_{m_0}+A^{01}=B$ of the fibered products
\begin{eqnarray*}
&&
W^u(Q_\op)
\times_{\oev}
(\cM^{A^0_1}(S_{\op}\,,S_{p^0_1};H_0,J_0,g_0)\!\times\!\R^+)\\
&& {_{\varphi_{f^0_{p^0_1}}\!\circ\uev}}\!\times
_{\oev}\dots {_{\varphi_{f^0_{p^0_{m_0 -1}}}\!\circ\uev}}\!\times
_{\oev} (\cM^{A^0_{m_0}}(S_{p^0_{m_0 -1}},S_{p^0_{m_0}};H_0,J_0,g_0)\!\times\!\R^+) \\
&& {_{\varphi_{f^0_{p^0_{m_0}}}\!\circ\uev}\times_{\oev}}
\cM^{A^{01}}(S_{p^0_{m_0}},S_p;H^{01}_s,J^{01}_s,g^{01}_s).
\end{eqnarray*}
We define $\cM^B(Q_\op,S_p;H^{01}_s, \{f^0_p\},J^{01}_s,g^{01}_s)$ as the union over $m_0\ge 0$ of the moduli spaces $\cM^B_{m_0}(Q_\op,S_p;H^{01}_s, \{f^0_p\},J^{01}_s,g^{01}_s)$. This is a smooth manifold of dimension
$$
\dim\, \cM^B(Q_\op,S_p;H^{01}_s, \{f^0_p\},J^{01}_s,g^{01}_s)=|Q_\op|-|pe^B|.
$$
Given $p\in\cP(H_1)$, $m_2\ge 0$, $B\in H_2(W;\Z)$, we define
the moduli space $\cM^B_{m_2}(S_p,Q_\up;H^{12}_s, \{f^2_p\},J^{12}_s,g^{12}_s)$ as the union for
$p^2_1,\dots,p^2_{m_2}\in\cP(H_2)$ and $A^{12}+A^2_1+\dots+A^2_{m_2}=B$ of the fibered products
\begin{eqnarray*}
&&
(\cM^{A^{12}}(S_p,S_{p^2_1};H^{12}_s,J^{12}_s,g^{12}_s)\!\times\!\R^+)\\
&& {_{\varphi_{f^2_{p^2_1}}\!\circ\uev}}\!\times
_{\oev} (\cM^{A^2_1}(S_{p^2_1},S_{p^2_2};H_2,J_2,g_2)\!\times\!\R^+) \\
&&{_{\varphi_{f^2_{p^2_2}}\!\circ\uev}}\!\times
_{\oev}\dots
{_{\varphi_{f^2_{p^2_{m_2}}}\!\!\circ\uev}}\!\!\times
_{\oev}
\cM^{A^2_{m_2}}(S_{p^2_{m_2}},\!S_{\up})
{_{\uev}\times} W^s(Q_\up).
\end{eqnarray*}
We define $\cM^B(S_p,Q_\up;H^{12}_s, \{f^2_p\},J^{12}_s,g^{12}_s)$ as the union over $m_2\ge 0$ of the moduli spaces $\cM^B_{m_2}(S_p,Q_\up;H^{12}_s, \{f^2_p\},J^{12}_s,g^{12}_s)$. This is a smooth manifold of dimension
$$
\dim\, \cM^B(S_p,Q_\up;H^{12}_s, \{f^2_p\},J^{12}_s,g^{12}_s)=|p|-|Q_\up e^B|+1.
$$
The boundary $\p^\infty\cM^A_1$ is then equal to
$$
\bigcup_{\substack{p\in\cP(H_1)\\ B_0+B_2=A}}
\hspace{-.4cm}\cM^{B_0}(Q_\op,S_p;H^{01}_s, \{f^0_p\},J^{01}_s,g^{01}_s)
{_{\uev}\times_{\oev}}
\cM^{B_2}(S_p,Q_\up;H^{12}_s, \{f^2_p\},J^{12}_s,g^{12}_s).
$$
The boundary $\p^{int}\cM^A_1$ corresponds to degeneracies at a point $R\in]R_0,\infty[$, namely
\begin{eqnarray*}
\lefteqn{\p^{int}\cM^A_1}\\
 &\hspace{-.4cm} = & \hspace{-.4cm}\bigcup_{R> R_0} \hspace{-.1cm}\cM^{B_0}(Q_\op,Q_{p_0};H_0,\{f^0_p\},J_0,g_0) \hspace{-.1cm}\times\hspace{-.1cm} \cM^{B_2}(Q_{p_0},Q_\up;H^{02,R}_s, \{f^i_p\},J^{02,R}_s, g^{02,R}_s) \\
& \hspace{-.4cm}& \hspace{-.7cm}\cup\hspace{-.1cm} \bigcup_{R> R_0}\hspace{-.1cm}\cM^{B_0}(Q_\op,Q_{p_2};H^{02,R}_s, \{f^i_p\},J^{02,R}_s, g^{02,R}_s) \hspace{-.1cm}\times\hspace{-.1cm} \cM^{B_2}(Q_{p_2},Q_\up;H_2,\{f^2_p\},J_2,g_2).
\end{eqnarray*}
Here we used the shortcut notation $\{f^i_p\}=\{f^0_p,f^2_p\}$, and the union is taken over $B_0+B_2=A$, $p_i\in\cP(H_i)$, $Q_{p_i}\in\mathrm{Crit}(f^i_{p_i})$, $i=0,2$, such that $|Q_{p_0}|-|Q_\up e^{A-B}|=-1$ and $|Q_\op|-|Q_{p_2} e^{A-B}|=-1$. For a generic choice of the triple $(H^{02,R}_s,J^{02,R}_s,g^{02,R}_s)$, $R\ge R_0$, there are only a finite number of values of $R$ involved in the above union. The elements of $\p^{int}\cM^A_1$ correspond to the breaking of a gradient trajectory involved in one of the fiber products defining $\cM^A(Q_\op,Q_\up;H^{02,R}_s, \{f^0_p,f^2_p\},J^{02,R}_s, g^{02,R}_s)$, as $R$ converges to some $R_{int}\in]R_0,\infty[$. There are yet two other types of degeneracy in $\cM^A_1$, which compensate each other: the length of a gradient trajectory in a fibered product as above can shrink to zero, or a Floer trajectory can break at a point $Q\in S_p\setminus \mathrm{Crit}(f^i_p)$, for some $p\in\cP(H_i)$, $i=0,2$.
We define a map $K_1:BC_*^{a,N}(H_0)\to BC_{*+1}^{a,N}(H_2)$ by
$$
K_1(Q_\op)=\sum_{R>R_0} \ \sum_{|Q_\op|-|Q_\up e^A|=-1} \ \sum_{\u\in\cM^A(Q_\op,Q_\up;H^{02,R}_s, \{f^0_p, f^2_p\},J^{02,R}_s, g^{02,R}_s)} \eps(\u)Q_\up e^A.
$$
The same argument as in the proof of Proposition~\ref{prop:chain-homotopy} shows that $K_1$ is of order $1$.
The previous description of $\p\cM^A_1$ can be summarized by saying that $d\circ K_1+K_1\circ d+\sigma_{H_2,H_0}^{H^{02,R_0}_s}$
is equal to the chain map obtained by the count of elements in $\p^\infty\cM^A_1$.

We now exhibit another $1$-dimensional moduli space whose boundary contains $\p^\infty \cM^A_1$.
Given $\op\in\cP(H_0)$, $\up\in\cP(H_2)$,
$Q_\op\in\mathrm{Crit}(f^0_\op)$, $Q_\up\in\mathrm{Crit}(f^2_\up)$, $A\in H_2(W;\Z)$, and
$m_0,m_1,m_2\ge 0$, we denote by
$$
\cM^A_{m_0,m_1,m_2}(Q_\op,Q_\up;H^{ij}_s,\{f^k_p\},J^{ij}_s,g^{ij}_s)
$$
the union for $p^0_1,\dots,p^0_{m_0}\in\cP(H_0)$, $p^1_1,\dots,p^1_{m_1+1}\in\cP(H_1)$,
$p^2_1,\dots,p^2_{m_2}\in\cP(H_2)$, and $A^0_1+\dots+A^0_{m_0}+A^{01}+A^1_1+\dots+A^1_{m_1}+A^{12}+A^2_1+\dots+A^2_{m_2}=A$ of the fibered products
\begin{eqnarray*}
&&
W^u(Q_\op)
\times_{\oev}
(\cM^{A^0_1}(S_{\op}\,,S_{p^0_1};H_0,J_0,g_0)\!\times\!\R^+)\\
&& {_{\varphi_{f^0_{p^0_1}}\!\circ\uev}}\!\times
_{\oev}\dots {_{\varphi_{f^0_{p^0_{m_0 -1}}}\!\circ\uev}}\!\times
_{\oev} (\cM^{A^0_{m_0}}(S_{p^0_{m_0 -1}},S_{p^0_{m_0}};H_0,J_0,g_0)\!\times\!\R^+) \\
&& {_{\varphi_{f^0_{p^0_{m_0}}}\!\circ\uev}\times_{\oev}}
(\cM^{A^{01}}(S_{p^0_{m_0}},S_{p^1_1};H^{01}_s,J^{01}_s,g^{01}_s)\!\times\!\R^+)\\
&& {_{\varphi_{f^1_{p^1_1}}\!\circ\uev}}\!\times_{\oev}
(\cM^{A^1_1}(S_{p^1_1}\,,S_{p^1_1};H_1,J_1,g_1)\!\times\!\R^+)\\
&& {_{\varphi_{f^1_{p^1_1}}\!\circ\uev}}\!\times
_{\oev}\dots {_{\varphi_{f^1_{p^1_{m_1}}}\!\circ\uev}}\!\times
_{\oev} (\cM^{A^1_{m_1}}(S_{p^1_{m_1}},S_{p^1_{m_1 +1}};H_1,J_1,g_1)\!\times\!\R^+) \\
&& {_{\varphi_{f^1_{p^1_{m_1 +1}}}\!\circ\uev}}\times_{\oev}
(\cM^{A^{12}}(S_{p^1_{m_1 +1}},S_{p^2_1};H^{12}_s,J^{12}_s,g^{12}_s)\!\times\!\R^+)\\
&& {_{\varphi_{f^2_{p^2_1}}\!\circ\uev}}\!\times_{\oev}
(\cM^{A^2_1}(S_{p^2_1},S_{p^2_2};H_2,J_2,g_2)\!\times\!\R^+) \\
&&{_{\varphi_{f^2_{p^2_2}}\!\circ\uev}}\!\times
_{\oev}\dots
{_{\varphi_{f^2_{p^2_{m_2}}}\!\!\circ\uev}}\!\!\times
_{\oev}
\cM^{A^2_{m_2}}(S_{p^2_{m_2}},\!S_{\up})
{_{\uev}\times} W^s(Q_\up).
\end{eqnarray*}
In the above notation we abridged $H^{ij}_s=\{H^{01}_s,H^{12}_s\}$ (similarly for $J^{ij}_s,g^{ij}_s$) and $\{f^k_p\}=\{f^0_p,f^1_p,f^2_p\}$.
We denote $\cM^A_2:=\cM^A(Q_\op,Q_\up;H^{ij}_s,\{f^k_p\},J^{ij}_s,g^{ij}_s)$ the union over $m_k\ge 0$, $k=0,1,2$ of the previously defined moduli spaces $\cM^A_{m_0,m_1,m_2}(Q_\op,Q_\up;H^{ij}_s,\{f^k_p\},J^{ij}_s,g^{ij}_s)$. This is a smooth manifold of dimension $|Q_\op|-|Q_\up e^A|+1$. In the case $|Q_\op|-|Q_\up e^A|=0$ that we are considering, $\cM^A_2$ is a smooth $1$-dimensional manifold whose boundary splits as
$$
\p\cM^A_2=\p^0\cM^A_2\cup \p^{\infty,0}\cM^A_2\cup \p^{\infty,1}\cM^A_2\cup \p^{\infty,2}\cM^A_2.
$$
Here $\p^0\cM^A_2$ corresponds to $m_0=0$ and the length of the gradient trajectory running between the endpoints of the two $s$-dependent Floer trajectories being equal to $0$. Thus $\p^0\cM^A_2=\p^\infty\cM^A_1$. The elements of $\p^{\infty,k}\cM^A_2$, $k=0,1,2$ correspond to the breaking of a gradient trajectory of $f^k_p$ appearing in the fibered product which defines $\cM^A_2$, for some $p\in\cP(H_k)$. Thus we have
\begin{eqnarray*}
\lefteqn{\p^{\infty,1}\cM^A_2}\\
&\hspace{-.5cm}=& \hspace{-.8cm}\bigcup_{\substack{p\in\cP(H_1) \\ Q_p\in\mathrm{Crit}(f^1_p) \\ B^{01}+B^{12}=A}}
\hspace{-.5cm}\cM^{B^{01}}(Q_\op,Q_p;H^{01}_s,\{f^i_p\},J^{01}_s,g^{01}_s) \times \cM^{B^{12}}(Q_p,Q_\up;H^{12}_s,\{f^j_p\},J^{12}_s,g^{12}_s).
\end{eqnarray*}
Here we abridged $\{f^i_p\}=\{f^0_p,f^1_p\}$ and $\{f^j_p\}=\{f^1_p,f^2_p\}$. Similarly, we have
\begin{eqnarray*}
\lefteqn{\p^{\infty,0}\cM^A_2}\\
&\hspace{-.2cm}=& \hspace{-.8cm}\bigcup_{\substack{p\in\cP(H_0) \\ Q_p\in\mathrm{Crit}(f^0_p) \\ B^0+B^{02}=A}}
\hspace{-.5cm}\cM^{B_0}(Q_\op,Q_p;H_0,\{f^0_p\},J_0,g_0)
 \times \cM^{B^{02}}(Q_p,Q_\up;H^{ij}_s,\{f^k_p\},J^{ij}_s,g^{ij}_s)
\end{eqnarray*}
with $|Q_p|-|Q_\up e^{B^{02}}|=-1$, and
\begin{eqnarray*}
\lefteqn{\p^{\infty,1}\cM^A_2}\\
&\hspace{-.2cm}=& \hspace{-.8cm}\bigcup_{\substack{p\in\cP(H_2) \\ Q_p\in\mathrm{Crit}(f^2_p) \\ B^{02}+B^2=A}}
\hspace{-.5cm}\cM^{B^{02}}(Q_\op,Q_p;H^{ij}_s,\{f^k_p\},J^{ij}_s,g^{ij}_s) \times \cM^{B^2}(Q_p,Q_\up;H_2,\{f^2_p\},J_2,g_2)
\end{eqnarray*}
with $|Q_\op|-|Q_p e^{B^{02}}|=-1$. We define $K_2:BC_*^{a,N}(H_0)\to BC_{*+1}^{a,N}(H_2)$ by
$$
K_2(Q_\op)=\sum_{\substack{\up\in\cP(H_2) \\ B^{02}\in H_2(W;\Z) \\ |Q_\op|-|Q_\up e^{B^{02}}|=-1}}
\
\sum_{\u\in \cM^{B^{02}}(Q_\op,Q_\up;H^{ij}_s,\{f^k_p\},J^{ij}_s,g^{ij}_s)}
\eps(\u) Q_\up e^{B^{02}}.
$$
The same argument as in the proof of Proposition~\ref{prop:chain-homotopy} shows that $K_2$ is of order $1$.
The previous description of $\p\cM^A_2$ shows that the chain map determined by the count of elements in $\p^0\cM^A_2$ is equal to $\sigma_{H_2,H_1}^{H_s^{12}}\circ \sigma_{H_1,H_0}^{H_s^{01}} - d\circ K_2 - K_2\circ d$. Since $\p^0\cM^A_2=\p^\infty\cM^A_1$, we obtain the conclusion of the Proposition by setting $K:=K_1+K_2$.
\end{proof}

\begin{proof}[Proof of Proposition~\ref{prop:indepJg}]
We consider a generic homotopy $(J^{12}_s,g^{12}_s)$, $s\in \R$ inside $\cJ^{S^1}_N$ from $(J_1,g_1)$ to $(J_2,g_2)$, which is constant near $\pm\infty$. Then $(J^{21}_s,g^{21}_s):=(J^{12}_{-s},g^{12}_{-s})$ is a homotopy from $(J_2,g_2)$ to $(J_1,g_1)$. These determine filtered chain maps $\sigma_{21}:BC_*^{a,N}(H,J_1,g_1)\to BC_*^{a,N}(H,J_2,g_2)$ and $\sigma_{12}:BC_*^{a,N}(H,J_2,g_2)\to BC_*^{a,N}(H,J_1,g_1)$.  By Proposition~\ref{prop:composition}, the composition $\sigma_{21}\circ \sigma_{12}$ is homotopic to
the filtered chain map $\sigma_{22}:BC_*^{a,N}(H,J_2,g_2)\to BC_*^{a,N}(H,J_2,g_2)$ determined by the concatenation $(J^{21}_s\#_R J^{12}_s, g^{21}_s\#_R g^{12}_s)$ for $R>0$ large enough. The latter is homotopic to the identity by Proposition~\ref{prop:chain-homotopy}.

Since all the homotopies involved are of order $1$, we obtain that $\sigma_{21}\circ \sigma_{12}$  induces on the first page $E^{a,N;1}_{*,*}(H,J_2,g_2)$ of the corresponding spectral sequence a chain morphism which is homotopic to the identity. The induced morphism on the second page $E^{a,N;2}_{*,*}(H,J_2,g_2)$ is therefore the identity. Similarly, $\sigma_{12}\circ \sigma_{21}$ induces the identity on the second page $E^{a,N;2}_{*,*}(H,J_1,g_1)$.

Thus the induced morphism $\sigma_{21}:E^{a,N;2}_{*,*}(H,J_1,g_1)\to E^{a,N;1}_{*,*}(H,J_2,g_2)$ is an isomorphism. Since $\sigma_{21}$ preserves both the degree and the filtration, it follows that
$\sigma_{21}(E^{a,N;2}_{*,1}(H,J_1,g_1))=E^{a,N;1}_{*,*}(H,J_2,g_2)$. Since $E^{a,N;2}_{*,1}(H,J_i,g_i)\simeq SH_{*+1}^{a,S^1,N}(H,J_i,g_i)$, $i=1,2$, we obtain the desired isomorphism. The fact that it does not depend on the choice of homotopy $(J^{12}_s,g^{12}_s)$ is a consequence of Proposition~\ref{prop:chain-homotopy}.
\end{proof}

\begin{remark} The isomorphism $SH_*^{a,S^1,N}(H,J_1,g_1)\stackrel\simeq\to SH_*^{a,S^1,N}(H,J_2,g_2)$ constructed in the proof of Proposition~\ref{prop:indepJg} is induced by the chain map
$SC_*^{a,S^1,N}(H,J_1,g_1)\to SC_*^{a,S^1,N}(H,J_2,g_2)$
given by the count of the elements of the $0$-dimensional moduli spaces
$$
\cM_{S^1}^A(S_\op,S_\up;H,J^{12}_s,g^{12}_s):=\cM^A(S_\op,S_\up;H,J^{12}_s,g^{12}_s)/S^1.
$$
\end{remark}

\appendix

\section{The parametrized Robbin-Salamon index} \label{app:RS}

We summarize in this section some important properties of the
parametrized Robbin-Salamon index. We recall from~\S\ref{sec:RS} the definition of the
subgroup $\cS_{n,m}\subset \Sp(2n+2m)$, consisting of matrices of the
form
$$
M=M(\Psi,X,E)=\left(\begin{array}{ccc}
\Psi & \Psi X & 0 \\
0 & \one & 0 \\
X^TJ_0 & E+\frac 1 2 X^T J_0 X & \one
\end{array}\right),
$$
with $\Psi\in\Sp(2n)$,
$X\in\mathrm{Mat}_{2n,m}(\R)$, and $E\in \mathrm{Mat}_m(\R)$
symmetric. We have denoted by $J_0:=\left(\begin{array}{cc} 0 &
    -\one \\ \one & 0 \end{array}\right)$ the standard complex
structure on $\R^{2n}$, so that $\Psi\in\Sp(2n)$ if and only if
$\Psi^T J_0\Psi=J_0$. The standard complex structure on $\R^{2n}\times
\R^{2m}$ is
$$
\widetilde J_0:=\left(
  \begin{array}{ccc}
J_0 & 0 & 0 \\
0 & 0 & -\one \\
0 & \one & 0
  \end{array}\right),
$$
and we have $M^T\widetilde J_0 M=\widetilde J_0$. Note that we have
natural embeddings (which respect the group structure)
$$
\cS_{n,m}\times \cS_{n',m'}\hookrightarrow \cS_{n+n',m+m'}
$$
which associate to $M=M(\Psi,X,E)\in\cS_{n,m}$ and
$M'=M(\Psi',X',E')\in\cS_{n',m'}$ the matrix
$$
M\oplus M':=M(\Psi\oplus \Psi',X\oplus X',E\oplus
E')\in\cS_{n+n',m+m'}.
$$
The space $\cS_{n,m}$ is stratified as $\coprod_{k=0}^{2n+m}
\cS^k_{n,m}$, with
$$
\cS^k_{n,m}:=\big\{M\in\cS_{n,m} \ : \ \dim\, \ker \, (M-\one) = m+k. \big\}.
$$

\begin{proposition} \label{prop:mu}
The Robbin-Salamon index $\mu=\mu_{RS}$ defined on paths $M:[a,b]\to
\cS_{n,m}$, $M(\theta)=M(\Psi(\theta),X(\theta),E(\theta))$ has the
following properties.
\begin{description}
\item[(Homotopy)] If $M,M':[a,b]\to \cS_{n,m}$ are homotopic with
  fixed endpoints then
$$
\mu(M)=\mu(M');
$$
\item[(Catenation)] For any $c\in [a,b]$ we have
$$
\mu(M)=\mu(M|_{[a,c]}) + \mu(M|_{[c,b]});
$$
\item[(Naturality)] For any path $P:[a,b]\to \Sp(2n)\times \Sp(2m)$ of
the form
\begin{equation} \label{eq:Ptheta}
P(\theta)=\left(\begin{array}{ccc}
\Phi(\theta) & 0 & 0 \\
0 & A(\theta) & 0 \\
0 & 0 & A(\theta)
  \end{array}\right)
\end{equation}
(with $\Phi(\theta)\in\Sp(2n)$ and $A(\theta)\in\OO(m)$), we have
$$
\mu(PMP^{-1})=\mu(M);
$$
\item[(Loop)] For any loop $P:[a,b]\to \Sp(2n)\times \Sp(2m)$ of the
form~\eqref{eq:Ptheta}, we have
$$
\mu(PM)=\mu(M)+ 2\mu(\Phi);
$$
\item[(Product)] For any $M\in \cS_{n,m}$ and $M'\in \cS_{n',m'}$ we
  have
$$
\mu(M\oplus M')=\mu(M) + \mu(M');
$$
\item[(Splitting)] Given $M=M(\Psi,0,E):[a,b]\to \cS_{n,m}$, we have
$$
\mu(M)=\mu(\Psi) + \frac 1 2 \mathrm{sign}\, E(b) - \frac 1 2
\mathrm{sign}\, E(a);
$$
\item[(Signature)] Given symmetric matrices $E\in \R^{m\times m}$
  and $S\in\R^{2n\times 2n}$ with $\| S \| <2\pi$, we have
$$
\mu\bigg\{\left(\begin{array}{cc}\exp(J_0St) & 0 \\ 0 & tE
  \end{array}\right),\ t\in[0,1]\bigg\} = \frac 1 2 \mathrm{sign}(S) +
\frac 1 2 \mathrm{sign}(E);
$$
\item[(Zero)] For any path $M:[a,b]\to \cS^k_{n,m}$ we have
$$
\mu(M)=0;
$$
\item[(Integrality)] Given a path $M:[a,b]\to \cS_{n,m}$ with $M(a)\in
  \cS^{k_a}_{n,m}$, $M(b)\in \cS^{k_b}_{n,m}$, we have
$$
\mu(M)+\frac {k_a-k_b} 2 \in \Z;
$$
\item[(Determinant)] Given a path $M=M(\Psi,X,E):[a,b]\to \cS_{n,m}$ with
  $M(a)=\one$ and $M(b)\in\cS^0_{n,m}$, we have
$$
(-1)^{n+\frac m 2 - \mu(M)}=\mathrm{sign}\, \det \,
\left(\begin{array}{cc}
\Psi-\one & \Psi X \\
X^TJ_0 & E+\frac 1 2 X^TJ_0 X
\end{array}\right).
$$
We have denoted for simplicity $\Psi=\Psi(b)$, $X=X(b)$, $E=E(b)$.
\item[(Involution)] For any $M=M(\Psi,X,E):[a,b]\to\cS_{n,m}$ we have
$$
\mu(M(\Psi,X,E))=\mu(M(\Psi,-X,E))
$$
and
$$
\mu(M(\Psi^{-1},X,-E))=\mu(M(\Psi^T,J_0X,-E))=-\mu(M(\Psi,X,E)).
$$
\end{description}
\end{proposition}

\begin{proof}

The \emph{(Homotopy)},  \emph{(Catenation)}, \emph{(Naturality)},
\emph{(Product)}, and \emph{(Zero)} properties are exactly the
corresponding properties of the Robbin-Salamon
index~\cite[Theorem~4.1]{RS}. For the \emph{(Naturality)} property, a
straightforward verification shows that $PMP^{-1}\in\cS_{n,m}$.

To prove the \emph{(Loop)} property we use the equality
$$
\mu(PM)=\mu(M)+2\mu(P)=\mu(M)+2\mu(\Phi)+
2\mu\left(\begin{array}{cc} A& 0 \\ 0 & A \end{array}\right).
$$
Since $\pi_1(\OO(m))=\Z/2\Z$ and $\pi_1(\Sp(2m))=\Z$, the last term
vanishes.

The \emph{(Splitting)} property follows from the \emph{(Product)}
property and the normalization axiom for the Robbin-Salamon index of a
symplectic shear.

The \emph{(Signature)} property follows from the
\emph{(Splitting)} property and from the identity
$\mu_{RS}(\exp(J_0St))=\frac 1 2
\mathrm{sign}(S)$~\cite[Theorem~3.3.(iv)]{SZ}.

The \emph{(Integrality)} property follows directly from the analogous
property for the Robbin-Salamon index~\cite[Theorem~4.7]{RS}.

We prove the \emph{(Involution)} property. The identity
$\mu(M(\Psi^{-1},X,-E))=\mu(M(\Psi^T,J_0X,-E))$ follows from the
\emph{(Naturality)} axiom by conjugating with the constant path
$J_0\oplus \one_{2m}$. The identity
$\tmu(M(\Psi,X,E))=\tmu(M(\Psi,-X,E))$ follows by conjugating twice
with $J_0\oplus \one_{2m}$.

It remains to prove the \emph{(Determinant)} property. Given a path
$N:[0,1]\to \Sp(2n+2m)$ satisfying $N(0)=\one$ and $\det\,
(N(1)-\one)\neq 0$, we have~\cite[Theorem~3.3.(iii)]{SZ}
$$
(-1)^{n+m-\mu_{RS}(N)}=\mathrm{sign}\, \det \, (N(1)-\one).
$$
We construct such a path $N:[a,b+\eps]\to \Sp(2n+2m)$ by catenating
$M=M(\Psi(\theta),X(\theta),E(\theta))$ with the path
$M':[b,b+\eps]\to \Sp(2n+2m)$ given by
$$
M'(b+\theta):=\left(\begin{array}{ccc}
\Psi & \Psi X & \theta \Psi X \\
0 & \one & \theta\one \\
X^TJ_0 & E+\frac 1 2 X^T J_0 X & \one+\theta (E + \frac
1 2 X^TJ_0 X)
\end{array}\right).
$$
We have denoted for simplicity $\Psi:=\Psi(b)$, $X:=X(b)$, and $E:=
E(b)$. Since $M(b)\in\cS^0_{n,m}$, the path $M'$ has a single crossing
at $b$ and the kernel of
$M'(b)-\one=M(b)-\one$ is $\{0\}\oplus \{0\}\oplus \R^m$. The crossing
form at $b$ is $-\one_m$, so that $\mu_{RS}(M')=-\frac m 2$. Thus
$\mu_{RS}(N)=\mu(M)-\frac m 2$. On the other hand
$$
\det\,(N(b+\eps)-\one)=\eps^m(-1)^m\det\,\left(\begin{array}{cc}
\Psi-\one & \Psi X \\
X^TJ_0 & E+\frac 1 2 X^TJ_0 X
\end{array}\right).
$$
This implies the desired statement.
\end{proof}

\begin{example}
The index $\mu(M(\Psi,X,E))$ depends in an essential way on $X$, as
the following example shows. Given $a,b\in\R$, let
$$
\Psi:=\left(\begin{array}{cc} 2 & 0 \\ 0 & \frac 1 2
\end{array}\right), \qquad
X_{a,b}:=\left(\begin{array}{c} a \\ b \end{array}\right), \qquad E:=1.
$$
We denote $M_{a,b}:=M(\Psi,X_{a,b},E)\in\cS_{1,1}$.
It follows from the
\emph{(Determinant)} property that a path in $\cS_{1,1}$ starting at
$\one$ and ending at $M_{0,0}$ has an index in $\frac 1 2
+ 2\Z$, whereas a path in $\cS_{1,1}$ starting at
$\one$ and ending at $M_{1,1}$ has an index in  $\frac 1 2
+ 2\Z+1$ (the value of the relevant determinant is $-\frac 1 2 +\frac
3 2 ab$).
\end{example}

For the rest of this Appendix we place ourselves in $\R^{2N}$ equipped with the standard symplectic form $\omega_0$ and the standard complex structure $J_0$. The next Proposition is relevant for the parametrized Robbin-Salamon index when applied with $N=n+m$ and $E(t)\equiv \{0\}\oplus \{0\}\oplus \R^m$. We recall that, given a path of symplectic matrices $M:[0,1]\to \Sp(2N)$, the crossing form at a point $t\in[0,1]$ is the quadratic form $\Gamma(M,t)$ on $\ker\, (M(t)-\one)$ given by $\Gamma(M,t)(v)=\langle v,-J_0\dot M(t) M(t)^{-1} v\rangle$.

\begin{proposition} \label{prop:RSstratum}
Let $M:[0,1]\to \Sp(2N)$ be a $C^1$-path of symplectic
  matrices with the following property: there exists a continuous family of vector spaces
  $t\mapsto E(t)\subset \R^{2N}$ such
  that $E(t)\subset \ker\,(M(t)-\one)$ and the crossing form
  $\Gamma(M,t)$ induces a nondegenerate quadratic form on $\ker\,(M(t)-\one)/E(t)$. Assume $\omega_0$ has constant rank on $E(t)$. Then
$$
\mu_{RS}(M)=\frac 1 2 \mathrm{sign}\, \Gamma(M,0) + \sum_{t:\dim F(t)\cap
  \ker(M(t)-\one)>0} \mathrm{sign}\, \Gamma(M,t) + \frac 1 2
\mathrm{sign}\, \Gamma(M,1).
$$
\end{proposition}

\begin{proof} Let us first assume that the rank of $\omega_0$ is constant equal to $0$ on $E(t)$, i.e. $E(t)$ is isotropic. Let us decompose $\R^{2N}=E(t)\oplus J_0 E(t)\oplus F(t)$, where $F(t)$ is the symplectic orthogonal of $E(t)\oplus J_0 E(t)$. Given $\eps>0$ we denote by $\beta_\eps:[0,1]\to[0,\eps]$ a smoothing of the function
$$
t\mapsto\left\{\begin{array}{ll} t, & 0\le t\le \eps,\\
\eps, & \eps\le t\le 1-\eps,\\
1-t, & 1-\eps\le t\le 1.
\end{array}\right.
$$
We define an element $\Phi^0_\eps(t)\in\Sp(2N)$ which has the following matrix form with respect to the splitting $E(t)\oplus J_0 E(t)\oplus F(t)$:
$$
\Phi^0_\eps(t)=\left(\begin{array}{ccc} \one & 0 & 0 \\ \beta_\eps(t) & \one & 0 \\ 0 & 0 & \one \end{array}\right).
$$
We define $\tM_\eps(t):=M(t)\Phi^0_\eps(t)$, and we have $\mu_{RS}(\tM)=\mu_{RS}(M)$ since these paths are homotopic with fixed endpoints. We claim that the following equality holds for all $t\in]0,1[$:
\begin{equation} \label{eq:isotropic}
\ker\, (\tM(t)-\one) = \ker\, (M(t)-\one) \cap (J_0 E(t)\oplus F(t)).
\end{equation}
That $\ker\, (M(t)-\one) \cap (J_0 E(t)\oplus F(t))\subset \ker\, (\tM(t)-\one)$ follows from the fact that
$\Phi^0_\eps(t)$ acts by the identity on $J_0 E(t)\oplus F(t)$. Conversely, let $v=v_1+v_2\in\ker\, (\tM(t)-\one)$, with $v_1\in E(t)$ and $v_2\in J_0 E(t)\oplus F(t)$. The identity $\tM(t)v=v$ is equivalent to $M(t)(v_1+\beta_\eps(t)J_0v_1+v_2)=v_1+v_2$, hence to $(M(t)-\one)v_2=-\beta_\eps(t)M(t)J_0v_1$. Using that $M(t)v_1=v_1$ we obtain
$$
0 = \omega_0(v_1,(M(t)-\one)v_2) = -\beta_\eps(t)\omega_0(v_1,M(t)J_0v_1)=-\beta_\eps(t)\omega_0(v_1,J_0v_1).
$$
Since $\beta_\eps(t)\neq 0$, this implies $v_1=0$, so that $v=v_2\in J_0 E(t)\oplus F(t)$
and $(M(t)-\one)v_2=(\tM(t)-\one)v_2=0$, as desired.

Since the restrictions of $M(t)$ and $\tM(t)$ to $J_0E(t)\oplus F(t)$ are the same, it follows that the crossing form $\Gamma(\tM,t)$ coincides with $\Gamma(M,t)$ on $\ker\, (\tM(t)-\one)$ for $t\in]0,1[$. On the other hand, a straightforward computation shows that
\begin{eqnarray} \label{eq:E01}
\mathrm{sign}\,\Gamma(\tM,0)&=&\mathrm{sign}\,\Gamma(M,0) + \dim\, E(0),\\
\mathrm{sign}\,\Gamma(\tM,1)&=&\mathrm{sign}\,\Gamma(M,1) - \dim\, E(1). \nonumber
\end{eqnarray}
Thus, the contributions at the endpoints compensate each other, and the conclusion follows using the definition of the Robbin-Salamon index via crossing forms.

We now assume that the rank of $\omega_0$ on $E(t)$ is equal to $\dim\, E(t)$, i.e. $E(t)$ symplectic. Let us decompose $\R^{2N}=E(t)\oplus F(t)$, where $F(t)$ is the symplectic orthogonal of $E(t)$. Let $J(t)$ be a continuous family of complex structures on $E(t)$ which are compatible with $\omega_0$. For $\eps>0$ we define a path $\Phi^1_\eps:[0,1]\to \Sp(2N)$ whose matrix with respect to the decomposition $E(t)\oplus F(t)$ is
$$
\Phi^1_\eps(t):=\left(\begin{array}{cc} \exp(J(t)\beta_\eps(t)) & 0 \\ 0 & \one \end{array}\right).
$$
We denote $\tM(t):=M(t)\Phi^1_\eps(t)$, so that we have $\mu_{RS}(\tM)=\mu_{RS}(M)$. We claim that
\begin{equation} \label{eq:symplectic}
\ker\,(\tM(t)-\one)=\ker \,(M(t)-\one) \cap F(t)
\end{equation}
for all $t\in]0,1[$, whenever $0<\eps<\pi$. That $\ker \,(M(t)-\one) \cap F(t)\subset \ker\,(\tM(t)-\one)$ follows from the fact that $\Phi^1_\eps(t)$ acts as the identity on $F(t)$. Conversely, let $v=v_1+v_2\in\ker \,(\tM(t)-\one)$ such that $v_1\in E(t)$ and $v_2\in F(t)$. The relation $\tM(t)v=v$ is equivalent to $(M(t)-\one)v_2=(\one- \exp(J(t)\beta_\eps(t)))v_1$. Then
\begin{eqnarray*}
0&=&\omega_0(v_1,(M(t)-\one)v_2)\\
&=&\omega_0(v_1,(\one- \exp(J(t)\beta_\eps(t)))v_1)\\
&=&-\sin(\beta_\eps(t))\omega_0(v_1,J(t)v_1).
\end{eqnarray*}
Since $\sin(\beta_\eps(t))\neq 0$, we obtain $v_1=0$ and the claim follows.

Since the restrictions of $M(t)$ and $\tM(t)$ to $F(t)$ are the same, it follows that the crossing form $\Gamma(\tM,t)$ coincides with $\Gamma(M,t)$ on $\ker\, (\tM(t)-\one)$ for $t\in]0,1[$. On the other hand, a straightforward computation shows that equations~\eqref{eq:E01} still hold, and the conclusion follows.

Finally, we assume that the rank of $\omega_0$ on $E(t)$ is lies strictly between $0$ and $\dim\, E(t)$. We choose a continuous splitting $E(t)=E_1(t)\oplus E_0(t)$ with $E_0(t):=E(t)\cap E(t)^{\omega_0}$ isotropic and $E_1(t)=E_0(t)^\perp$ symplectic. Here $E(t)^{\omega_0}$ denotes the symplectic orthogonal of $E(t)$, and $E_0(t)^\perp$ denotes the Euclidean orthogonal of $E_0(t)$ in $E(t)$. We decompose $\R^{2N}=E_1(t)\oplus E_0(t)\oplus J_0E_0(t)\oplus F(t)$, such that $F(t)$ is the symplectic orthogonal of $E_1(t)\oplus E_0(t)\oplus J_0E_0(t)$. Given $0<\eps<\pi$ we define as above two paths $\Phi^0_\eps(t)$  acting as the identity on $E_1(t)\oplus F(t)$, and $\Phi^1_\eps(t)$ acting as the identity on $E_0(t)\oplus J_0E_0(t)\oplus F(t)$. We denote $\tM(t):=M(t)\Phi^0_\eps(t)\Phi^1_\eps(t)$, so that $\mu_{RS}(\tM)=\mu_{RS}(M)$. One proves as above that the crossings of $\tM$ and $M$ on $]0,1[$ are the same, with the same crossing forms on $\ker\, (\tM(t)-\one)$, and moreover equations~\eqref{eq:E01} still hold. This finishes the proof.
\end{proof}

\begin{remark}
The crossing form $\Gamma(M,t)$ vanishes identically on
$E(t)$. Indeed, given a path $v(t)\in
E(t)$ we have $M(t)v(t)=v(t)$ and $\dot M(t) v(t)+M(t)\dot v(t)=\dot
v(t)$. Dropping the $t$-variable for clarity, we have
\begin{eqnarray*}
  \Gamma(M,t)(v(t)) & = & \langle  v,-J_0\dot M M^{-1}v\rangle \\
&=& \langle  v,-J_0(\dot v-M\dot v)\rangle \\
&=& -\langle v,J_0\dot v\rangle + \langle v,(M^{-1})^TJ_0\dot v\rangle
\\
&=& -\langle v,J_0\dot v\rangle + \langle M^{-1}v,J_0\dot v\rangle \ =
\ 0.
\end{eqnarray*}
\end{remark}


\end{document}